\documentclass[11pt,letterpaper,reqno]{amsart}
\usepackage{hyperref}
\usepackage{amssymb}
\usepackage{amsmath,amscd,color,amsfonts}
\DeclareSymbolFontAlphabet{\mathbb}{AMSb}
\DeclareMathAlphabet{\mathbscr} {U}{BOONDOX-cal}{r}{n}
\usepackage{combelow}
\usepackage[all]{xy}
\usepackage{enumerate}
\usepackage{mathrsfs}
\usepackage{tensor}
\usepackage{stackrel}
\usepackage{marginnote}
\usepackage[normalem]{ulem} 

\providecommand{\U}[1]{\protect\rule{.1in}{.1in}}
\newtheorem{theorem}{Theorem}[section]
\newtheorem*{theorem*}{Theorem}
\newtheorem*{corollary*}{Corollary}
\newtheorem*{claim*}{Claim}

\newtheorem{corollary}[theorem]{Corollary}

\newtheorem{definition}[theorem]{Definition}

\newtheorem{example}[theorem]{Example}
\newtheorem{lemma}[theorem]{Lemma}
\newtheorem{proposition}[theorem]{Proposition}
\newtheorem*{proposition*}{Proposition}
\newtheorem{remark}[theorem]{Remark}
\numberwithin{equation}{section}

\newcommand{\id}{\mathrm{Id}}       
\newcommand{\pr}{\operatorname{pr}} 
\newcommand{\im}{\operatorname{Im}} 
\renewcommand{\graph}{\mathrm{Graph}} 
\renewcommand{\d}{\mathrm{d}} 
\newcommand{\can}{\mathrm{can}}     
\newcommand{\red}{\mathrm{red}}     
\newcommand{\Lie}{\mathscr{L}}      
\newcommand{\hol}{\operatorname{hol}} 
\newcommand{\hor}{\operatorname{hor}} 
\newcommand{\dev}{\operatorname{dev}} 
\newcommand{\Aff}{\operatorname{Aff}} 
\newcommand{\R}{\mathbb{R}}         
\newcommand{\C}{\mathbb{C}}         
\newcommand{\Z}{\mathbb{Z}}         
\newcommand{\Tt}{\mathbb{T}}         
\newcommand{\Ss}{\mathbb{S}}         
\newcommand{\X}{\mathfrak{X}}       
\newcommand{\depth}{\textrm{depth}} 
\newcommand{\st}{\mathrm{st}}

\newcommand{\diffto}{\xrightarrow{\raisebox{-0.2 em}[0pt][0pt]{\smash{\ensuremath{\sim}}}}}

\newcommand{\ract}{\curvearrowright}    
\newcommand{\GL}{\operatorname{GL}}     
\newcommand{\Span}{\operatorname{span}} 
 
\newcommand{\grad}{\operatorname{grad}}

\newcommand{\act}{\mathbscr{a}}     
\newcommand{\Act}{\mathbscr{A}}     

\newcommand{\G}{\mathcal{G}}            
\renewcommand{\H}{\mathcal{H}}          
\newcommand{\T}{\mathcal{T}}            
\renewcommand{\O}{\mathcal{O}}          
\newcommand{\tto}{\rightrightarrows}    
\newcommand{\al}{\alpha}                
\newcommand{\be}{\beta}                 
\newcommand{\dto}{\dashrightarrow}      

\newcommand{\sslash}{\mathbin{/\mkern-6mu/}}
\newcommand{\TLag}{\underline{\T_{\textrm{Lag}}}}
\newcommand{\TFlat}{\underline{\T_{\textrm{Flat}}}}
\newcommand{\TLagn}{\underline{\Tt^n_{\textrm{Lag}}}}

\newcommand{\IAMan}{\mathbf{IntAffMan}} 
\newcommand{\SympTB}{\mathbf{SympTorBun}} 

\newcommand{\twprod}{\mathbin{%
    \ooalign{\raise1.15ex\hbox{$\scriptstyle\sim$}\cr\hidewidth$\times$\hidewidth\cr}%
    }}

\newcommand{\rvline}{\hspace*{-\arraycolsep}\vline\hspace*{-\arraycolsep}}

\begin{document}
\title{K\"ahler metrics and toric Lagrangian fibrations}

\author{Rui Loja Fernandes}
\address{Department of Mathematics, University of Illinois at Urbana-Champaign, 1409 W. Green Street, Urbana, IL 61801 USA}
\email{ruiloja@illinois.edu}

\author{Maarten Mol}
\address{Department of Mathematics, University of Toronto, 40 St. George Street, Toronto, ON M5S 2E4 Canada}
\email{maarten.mol.math@gmail.com}

\thanks{RLF was partially supported by NSF grants DMS-2003223 and DMS-2303586. MM was supported by the Max Planck Institute for Mathematics.}

\begin{abstract}
We extend the Abreu-Guillemin theory of invariant K\"ahler metrics from compact toric symplectic manifolds to any, compact or not, symplectic manifold admitting a toric action of a symplectic torus bundle. This is a much wider class of manifolds which can be characterized as those symplectic manifolds admitting a Lagrangian fibration with at most elliptic singularities. The base of such a toric Lagrangian fibration is a codimension 0 submanifold with corners of an integral affine manifold, called a Delzant domain. This concept generalizes the Delzant polytope associated with a compact symplectic toric manifold. Given a Delzant domain of finite type, we provide a Delzant-type construction of a Lagrangian fibration with moment image being the specified Delzant domain. We establish a 1:1 correspondence between invariant K\"ahler metrics and pairs consisting of an elliptic connection on the total space of the fibration and a hybrid $b$-metric on the base Delzant domain, both with specified residues over the facets. Finally, we characterize extremal invariant K\"ahler metrics as those whose scalar curvature descends to an affine function on the base integral affine manifold. We show that this provides a method for finding and constructing extremal K\"ahler metrics.
\end{abstract}
\maketitle

\setcounter{tocdepth}{1}
\tableofcontents

\section{Introduction}

In the late 1990s, Guillemin \cite{Guill94} and Abreu \cite{Abr98} described all invariant, compatible K\"ahler metrics for a compact symplectic toric manifold, in terms of singular Hessian metrics on the associated Delzant polytope. Abreu's work also encompasses a fourth-order nonlinear PDE expressing the condition for an invariant K\"ahler metric to be extremal in the sense of Calabi. Subsequently, Donaldson \cite{Donald02} developed the analysis of Abreu's equation and formulated K-stability for polytopes, sparking a series of subsequent research works in the search for extremal K\"ahler metrics (see, e.g., the recent survey \cite{LS23}). The main aim of this work is to extend the Abreu-Guillemin theory to a much wider class of symplectic manifolds.

The class of symplectic manifolds we consider here, which includes both compact and non-compact manifolds, have a type of symmetry which is described by an action of a symplectic torus bundle $(\T,\Omega)$. This type of action includes as a special case the ordinary Hamiltonian torus actions but also, e.g., torus actions with cylinder valued moment maps or even examples with no globally defined Hamiltonian torus action. More precisely, we consider symplectic manifolds admitting \emph{toric} actions of symplectic torus bundles $(\T,\Omega)$ (see Definition \ref{def:toric:Hamiltonian:space}). These can be characterized, alternatively, as symplectic manifolds admitting a \emph{toric Lagrangian fibration}, i.e., a singular Lagrangian fibration with only elliptic-type singularities, or as \emph{locally toric symplectic manifolds}  (see Theorems \ref{thm:Dufour-Molino} and \ref{thm:equiv:toric:Lagrangian}):

\begin{theorem*}
For a map $\mu:(S,\omega)\to M$, the following are equivalent:
\vskip 5 pt
\begin{enumerate}[(i)]
\item $\mu:(S,\omega)\to M$ is a toric Lagrangian fibration;
\item $\mu:(S,\omega)\to M$ is the moment map of a toric Hamiltonian $\T$-space;
\item For each $x\in \mu(S)$ there is a local chart $(U,\phi)$ centered at $x$ such that $\phi\circ\mu:(\mu^{-1}(U),\omega)\to \R^n$ is the moment map of a toric $\Tt^n$-action.
\end{enumerate}
\end{theorem*}

We picture a toric Hamiltonian $\T$-space as in the following diagram
\[
\vcenter{\xymatrix@C=10pt@R=30pt{
(\T,\Omega) \ar[dr]_p & {}\save[]+<-22pt,0cm>*\txt{\Large $\circlearrowright$}\restore (S, \omega) \ar[d]^-{\mu} \\
  & M   }}
\]
where $M$ is an integral affine manifold, $p:(\T,\Omega)\to M$ is the corresponding symplectic torus bundle, $(S,\omega)$ is a symplectic manifold and $\T$ acts on $S$ along the moment map $\mu$. Delzant's classification of toric symplectic manifolds \cite{De88} has been extended to a classification of toric $\T$-spaces in terms of their moment map image $\Delta=\mu(S)$ \cite{Maarten} (see also Section \ref{sec:toricspaces}). The space $\Delta$ is now a codimension 0 submanifold with corners of the integral affine manifold $M$, called a \textit{Delzant domain}. Delzant domains are much more general than Delzant polytopes. For instance, they can have non-trivial fundamental group. 
In an informal sense, one can say that we consider ``toric'' symplectic manifolds whose associated Delzant polytopes -- which now live inside an integral affine manifold $M$, instead of $\R^n$ -- can have non-trivial topology. This allows for a much broader class of symplectic manifolds.

In the particular case of principal Hamiltonian $\T$-spaces, which correspond to Lagrangian fibrations without singularities, the classification reduces to Duistermaat's result in \cite{Dui80} (see also Section \ref{sec:principal}). The classical symplectic toric manifolds, both compact and non-compact (see \cite{KL15}), fit into this framework as follows (see Theorem \ref{thm:classical:toric:spaces}):

\begin{theorem*}
    A toric Hamiltonian $\T$-space is a classical symplectic toric manifold if and only if the associated Delzant space has trivial affine holonomy.
\end{theorem*}

Given an integral affine manifold $M$ and a Delzant domain $\Delta\subset M$, the classification mentioned above requires the construction of a toric $\T$-space $\mu:(S,\omega)\to M$
with $\mu(S)=\Delta$. This is done for general Delzant domains in \cite{Maarten}. Here, for a large class of Delzant domains, we give an alternative construction in the spirit of Delzant's original construction. This is the content of the following theorem (see Section \ref{sec:delzant} for details). 

\begin{theorem*} 
Let $\Delta\subset M$ be a finite type Delzant domain with $d$ facets. Every toric $\T$-space with moment image $\Delta$ can be realized as a symplectic quotient
\begin{equation}
    \label{eq:sympletic:quotient}
    \big((P\times \C^d)\sslash(\Gamma\ltimes \Tt^d),\,\omega_\red\big),
\end{equation}
with $P$ a principal Hamiltonian $p^*\T$-space, $p:\tilde{M}\to M$ the universal covering space {and $\Gamma$ the image of the group homomorphism $\pi_1(\Delta)\to \pi_1(M)$ induced by the inclusion $\Delta\hookrightarrow M$}. If the Lagrangian Chern class vanishes one can let $P=p^*\T$.
\end{theorem*}

As we explain in Section \ref{sec:torus:bundles}, one can still make sense of $\T$-invariant objects, such as $\T$-invariant metrics and complex structures. Using this we extend the Abreu-Guillemin theory to toric $\T$-spaces. Our main result is the following (see Section \ref{sec:inv:metrics:toric}). 

\begin{theorem*}
Let $\mu:(S,\omega)\to M$ be a toric $\T$-space, with Delzant domain $\Delta:=\mu(S)\subset M$. There is a 1-1 correspondence between invariant K\"ahler metrics on $S$ compatible with $\omega$ and the following data.
\begin{enumerate}[(i)]
    \item A flat Lagrangian elliptic connection for $\mu:S\to M$ with zero radial residue over the open facets of $\Delta$.
    \item A Hessian hybrid $b$-metric on $\Delta$ with residues at the open facets given by the primitive outward-pointing normals multiplied by $-\tfrac{1}{4\pi}$.
\end{enumerate} 
Moreover, an invariant K\"ahler metric $G$ is extremal if and only if its scalar curvature $S_G$, viewed as function on $\Delta$, is affine.
\end{theorem*}

Elliptic connections and hybrid $b$-metrics are tensors with specified singularities over the boundary of $\Delta$, introduced in Sections \ref{sec:elliptic:connections} and \ref{sec:bmetrics}, respectively. These are defined in a coordinate-free way using the elliptic tangent bundle of $\mu$ \cite{CaGu,CKW22,CKW23,CW22} and the $b$-tangent bundle of $\Delta$ \cite{Mel93}. In the classical theory, the presence of an elliptic connection is hidden in a choice of action-angle coordinates over the interior of the polytope, and any two such choices lead to isomorphic invariant K\"ahler metrics. This is no longer the case in our more general setting. {We also emphasize that the theorem does not assume $S$ to be compact -- though the notion of toric $\T$-space does require $\mu$ to be proper as map onto its image -- and so includes, as a special case, Theorem 2.14 in \cite{AS12} concerning non-compact symplectic toric manifolds.}

When $\Delta\subset M$ is a finite type Delzant domain and $M$ carries a Hessian metric, our Delzant type construction above allows one to apply K\"ahler reduction, leading to the following result (see Section \ref{sec:inv:metrics:toric}).

\begin{corollary*}
    Let $(M,g)$ be a connected integral affine Hessian manifold and $\Delta\subset M$ a Delzant domain of finite type with primitive boundary defining functions $\ell_1,\dots,\ell_d$. When $P=p^*\T$, there is a reduced K\"ahler metric on the symplectic quotient \eqref{eq:sympletic:quotient} which induces the Hessian hybrid $b$-metric 
    \[
    g_\Delta=g+\mathrm{Hess}_\Lambda(\phi)
    \] 
    on $\Delta$, where $\phi$ is the smooth function on $\mathring{\Delta}$ given by 
    \begin{equation}
    \label{eq:canonical:potential:introd}
        \phi=-\frac{1}{4\pi}\sum_{i=1}^d\ell_i\log|\ell_i|.
    \end{equation}
\end{corollary*}

The relationship between our Delzant type construction and the standard Delzant construction will be explained in the main body of the paper.

Our main theorem above provides an effective way of constructing invariant K\"ahler metrics. In this paper, we only discuss a few illustrative examples, although we do show how to recover, by our techniques, the extremal K\"ahler metrics on complex ruled surfaces over an elliptic curve that appear in the work of Apostolov et al. \cite{ACGT08} -- see Example \ref{ex:nontrivspherebun:nonstdIA:cylinder:3}. We hope in future work to discuss more elaborate examples.


{This paper is organized as follows. Section \ref{sec:torus:bundles}  collects some basic definitions, facts and some background on Hamiltonian actions of symplectic torus bundles and integral affine structures. We also start discussing some examples that will be used throughout the paper to illustrate the theory. Section \ref{sec:principal} discusses non-singular Lagrangian fibrations, which we view as principal Hamiltonian spaces for symplectic torus bundles. We recall Duistermaat's classification of non-singular Lagrangian fibrations from this point of view and we introduce the notion of Lagrangian connection for such fibrations. We explain the relation of such connections to the Lagrangian Chern class of the fibration, and give a cohomological classification of flat Lagrangian fibrations. Section \ref{sec:toricspaces} extends the theory in the previous
section to singular Lagrangian fibrations with elliptic singularities, i.e., toric Lagrangian fibrations. For that, we give a short overview of the notions of elliptic and b-tangent bundles, and we introduce the notion of elliptic Lagrangian connection. This section also includes a discussion of how classical symplectic toric manifolds fit into this more general framework. Section \ref{sec:delzant} discusses Delzant domains of finite type and contains a proof of our Delzant type construction for such domains. The last two sections of the paper are dedicated to the study of invariant K\"ahler metrics: Section \ref{sec:inv:metrics:non-singular} discusses the theory for non-singular Lagrangian fibrations and Section \ref{sec:inv:metrics:toric} extends this to toric Lagrangian fibrations. }
\medskip

{\bf Acknowledgments.} We would like to express our gratitude to Marius Crainic, Eugene Lerman, Ioan Marcut, David Mart\'{i}nez Torres, Daniele Sepe and Aldo Witte for the many inspiring discussions that 
contributed to shaping the ideas presented in this paper. Special thanks are extended to Miguel Abreu, who provided invaluable advice on 
this project and has become a collaborator. We also acknowledge the generous support of the Max Planck Institute for Mathematics in Bonn and the University of Illinois Urbana-Champaign at various stages of this project. Finally, we are indebted to the anonymous referee, whose many helpful suggestions and comments greatly improved the paper.

\medskip
{\bf Conventions and notations.} In this paper, by a torus we mean a compact, connected, real, abelian Lie group. For manifolds with corners, always denoted by $\Delta$, we follow the conventions and terminology from \cite{Maarten}. In particular, $\R^n_k$ denotes the manifold with corners $[0,\infty[^k\times \R^{n-k}$. For further background on manifolds with corners see, e.g., \cite{Michor80}. 

\section{Hamiltonian $\T$-spaces}
\label{sec:torus:bundles}

In this sections we introduce Hamiltonian actions of symplectic torus bundles and  discuss some basic properties. 

\subsection{Symplectic torus bundles}
\label{subsec:torus:bundles:background}

In this paper, by a {\bf torus bundle} we mean a fiber bundle $p:\T\to M$ where:
\begin{itemize}
    \item each fiber is a Lie group isomorphic to a torus;
    \item the fiberwise multiplication $m:\T\times_M \T\to\T$, inversion $i:\T\to\T$ and unit $u:M\to\T$ are smooth maps.
\end{itemize}
On the other hand, an {\bf equivalence} of torus bundles is a diffeomorphism
\[
\xymatrix{
\T_1\ar[r]^\Psi\ar[d]_{p_1} & \T_2\ar[d]^{p_2}\\
M\ar@{=}[r]_{\id} & M}
\]
which restricts to a Lie group isomorphism on each fiber. 

We will be specially interested in torus bundles carrying a symplectic form: 

\begin{definition}
A {\bf symplectic torus bundle} is a torus bundle $p:\T\to M$ with a symplectic form $\Omega\in\Omega^2(\T)$ which is multiplicative, i.e., such that
\[ m^*\Omega=\pr_1^*\Omega+\pr_2^*\Omega,\] where $m,\pr_1,\pr_2:\T\times_M \T\to\T$ 
are the fiberwise multiplication and the projections onto each factor, respectively. An {\bf equivalence} of symplectic torus bundles is an equivalence of torus bundles preserving the symplectic forms.
\end{definition}

The multiplicative condition forces the fibers of $p:\T\to M$ and the unit section $u:M\to\T$ to be Lagrangian submanifolds, so $\dim\T=2\dim M$.

\begin{example}
    The simplest example of a symplectic torus bundle is 
    \[ \pr:\Tt^n\times\R^n\to\R^n,\qquad \Omega=\sum_{i=1}^n \d\theta_i\wedge \d x^i, \]
    where $(\theta_1,\dots,\theta_n)$ denote the angle coordinates on $\Tt^n:=\R^n/\Z^n$.
\end{example}

\begin{remark}[Symplectic groupoids]
 The reader familiar with the groupoid language will recognize that a (symplectic) torus bundle is just a source-connected (symplectic) groupoid where the source and target maps coincide and multiplication is abelian. 
 
 Many of our results extend to the more general setting of symplectic groupoids. We will discuss this in future work. For now, we will sometimes use the language of symplectic groupoids in an auxiliary manner, but keeping the exposition self-contained. The curious reader can consult \cite{CFM21} for the basics of (symplectic) groupoids.
\end{remark}


Given a symplectic torus bundle $p:(\T,\Omega)\to M$, the symplectic form allows to identify the Lie algebras of the fibers with the fibers of $T^*M$. Using this identification, the fiberwise exponential maps assembles to a surjective, symplectic, submersion
\[ \exp:T^*M\to \T, \quad \exp^*\Omega=\Omega_\can. \]
Since the identity section is Lagrangian, it follows that its preimage
\[ \Lambda:=\ker(\exp)\subset T^*M, \]
satisfies the two properties in the following definition.

\begin{definition}
\label{defnintegral:aff:manifold}
An {\bf integral affine structure} on a manifold $M$ is a closed submanifold $\Lambda\subset T^*M$ satisfying:
\begin{enumerate}[(i)]
    \item $\Lambda$ is a Lagrangian submanifold of $(T^*M,\Omega_\can)$;
    \item $\Lambda_x:=\Lambda\cap T_x^*M$ is a full rank lattice for all $x\in M$.
\end{enumerate}
\end{definition}

\begin{remark}
One usually defines an integral affine structure on a manifold $M$ as a maximal atlas for which the coordinate changes are (restrictions of) maps of the form 
\[ x\mapsto Ax+b,\quad A\in \GL_m(\Z),\ b\in \R.\] 
Such atlases are in bijective correpondence with submanifolds $\Lambda\subset T^*M$ as in Definition \ref{defnintegral:aff:manifold}. Given $\Lambda$ one defines the integral affine charts to be any chart $(U,x^1,\dots,x^n)$ such that
\[ \Lambda|_U=\Z\d x^1+\cdots+\Z\d x^n. \]
Conversely, given such an integrable affine atlas the last expression defines $\Lambda$. Henceforth, we shall make use of this correspondence with no further notice.  We refer to \cite{CFT19,AM55,GH84} for details on integral affine structures.
\end{remark}

Hence, a symplectic torus bundle $p:(\T,\Omega)\to M$ determines a canonical integral affine structure $\Lambda$ on $M$ and the exponential map induces an equivalence
\[ T^*M/\Lambda\diffto \T. \]
This diffeomorphism pulls back $\Omega$ to the unique symplectic form ${\Omega}_\Lambda\in\Omega^2(T^*M/\Lambda)$
such that
\[ q^*{\Omega}_\Lambda=\Omega_\can, \]
where $q:T^*M\to T^*M/\Lambda$ is the projection.

Conversely, given any integral affine structure $\Lambda$ on a manifold $M$, we obtain a symplectic torus bundle $(\T_\Lambda,\Omega_\Lambda)$ where
\[ \T_\Lambda:=T^*M/\Lambda,\quad q^*\Omega_\Lambda=\Omega_\can, \]
and the projection $q:T^*M\to\T_\Lambda$ becomes the exponential map. It follows that there is a 1:1 correspondence:
\[ 
\left\{\txt{symplectic torus bundles\\ over $M$\\up to equivalence \,}\right\}
\quad \tilde{\longleftrightarrow}\quad 
\left\{\txt{integral affine structures\\ 
$\Lambda\subset T^*M$\\ \,}\right\}
\]

\begin{example}
    On $\R^n$ we have the standard integral affine structure
    \[ \Lambda_{\mathrm{st}}=\Z\{\d x^1,\dots,\d x^n\}, \]
    which corresponds to the trivial symplectic torus bundle
    \[ \pr_2:\Tt^n\times\R^n\to\R^n,\quad \Omega_{\mathrm{st}}=\sum_{i=1}^n \d\theta_i\wedge \d x^i.\]
    Similarly, on the torus $\Tt^n=\R^n/\Z^n$ we have the integral affine structure
    \[ \Lambda_{\mathrm{st}}:=\Z\{\d x^1,\dots,\d x^n\}, \]
    where now $x^i$ are standard coordinates on the torus (so defined mod $\Z$). The corresponding symplectic torus bundle is
    \[ \pr_2:\Tt^n\times\Tt^n\to \Tt^n,\quad \Omega_{\mathrm{st}}=\sum_{i=1}^n \d\theta_i\wedge \d x^i.\]
\end{example}

\begin{remark}
\label{rm:local:sympl:torus:bundle}
    Every symplectic torus $p:(\T,\Omega)\to M$ is locally equivalent to the standard one $(\Tt^n\times\R^n,\Omega_\st)\to\R^n$. Indeed, an integral affine chart $(U,x^i)$ induces  an isomorphism
\[ \T|_U\simeq \Tt^n\times U,  \]
 which identifies $\Omega$ and $\Omega_\st$.
\end{remark}

\begin{example}
\label{ex:KodairaThurston1}
    When $n=2$ we have an exotic integral affine structure on $\Tt^2$ defined by
    \[ \Lambda_{\mathrm{ex}}=\Z\{\d x^1,\d x^2-x^1\d x^1\}. \]
    The corresponding symplectic torus bundle can be described as follows. Consider the discrete subgroup $\Gamma\subset\Aff_\Z(\R^{4})$ generated by the translations
    \begin{align*}
        \gamma_1(\theta_1,\theta_2,x^1,x^2)=(\theta_1,\theta_2,x^1+1,x^2), &\quad \gamma_2(\theta_1,\theta_2,x^1,x^2)=(\theta_1,\theta_2,x^1,x^2+1),\\ \gamma_3(\theta_1,\theta_2,x^1,x^2)=(\theta_1+1,\theta_2,x^1,x^2),
    \end{align*}
    and by the element
    \[ 
        \gamma_4(\theta_1,\theta_2,x^1,x^2)=(\theta_1-x^1,\theta_2+1,x^1,x^2).
    \]
    This group acts proper and free, by symplectic transformations of the canonical symplectic form
    \[ \Omega_\can=\d \theta_1\wedge \d x^1+\d \theta_2\wedge \d x^2. \]
    The sought after symplectic torus bundle has total space $\T:=\R^4/\Gamma$, with symplectic form $\Omega$ induced by $\Omega_\can$, and projection 
    \[ p:\T\to \Tt^2,\quad [\theta_1,\theta_2,x^1,x^2]
    \mapsto [x^1,x^2]. \] 
    Indeed, this map has Lagrangian fibers and admits the Lagrangian section
    \[ s(x^1,x^2)=[0,0,x^1,x^2]. \]
    This is the zero section for the multiplication
    \[ m([\theta_1,\theta_2,x^1,x^2],[\bar{\theta}_1,\bar{\theta}_2,x^1,x^2])=[\theta_1+\bar{\theta}_1,\theta_2+\bar{\theta}_2,x^1,x^2]. \]
    Note that $(\T,\Omega)$ is the Kodaira-Thurston symplectic manifold.
\end{example}
\medskip 

Let $p:(\T,\Omega)\to M$ be a symplectic torus bundle inducing an integral affine structure $\Lambda\subset T^*M$. Then $M$ carries a {\bf canonical flat connection} $\nabla$, namely the unique torsion free connection satisfying
\[ \nabla_{\partial_{x^i}}\partial_{x^j}=0, \]
for any integral affine chart $(U,x^i)$.
We will denote by the same symbol $\nabla$ the dual connection on $T^*M$: it is the unique linear connection for which every local section of $\Lambda$ is flat.  The corresponding {\bf horizontal distribution} will be denoted by $H^\nabla\subset T(T^*M)$.

At the level of the torus bundle $\T\to M$ the existence of this flat connection is reflected in the existence of a Lagrangian foliation transverse to the fibers.

\begin{proposition}
\label{prop:connection}
    Let $p:(\T,\Omega)\to M$ be a symplectic torus bundle inducing an integral affine structure $\Lambda\subset T^*M$. Then
    \begin{equation}
        \label{eq:multiplicative:connection}
        H:=(\d \exp)(H^\nabla)\subset T(\T)
    \end{equation}
    is a flat Ehresmann connection by Lagrangian subspaces, hence defines a Lagrangian foliation transverse to the fibers of $p:\T\to M$.
\end{proposition}

\begin{proof}
Under the local trivialization induced by an integral affine chart (see Remark \eqref{rm:local:sympl:torus:bundle}), the exponential map $\exp:T^*M\to \T$ is just the projection
\[ \exp: \R^n\times U\to \Tt^n\times U,\]
and we find
\[ 
H=(\d \exp)(H^\nabla)=\bigcap_{i=1}^n\ker(\d\theta_i), 
\]
so the statement follows. 
\end{proof}

\begin{definition}
\label{defn:hor}
The distribution $H\subset T(\T)$ given by \eqref{eq:multiplicative:connection}
is called the {\bf canonical horizontal distribution} of the symplectic torus bundle.
\end{definition}

\begin{remark}
 \label{rem:multiplicative:distribuition}
 The distribution $H$ is an example of a \emph{multiplicative Ehresmann connection}. In general, a distribution in a Lie groupoid $\G$ is called \emph{multiplicative} if it is a subgroupoid of the tangent groupoid $T(\G)$. We refer to \cite{FM22a,FM22b} for more details on multiplicative Ehresmann connections.
 \end{remark}

\subsection{Morphisms of symplectic torus bundles}

We will make use of the following notion of morphism of symplectic torus bundles.

\begin{definition}
    \label{defn:morphism}
    A {\bf morphism} of symplectic torus bundles  
    \[ (\phi,\Phi):(\T_1,\Omega_1)\dto(\T_2,\Omega_2) \]
    consists of a map $\phi:M_1\to M_2$ between the bases and a bundle map 
    \[ 
    \xymatrix{
    \phi^*\T_2\ar[r]^{\Phi}\ar[d]_{\pr}& \T_1\ar[d]^{p_1}\\
    M_1\ar@{=}[r]_{\id} & M_1
    }
    \]
    satisfying:
    \begin{enumerate}[(i)]
    \item $\Phi$ restricts on each fiber to a group morphism;
    \item $\graph(\Phi)\subset \overline{\T_1}\times\T_2$ is a Lagrangian submanifold.
\end{enumerate}
\end{definition}

The graph in the last item in the previous definition is the relation defined by:
\[ \graph(\Phi):=\{(\Phi(g,x),g):(g,x)\in \T_2\times_{M_1} M_1\}. \] 
The notation $\overline{\T}$ means, as usual, the symplectic manifold $\T$ with the opposite symplectic structure. 

\begin{remark} 
The notion of morphism in the previous definition is actually an instance of the notion of a groupoid comorphism \cite{CDW13}. Since we do not consider groupoid morphisms or other notions of morphism for torus bundles, we will use the term ``morphism''.
\end{remark}

The notion of morphism of symplectic torus bundles is intimately related to the concept of integral affine map which we now recall (see \cite{CFT19,AM55,GH84} for more details).

\begin{definition}
\label{defn:morphism:IA:manifolds}
Let $(M_1,\Lambda_1)$ and $(M_2,\Lambda_2)$ be integral affine manifolds. A map
\begin{equation}\label{arrow:morphism:IA:manifolds} \phi:(M_1,\Lambda_1)\to (M_2,\Lambda_2)
    \end{equation} is called an \textbf{integral affine map} if $\phi^*(\Lambda_2)\subset \Lambda_1$. 
\end{definition}
\begin{remark}
A map $\phi:(M_1,\Lambda_1)\to (M_2,\Lambda_2)$ is integral affine if and only if in integral affine charts (with connected domains) it is of the form 
\[ x\mapsto Ax+b,  \] 
for some linear integral affine map $A:(\R^{m_1},\Z^{m_1})\to (\R^{m_2},\Z^{m_2})$ and some $b\in \R$. 
\end{remark}

The next proposition shows that the category $\IAMan$ of integral affine manifolds is equivalent to the category $\SympTB$.  

\begin{proposition} 
\label{prop:ia:morphisms}
The assignment $(M,\Lambda)\mapsto (\T_\Lambda,\Omega_\Lambda)$ extends to a functor
\[
\IAMan\to \SympTB
\]
by assigning to an integral affine map $\phi$ as in  (\ref{arrow:morphism:IA:manifolds}) the morphism
\[
    (\phi,\underline{\phi}^*):(\T_{\Lambda_1},\Omega_{\Lambda_1})\dto (\T_{\Lambda_2},\Omega_{\Lambda_2}), \quad \underline{\phi}^*([\alpha],x)=[(\d_x\phi)^*\alpha].
\] 
This functor is in fact an equivalence of categories.
\end{proposition}

\begin{proof} By the remark above the map $\underline{\phi}^*$ is well-defined. It is readily verified that $(\phi,\underline{\phi}^*)$ is a morphism, and that this construction defines a functor. To see that this is an equivalence of categories, let $(\phi,\Phi):(\T_1,\Omega_1)\dto (\T_2,\Omega_2)$ be any morphism between symplectic torus bundles. Then $\Phi$ has a unique lift to a vector bundle map
\[
\xymatrix{
\phi^*T^*M_2\ar[r]^{\widetilde{\Phi}}\ar[d]_{\id\times\exp} & T^*M_1\ar[d]^{\exp}\\
\phi^*\T_2\ar[r]_{\Phi} & \T_1
}
\]
whose graph
\[ \graph(\widetilde{\Phi})\subset \overline{T^*M_1}\times T^*M_2\]
is a Lagrangian submanifold. One checks (e.g., by working in local coordinates) that this last condition implies that
\[ \widetilde{\Phi}=\phi^*. \]
The commutativity of the diagram then forces 
\[ \phi^*(\Lambda_2)\subset \Lambda_1,\] 
where $\Lambda_1$ and $\Lambda_2$ are the integral affine structures induced by $(\T_1,\Omega_1)$ and $(\T_2,\Omega_2)$. By the remark above, we conclude that $\phi$ is an integral affine map.
\end{proof}

\subsection{Hamiltonian $\T$-spaces}
One of our main objects of study are Hamiltonian spaces for symplectic torus bundles. An {\bf action of a torus bundle} $p:\T\to M$ on a map $\mu:S\to M$ is given by a map
\[ \Act:\T\times_M S\to S, \quad (g,p)\mapsto g\cdot p,\]
satisfying the usual properties: 
\begin{enumerate}[(i)]
    \item $g\cdot (h\cdot p)=(gh)\cdot p$;
    \item $u(\mu(p))\cdot p=p$.
\end{enumerate}
In other words, this is just a smoothly varying fiberwise action.

Symplectic torus bundles act on maps whose domain is a symplectic manifold:

\begin{definition}
Let $p:(\T,\Omega)\to M$ be a symplectic torus bundle. A {\bf Hamiltonian $\T$-space} consists of:
\begin{enumerate}[(i)]
    \item a symplectic manifold $(S,\omega)$;
    \item a map $\mu:S\to M$;
    \item an action $\Act:\T\times_M S\to S$ which is symplectic, i.e., that satisfies
    \[ \Act^*\omega=\pr_{\T}^*\Omega+\pr_S^*\omega. \]
\end{enumerate}
\end{definition}

Since a Hamiltonian $\T$-space consists fiberwise of ordinary abelian actions, we obtain an {\bf infinitesimal action} of the bundle of abelian Lie algebras $T^*M\to M$ on $\mu:S\to M$. In other words, we have a linear map $\act:\Omega^1(M)\to \X(S)$ satisfying
\[ \mu_*(\act(\al))=0, \qquad [\act(\al),\act(\be)]=0, \quad (\al,\be\in \Omega^1(M)). \]

We call $\mu:S\to M$ the {\bf moment map} of the action if it satisfies the moment map condition (see \cite{CFM21})
\begin{equation}
\label{eqn:mommapcond}
i_{\act(\al)}\omega=\mu^*\al\quad (\al\in\Omega^1(M)).
\end{equation}

\begin{example}[Classical Hamiltonian spaces]\label{ex:classical:toric}
A classical Hamiltonian $\Tt^n$-space with moment map $\mu:(S,\omega)\to\R^n$ is the same thing as  Hamiltonian $\T$-space for the symplectic torus bundle $\T:=(\Tt^n\times\R^n,\Omega_\can)\to\R^n$. The $\T$-action and is given in terms of the $\Tt^n$-action determine each other by
\[ \Act:(\Tt^n\times\R^n)\times_{\R^n} S\to S,\ (g,x,p)\mapsto gp,\quad\text{ if }\quad
x=\mu(p). 
\]
Conversely, any $\T$-action takes this form for a unique $\Tt^n$-action.
\end{example}

\begin{example}[The trivial Hamiltonian $\T$-space]
Let $(\T,\Omega)$ be any symplectic torus bundle. Then $\T$ acts on itself by left translations
\[ \Act:\T\times_M \T\to \T, (g,h)\mapsto gh, \]
with moment map the projection $p:\T\to M$. The multiplicativity of $\Omega$ shows that this is a Hamiltonian $\T$-space.
\end{example}

\begin{example}[The trivial $\mathbb{S}^2$-bundle over $\mathbb{T}^2$]
\label{ex:trivspherebun:stdIA:cylinder}
 Consider the symplectic manifold
 \[ 
 (S,\omega):=(\mathbb{T}^2,\omega_{\mathbb{T}^2})\times (\mathbb{S}^2,\omega_{\mathbb{S}^2}),
 \] where $\omega_{\mathbb{T}^2}$ and $\omega_{\mathbb{S}^2}$ denote the area forms with total area 1 and 2, respectively. So, 
 \[
     \omega_{\mathbb{T}^2}:=\d y\wedge \d x,\quad \omega_{\mathbb{S}^2}:=\d \phi \wedge \d h,
 \] 
 where $(\phi,h)$ are cylindrical coordinates on $\Ss^2$.
 This admits an action of the symplectic torus bundle $(\T,\Omega)$ corresponding to the standard integral affine cylinder 
 \[ (M,\Lambda):=(\mathbb{S}^1\times \R,\Z\d x\oplus \Z\d h).\] 
 Explicitly, the symplectic torus bundle
 \begin{align*}
 (\T,\Omega)&=(\mathbb{T}^2\times \mathbb{S}^1\times \R,\d y\wedge \d x+\d \theta\wedge \d h), \\ 
 & p:\T\to \mathbb{S}^1\times \R,\quad (y,\theta,x,h)\mapsto (x,h),
 \end{align*}
 acts along the map 
 \[ 
 \mu:\mathbb{T}^2\times \mathbb{S}^2\to \mathbb{S}^1\times \R, \quad (y,x,\phi,h)\to (x,h),
 \] 
 as follows
 \[ 
 (y,\theta,x,h)\cdot (\tilde{y},x,\phi,h)=(y+\tilde{y},x,\theta+\phi,h).   
 \] 
\end{example}

\begin{example}[The non-trivial $\mathbb{S}^2$-bundle over $\mathbb{T}^2$] 
\label{ex:nontrivspherebun:nonstdIA:cylinder} 

The non-trivial orientable $\mathbb{S}^2$-bundle over $\mathbb{T}^2$ can be obtained by taking the quotient of the trivial $\mathbb{S}^2$-bundle $\R^2\times \mathbb{S}^2\to \R^2$ by the proper and free action of $\Z^2$ given by
\[ 
(k_1,k_2)\cdot (y,x,\phi,h)=(y+k_1,x+k_2,\phi-k_2 y,h).
\]
Let $(S,\omega)$ be the resulting manifold $\mathbb{T}^2\twprod \mathbb{S}^2:=(\R^2\times \mathbb{S}^2)/\Z^2$ equipped with the symplectic form induced by the $\Z^2$-invariant symplectic form
\[ (h+2)\d y\wedge\d x +(x\d y+\d \phi)\wedge \d h.\]      
This symplectic manifold admits an action of the symplectic torus bundle $(\T,\Omega)$ associated to the integral affine manifold
\[ 
(M,\Lambda)=(\mathbb{S}^1\times ]-2,\infty[,\Z((h+2)\d x+x\d h)\oplus \Z\d h).
\]
Explicitly, this torus bundle is
\begin{align*} 
\T&=\left(\mathbb{T}^2\times\R\times ]-2,\infty[\right)/{\Z} ,
\\ 
 & p:\T\to \mathbb{S}^1\times ]-2,\infty[, \quad (y,\theta,x,h)\mapsto (x,h),
\end{align*} 
where $\Z$ acts as
\[ 
k\cdot (y,\theta,x,h)=(y,\theta-ky,x+k,h),
\] 
and the multiplication on the fibers is just addition. The symplectic form $\Omega$ is induced by the $\Z$-invariant symplectic form
\[ (h+2)\d y\wedge\d x +(x\d y+\d \theta)\wedge \d h.\]
The symplectic torus bundle $(\T,\Omega)$ acts on $(S,\omega)$ along the map
\[ \mu:\mathbb{T}^2\twprod \mathbb{S}^2\to \mathbb{S}^1\times ]-2,\infty[, \quad (y,x,\phi,h)\mapsto (x,h),\]
as follows
\[ 
(y,\theta,x,h)\cdot(\tilde{y},x,\phi,h)=(y+\tilde{y},x,\theta+\phi,h).
\]
\end{example}
\medskip

Let $\mu_1:(S_1,\omega_1)\to M$ and $\mu_2:(S_2,\omega_2)\to M$ be Hamiltonian $\T$-spaces. Then a morphism from $\mu_1$ to $\mu_2$ is a symplectic map 
\[
\xymatrix{
S_1\ar[r]^\Psi\ar[d]_{\mu_1} & S_2\ar[d]^{\mu_2}\\
M\ar@{=}[r]_{\id} & M}
\]
which is $\T$-equivariant:
\[ \Psi(g\cdot p)=g\cdot\Psi(p).\]

We need a generalization of this notion to Hamiltonian spaces of distinct symplectic torus bundles. 

\begin{definition}
\label{defn:morphism:Hamiltonian:spaces}
Let $(\phi,\Phi):(\T_1,\Omega_1)\dto (\T_2,\Omega_2)$ be a morphism of symplectic torus bundles, and 
let $\mu_i:(S_i,\omega_i)\to M_i$ be Hamiltonian $\T_i$-spaces. A {\bf morphism} from $\mu_1$ to $\mu_2$ covering $(\phi,\Phi)$ is a map
\[
\xymatrix{
S_1\ar[r]^\Psi\ar[d]_{\mu_1} & S_2\ar[d]^{\mu_2}\\
M_1\ar[r]_{\phi} & M_2}
\]
satisfying:
\begin{enumerate}[(i)]
    \item $\Psi$ is a Poisson map;
    \item $\Psi$ is $(\phi,\Phi)$-equivariant, i.e.,
    $\Psi(\Phi(g,\mu_1(p))\cdot p)=g\cdot\Psi(p)$.
\end{enumerate}
\end{definition}

The base map of a morphism of Hamiltonian spaces is an integral affine map by Proposition \ref{prop:ia:morphisms}. For an {\bf isomorphism} this map must be an integral affine isomorphism, while the the map $\Psi$ between the total spaces must be a symplectomorphism. An isomorphism for which the base map $\phi$ is the identity, and so $\T_1=\T_2$, will be called an \textbf{equivalence}.

\begin{example}
Let $\mu:(S,\omega)\to\R^n$ be a Hamiltonian $\Tt^n$-space and $K_{\theta}:\Tt^m\to \Tt^n$ a Lie group homomorphism. The induced Lie algebra homomorphism $\theta:\R^m\to \R^n$ is a linear integral affine map $(\R^m,\Z^m)\to (\R^n,\Z^n)$, and we have a morphism of symplectic torus bundles
\[ (\phi,\Phi):\Tt^n\times\R^n\dto \Tt^m\times\R^m, \phi=\theta^*,\ \Phi(g,x)=(K_{\theta}(g),x). \]

On the other hand, the induced $\Tt^m$-action on $S$
\[ h\cdot x:=K_{\theta}(h)\cdot x,\]
is Hamiltonian, with momentum map $\tilde{\mu}:=\theta^*\circ\mu:S\to\R^m$. Moreover,
\[
\xymatrix{
S\ar[r]^{\id}\ar[d]_\mu & S\ar[d]^{\tilde{\mu}}\\
\R^n\ar[r]_{\phi} & \R^m}
\]
is a a morphism of Hamiltonian spaces covering $(\phi,\Phi)$.
\end{example}

\subsection{$\T$-invariant tensors}
Let $p:(\T,\Omega)\to M$ be a symplectic torus bundle. Given a Hamiltonian $\T$-space $\mu:(S,\omega)\to M$, the torus bundle $\T$ acts on the $\mu$-fibers so each $g\in \T$ induces a map between the tangent spaces to the fibers:
\[ g_*:T\mu^{-1}(x)\to T\mu^{-1}(x), \text{ with }x=p(g). \]
In order to extend the action of $g$ to the tangent bundle $TS$ we use the canonical horizontal distribution $H$ (see Definition \ref{defn:hor}). We denote the corresponding horizontal lift of tangent vectors by: 
\[ \hor_g:T_x M\to T_g\T,\quad \hor_g(w):=(\d_g p)|_H^{-1}(w). \]

\begin{definition}
\label{defn:action}
Given a symplectic torus bundle $p:(\T,\Omega)\to M$ and a Hamiltonian $\T$-space $\mu:(S,\omega)\to M$ the {\bf tangent action} of $g\in\T$ is the map
\[ 
g_*:T_p S \to T_{gp}S,\quad g_*v:=\d\Act(\hor_g(\d\mu(v)),v).
\]
\end{definition}

The multiplicativity of the distribution $H$ ensures that this defines an action of the torus bundle $\T$ on the map $\mu\circ\pr:TS\to M$.

\begin{remark}
The distribution $H$ is an example of a Cartan connection on the torus bundle $p:\T\to M$. The construction of the action in Definition \ref{defn:action} is a special case of a more general construction valid for actions of Lie groupoids $\G\tto M$ on a map $\mu:S\to M$. Whenever the Lie groupoid $\G$ is equipped with a Cartan connection one obtains an action of $\G$ on $\mu\circ\pr:TS\to M$ (see, e.g., \cite{AC13}).
\end{remark}

\begin{definition}
We say that a tensor field on the total space of a Hamiltonian $\T$-space $\mu:(S,\omega)\to M$ is {\bf $\T$-invariant} if it is invariant under the tangent action of every $g\in\T$.
\end{definition}

Similarly, one can talk about $\T$-invariance of other objects. For example, a distribution $D\subset TS$ is $\T$-invariant if is preserved under the tangent action. A foliation of $S$ is $\T$-invariant if the corresponding distribution is $\T$-invariant.

\begin{example}
    For any symplectic torus bundle $p:(\T,\Omega)\to M$ the symplectic form $\Omega$ and the horizontal distribution $H$ are both $\T$-invariant. For any Hamiltonian $\T$-space $\mu:(S,\omega)\to M$ the symplectic form $\omega$ is also $\T$-invariant. 
\end{example}

Given a Hamiltonian $(\T,\Omega)$-space $\mu:(S,\omega)\to M$, fixing a integral affine chart $(U,x^i)$, one obtains a $\Tt^n$-action on $S|_U$. This follows from the fact that a choice of integral affine chart induces a local trivialization of the symplectic torus bundle (cf.~Remark \ref{rm:local:sympl:torus:bundle}). Although these locally defined $\Tt^n$-actions depend on the choice of integral affine chart, it turns out that $\T$-invariance is equivalent to $\Tt^n$-invariance, as stated in the next proposition. For the statement, we will use the fact that a local section $s:U\to \T$ of $\T\to M$ acts on $S|_U$ as  the fiber preserving diffeomorphism 
\[ p\mapsto s(\mu(p))\cdot p. \]

\begin{proposition}
\label{prop:invariance:condition}
Given a tensor field $K$ on the total space of a Hamiltonian $\T$-space $\mu:(S,\omega)\to M$ the following are equivalent:
\begin{enumerate}[(i)]
    \item $K$ is  $\T$-invariant;
    \item $K$ is invariant under the action of any horizontal local section of $\T$;
    \item $K$ is invariant under the infinitesimal action of any locally defined flat 1-form
    \begin{equation}
    \label{eq:invariance}
    \Lie_{\act(\al)}K=0 \text{ if }\nabla\al=0;
    \end{equation} 
    \item $K$ is $\Tt^n$-invariant relative to the locally defined $\Tt^n$-action induced by any integral affine chart.
\end{enumerate}
\end{proposition}

\begin{proof}
    Obviously, a tensor is $\T$-invariant if and only if it is $\T|_{U_a}$-invariant for some open cover $\{U_a\}$ of $M$. Hence, it is enough to check the statement on an integral affine chart $(U,x^i)$ with connected domain.

    The equivalence of (i) and (ii) follows from the fact that any element $g\in\T$ extends to a horizontal local section (see proof of Proposition \ref{prop:connection}).

    For the equivalence between (ii) and (iii), observe that if $\al\in\Omega^1(U)$ is any flat 1-form then $\exp(t\al):U\to\T$ is a horizontal section, and acting by this local section is given by the flow of the vector field $\act(\al)\in\X(S)$ (cf.~\eqref{eqn:mommapcond})
    \[ \Act(\exp(t\al),p)=\phi^t_{\act(\al)}(p). \]

    For the equivalence between (ii) and (iv), observe that the choice of integral affine chart gives an identification
    \[ \Tt^n\times U\simeq \T|_U. \]
    Under this identification, the horizontal sections of $\T$ become the constant sections of $\Tt^n\times U\to U$. Therefore, the action by $g\in\Tt^n$ coincides with the action by the horizontal section $x\mapsto (g,x)$, and the equivalence follows.
\end{proof}

\section{Principal Hamiltonian $\T$-spaces a.k.a.~Lagrangian fibrations}
\label{sec:principal}

A very important class of Hamiltonian $\T$-spaces are \emph{Lagrangian fibrations}. In this section we show that they correspond exactly to principal Hamiltonian $\T$-spaces and we recall their classification, from this point of view.

\subsection{Lagrangian fibrations}
\label{subsec:lagrangian:fibrations}

By a \emph{Lagrangian fibration} we mean the following.

\begin{definition}
A {\bf Lagrangian fibration} is a symplectic manifold $(S,\omega)$ together with a surjective submersion $\mu:S\to M$ satisfying the following properties:
\begin{enumerate}[(i)]
\item $\dim(S)=2\dim(M)$;
\item $\mu:(S,\omega)\to M$ is a Poisson map, where $M$ is equipped with the zero Poisson structure;
\item $\mu$ is a proper with connected fibers.
\end{enumerate}
\end{definition}

Notice that conditions (i) and (ii) can be replaced by the simpler statement that the fibers are Lagrangian submanifolds. We prefer this formulation since it is more natural from the perspective of Poisson geometry adopted here, specially having in mind generalizations to the non-commutative setting.

Given any Lagrangian fibration $\mu:(S,\omega)\to M$ one obtains a symplectic torus bundle $p:(\T,\Omega)\to M$ which acts in a Hamiltonian fashion on $\mu$ as follows. First, the Lagrangian fibration induces an integral affine structure $\Lambda$ on $M$ given by
\[ \Lambda:=\{\al\in T^*M: \phi^1_{\act(\al)}=\id\}. \]
Here, for $\al\in T^*_xM$ we denoted by $\phi^1_{\act(\al)}$ the time-1 flow of the vector field $\act(\al)$ on the fiber $\mu^{-1}(x)$ defined by \eqref{eqn:mommapcond}. Then the symplectic torus bundle $(\T_\Lambda,\Omega_\Lambda)$ associated with $\Lambda$ acts on the fibration by setting:
\[ [\al]\cdot p:=\phi_{\act(\al)}^1(p)\quad (\al\in T_{\mu(p)}^*M). \]
One checks that this is indeed a Hamiltonian $\T_\Lambda$-action (see, e.g., \cite[Chp. 12]{CFM21}). Moreover, this is a {\bf principal action}, i.e., the map
\[
\T_\Lambda\times_M S\to S\times_M S,\quad (g,p)\mapsto (gp,p),
\]
is a diffeomorphism. 

Conversely, given a symplectic torus bundle $p:(\T,\Omega)\to M$, every \emph{principal} Hamiltonian $\T$-space $\mu:(S,\omega)\to M$ is a Lagrangian fibration. In other words:
\[ 
\left\{\\ \txt{principal Hamiltonian $\T_\Lambda$-spaces\\ $\mu:(S,\omega)\to M$}\right\}\ 
=\ 
\left\{\\ \txt{Lagrangian fibrations\\ $\mu:(S,\omega)\to M$ inducing $\Lambda$}\right\}
\]

Two Lagrangian fibrations $\mu_i:(S_i,\omega_i)\to M$ are called {\bf equivalent} if there exists a symplectomorphism $\Psi:(S_1,\omega_1)\to (S_2,\omega_2)$ commuting with the projections:
\[
\xymatrix{
S_1\ar[r]^\Psi\ar[d]_{\mu_1} & S_2\ar[d]^{\mu_2}\\
M\ar@{=}[r]_{\id} & M}
\]
It is easy to see that equivalent Lagrangian fibrations induce the same integral affine structure $\Lambda\subset T^*M$ and $\Psi$ is an equivalence of Hamiltonian $\T_\Lambda$-spaces.

\begin{example}[Symplectic torus bundles]
A symplectic torus bundle $p:(\T,\omega)\to M$ is an example of a Lagrangian fibration satisfying a very special property: it admits a global Lagrangian section, namely the identity section. Conversely, if a principal Hamiltonian $\T$-space $\mu:(S,\omega)\to M$ admits a global Lagrangian section $s:M\to S$ then we obtain an equivalence
\[ (\T,\Omega)\diffto (S,\omega),\quad g\mapsto g\cdot s(p(g)). \]
Under this equivalence, the $\T$-action on $S$ becomes the action of $\T$ on itself by translations.
\end{example}

In general, a Lagrangian fibration does not admit a global Lagrangian section. However, locally such sections always exists, and this gives a local model for a Lagrangian fibrations, as discussed in the next example.

\begin{remark}[Action-angle coordinates]
\label{rem:action-angle}
Let $p:(\T,\Omega)\to M$ be a symplectic torus bundle. Any principal Hamiltonian $\T$-space $\mu:(S,\omega)\to M$ is locally equivalent to $p:(\T,\Omega)\to M$. Namely, choosing an open set $U\subset M$ where there exists a Lagrangian section $s:U\to S$ of $\mu$, we obtain an equivalence as in the previous example
\[ \Psi:\T|_U\diffto S|_U, \quad g\mapsto g\cdot s(p(g)). \]
It follows also from the previous discussion that any two Lagrangian fibrations inducing the same integral affine structure $\Lambda$ are locally equivalent.

Moreover, if we fix integral affine coordinates $(U,x^i)$ over which there exists a Lagrangian section $s:U\to S$ of $\mu$, we obtain an equivalence (cf.~Remark \eqref{rm:local:sympl:torus:bundle})
\[ S|_U\simeq \Tt^n\times U,\quad \omega|_U\simeq\sum_{i=1}^n \d\theta_i\wedge \d x^i. \]
This local form only depends on the choice of an integral affine chart and a Lagrangian section. This is sometimes referred to as the Arnold-Liouville Theorem, and the coordinates $(x^i,\theta_i)$ are called action-angle coordinates (see, e.g., \cite{Evans22,SV18}). 
\end{remark}

\subsection{Classification of principal Hamiltonian $\T$-spaces}
\label{subsec:Lagrange:Chern:class}

The obstruction for a Lagrangian fibration to admit a global Lagrangian section is given by the \emph{Lagrangian Chern class}, as we now recall.
This class also yields Duistermaat's global classification of Lagrangian fibrations \cite{Dui80}. In our language, it allows to classify principal Hamiltonian $\T$-spaces, up to equivalence.

Given a symplectic torus bundle $p:(\T,\Omega)\to M$ denote by $\TLag$ its sheaf of Lagrangian sections. The groupoid multiplication makes this into an abelian sheaf. Given a principal Hamiltonian $\T$-space $\mu:(S,\omega)\to M$, the usual construction of the Chern class using \v{C}ech cocycles can be performed using local \emph{Lagrangian} sections. This yields a \v{C}ech cohomology class
\[ c_1(S,\omega)\in\check{H}^1(M,\TLag), \]
called the {\bf Lagrangian Chern class}.

\begin{theorem}[Duistermaat]
\label{thm:Duistermaat}
Given a symplectic torus bundle $p:(\T,\Omega)\to M$ the Lagrangian Chern class gives a 1:1 correspondence:
\[ 
c_1:\left\{\\ \txt{principal Hamiltonian $\T$-spaces\\ $\mu:(S,\omega)\to M$\\ up to equivalence\,}\right\}\ 
\tilde{\longrightarrow}\ 
\check{H}^1(M,\TLag)
\]
\end{theorem}

Since $\check{H}^1(M,\TLag)$ is an abelian group, one also has an abelian group structure on the equivalence classes of principal Hamiltonian $\T$-spaces. This group structure can be described explicitly using the fusion product of principal Hamiltonian $\T$-spaces (see, e.g., \cite{CFT19}).

An integral affine manifold $(M,\Lambda)$ is called \textbf{complete} if the canonical flat connection $\nabla$ is complete. Lagrangian fibrations over a complete integral affine manifold have the following simple description.

\begin{proposition}
\label{prop:classification:complete}
    Let $(M,\Lambda)$ be a complete, integral affine manifold so $M\simeq\R^n/\Gamma$. There is a 1:1 correspondence:
\[ 
\left\{\\ \txt{principal Hamiltonian $\T_\Lambda$-spaces\\ $\mu:(S,\omega)\to M$\,}\right\}\ 
\tilde{\longleftrightarrow}\ 
\left\{\\ \txt{symplectic actions\\ $\Gamma\ract (\Tt^n\times\R^n,\omega_\can)$\\ covering $\Gamma\ract\R^n$\,}\right\}
\]
\end{proposition}

\begin{proof}
    Since $(M,\Lambda)$ is complete, we have $M\simeq \R^n/\Gamma$, with $\Gamma=\pi_1(M)\subset\Aff_\Z(\R^n)$ acting properly and free on $\R^n$. 
    
    Assume $\Gamma\ract (\Tt^n\times\R^n,\omega_\can)$ is a symplectic action covering $\Gamma\ract\R^n$. It is a proper and free action, and the induced symplectic form on the quotient
    \[ S:=(\Tt^n\times\R^n)/\Gamma\to M=\R^n/\Gamma,\]
    yields a Lagrangian fibration with base the integral affine manifold $(M,\Lambda)$. So it is a principal Hamiltonian $\T_\Lambda$-space.

    Conversely, given a principal Hamiltonian $\T_\Lambda$-space $\mu:(S,\omega)\to M$ we consider its pullback under the covering projection $q:\R^n\to M$
    \[ 
    \xymatrix{\widetilde{S} \ar[d]_{\tilde{\mu}}\ar[r]^{\tilde{q}} & S\ar[d]^{\mu}\\
    \R^n\ar[r]^{q} & M}.
    \]
    Here $\tilde{q}:\widetilde{S}\to S$ is also a $\Gamma$-covering where $\Gamma$ acts on the first factor of $\widetilde{S}=\R^n\times_M S$. This action makes $\tilde{\mu}$ equivariant and preserves the pullback symplectic form
    \[\tilde{\omega}:=\tilde{q}^*\omega. \]
    On the other hand, $\tilde{\mu}:(\widetilde{S},\tilde{\omega})\to \R^n$ 
    is a principal Hamiltonian $\Tt^n$-space. Any such space is isomorphic to the canonical Hamiltonian $\Tt^n$-space $(\Tt^n\times\R^n,\omega_\can)\to\R^n$, so the result follows. 
\end{proof}

\begin{remark}
    The previous result can be related to Duistermaat's classification Theorem \ref{thm:Duistermaat} as follows. 
    On the one hand, under the assumptions of Proposition \ref{prop:classification:complete},one has an identification:
    \[ 
    \check{H}^1(M,\TLag)\simeq
    \check{H}^1_\Gamma(\R^n,\TLagn)
    \]
    where the right-hand side denotes $\Gamma$-equivariant cohomology for the $\Gamma$-sheaf of Lagrangian sections of $(\Tt^n\times\R^n,\omega_\can)\to\R^n$. 
    
    On the other hand, one has a $\Gamma$-equivariant version of Duistermaat result, which gives a 1:1 correspondence:
\[ 
\left\{\\ \txt{$\Gamma$-equivariant \\
principal
Hamiltonian $\Tt^n$-spaces\\ $\tilde{\mu}:(\widetilde{S},\tilde{\omega})\to \R^n$, up to isomorphism \,}\right\}\ 
\tilde{\longrightarrow}\ 
\check{H}^1_\Gamma(\R^n,\TLagn)
\]
 Any such Hamiltonian space admits a global Lagrangian section and so is symplectomorphic to the canonical bundle $\pr:(\Tt^n\times\R^n,\omega_\can)\to\R^n$, recovering the correspondence in the proposition. 
 
 Notice that the induced symplectic action of $\Gamma$ on $\Tt^n\times\R^n$, in general, does not coincide with the canonical lift of the action $\Gamma\ract\R^n$. This happens precisely when $\tilde{\mu}:(\widetilde{S},\tilde{\omega})\to \R^n$ admits a $\Gamma$-invariant Lagrangian section, i.e., when it is represented by the trivial class in $\check{H}^1_\Gamma(\R^n,\TLagn)$.
\end{remark}

Proposition \ref{prop:classification:complete} is also useful to construct examples. 

\begin{example}
\label{ex:trivial:chern:class}
    The standard integral affine structure on the torus $M=\R^2/\Gamma$ is obtained from the action of $\Gamma=\Z^2$ on $\R^2$ by translations:
    \[ 
    (m,n)\cdot(x^1,x^2)=(x^1+m,x^2+n). 
    \]
    The canonical lift $\Gamma\ract\R^2\times\Tt^2$ is given by
    \[
    (m,n)\cdot(x^1,x^2,\theta_1,\theta_2)=(x^1+m,x^2+n,\theta_1,\theta_2).
    \]
    The zero section is a $\Gamma$-equivariant Lagrangian section and this agrees with the fact that the corresponding quotient is the trivial bundle $(\Tt^2\times\Tt^2,\omega_\can)\to \Tt^2$, which has representative the trivial class in $\check{H}^1(\Tt^2,\TLag)$
    
    More generally, fix $a\in [0,1[$ and consider the symplectic action $\Gamma\ract\R^2\times\Tt^2$ covering $\Gamma\ract\R^2$ given by
    \begin{equation}
        \label{eq:action:funny}
        (m,n)\cdot(x^1,x^2,\theta_1,\theta_2)=(x^1+m,x^2+n,\theta_1+a n,\theta_2).
    \end{equation}
    For this action, the bundle $\R^2\times\Tt^2\to \R^2$ still admits a $\Gamma$-equivariant section, namely
    \[ s(x^1,x^2)=(x^1,x^2,-a x^2,0). \]
    Hence, we obtain a family of Lagrangian fibrations, each isomorphic to the trivial fibration $\Tt^2\times\Tt^2\to \Tt^2$, but with a new symplectic form
    \[ \omega_a=\omega_\can+a\,\d x^1\wedge \d x^2. \]
    Notice that the integral affine structure induced on the base $\Tt^2$ is still the standard one. 
    If $s:\Tt^2\to \Tt^2\times\Tt^2 $ is any section, $s^*\omega_\can$ is an exact form and applying Stokes formula we find that
    \[ \int_{\Tt^2}s^*\omega_a=a\int_{\Tt^2}\d x^1\wedge \d x^2 =a. \]
    Therefore, if $a\in ]0,1[$ the fibration $(\Tt^2\times\Tt^2,\omega_a)\to \Tt^2$ does not admit any Lagrangian section. Equivalently, there is no $\Gamma$-equivariant section of $\R^2\times\Tt^2\to \R^2$ which is Lagrangian (this can also be checked directly). This gives examples of Lagrangian fibrations with trivial Chern class and non-trivial Lagrangian Chern class $c_1(S,\omega)\in\check{H}^1(\Tt^2,\TLag)$.
    
\end{example}

\subsection{Lagrangian connections} 
\label{sec:Lag:connections}

In the construction of the Chern class of a principal torus bundle one passes from 
integral to real coefficients using connections. One has a similar situation for principal Hamiltonian $\T$-spaces, for which the concept of a \emph{Lagrangian connection} is needed. 

\begin{definition} 
\label{def:connection}
A \textbf{connection} for a principal Hamiltonian $\T$-space $\mu:(S,\omega)\to M$ is a $\T$-invariant distribution $D\subset TS$ such that $TS=D\oplus \ker(\d\mu)$. We call such a connection \textbf{Lagrangian} if $D_p\subset (T_pS,\omega_p)$ is Lagrangian for all $p\in S$. 
\end{definition}

\begin{example}
\label{example:hordistr=connection}
    For $p:(\T,\Omega)\to M$, viewed as a principal Hamiltonian $\T$-space with the action by translations, the canonical horizontal distribution (see Definition \ref{defn:hor}) is a Lagrangian connection. 
\end{example}

\begin{example}
    Let $\mu:(S,\omega)\to \R^n$ be a classical Hamiltonian $\Tt^n$-space, viewed as Hamiltonian $\T$-space for the symplectic torus bundle $\T=(\Tt^n\times\R^n,\Omega_\can)\to\R^n$. If the action is principal, then a connection in the sense of Definition \ref{def:connection} is just an ordinary principal bundle connection. It is Lagrangian if and only if its horizontal spaces are Lagrangian.
\end{example}

Connections can also be described by \emph{connection $1$-forms}, which are defined as follows. For a Hamiltonian $\T$-space $\mu:(S,\omega)\to M$ and a vector bundle $E\to M$, let $\Omega^k(S,E)$ be the space of $k$-forms with values in $\mu^*E$.

\begin{definition} 
\label{def:connection:form}
A \textbf{connection 1-form} for a principal Hamiltonian $\T$-space $\mu:(S,\omega)\to M$ is a form $\theta\in \Omega^1(S,T^*M)$ which satisfies:
\begin{itemize}
\item[(i)] $\T$-invariance: $\theta_{gp}(g\cdot v)=\theta_p(v)$, for all $v\in T_pS$, $g\in\T_{\mu(p)}$;
\item[(ii)] $\theta_p(\act_p(\alpha))=\alpha$, for all $\alpha\in T_{\mu(p)}^*M$ and $p\in S$.
\end{itemize}
\end{definition}
Note that in the previous definition, in (i) we have used the tangent action of $\T$ on $S$ and in (ii) its infinitesimal version $\act:\mu^*T^*M\to TS$. 

Similar to ordinary principal connections, there is a bijective correspondence between connections and connection $1$-forms for a principal Hamiltonian $\T$-space $\mu:(S,\omega)\to M$, in which the distribution corresponding to a connection $1$-form $\theta$ is $D=\ker\theta$. The following lemma characterizes those connection 1-forms giving rise to Lagrangian connections. 

\begin{lemma} 
\label{lem:Lagrangian:connection}
Let $\theta$ be a connection $1$-form for a principal Hamiltonian $\T$-space $\mu:(S,\omega)\to M$. Consider the $\T$-invariant, non-degenerate, $2$-form $\omega_\theta\in \Omega^2(S)$ defined by
\[
    \omega_\theta(v,w):=\langle \theta(v),\d\mu(w)\rangle - \langle \theta(w),\d\mu(v)\rangle. 
\] 
The connection $\ker\theta$ is Lagrangian if and only if $\omega=\omega_\theta$.
\end{lemma}

\begin{proof} The backward direction is immediate. For the forward direction, note that $\im(\act)=\ker(\d\mu)$ by principality of the action. So, it suffices to verify the equality $\omega=\omega_\theta$ on tangent vectors of the form $v+\act_p(\alpha)$, with $v\in \ker(\theta_p)$ and $\alpha\in T^*_{\mu(p)}M$. This is readily done using the moment map condition (\ref{eqn:mommapcond}).
\end{proof}

\begin{proposition} \label{prop:existenceLagrconn:principalcase}
Every principal Hamiltonian $\T$-space admits a Lagrangian connection. 
\end{proposition}
\begin{proof} 
Remark \ref{rem:action-angle} and Example \ref{example:hordistr=connection} show that over a sufficiently small integral affine chart a Lagrangian connection always exists. Consider an open cover $\{U_i\}_{i\in I}$ of $M$ together with connection $1$-forms $\theta_i$ for $\mu:S|_{U_i}\to U_i$, such that $\ker(\theta_i)$ is Lagrangian. Choosing a partition of unity $\{\rho_i\}_{i\in I}$ subordinated to this cover, we can construct a global connection $1$-form $\theta:=\sum_{i\in I}(\rho_i\circ\mu)\theta_i$. In view of the lemma above, $\ker(\theta)$ is Lagrangian, since
\[
    \omega_\theta=\sum_{i\in I}(\rho_i\circ\mu)\omega_{\theta_i}=\sum_{i\in I}(\rho_i\circ\mu)\omega=\omega.
\] 
\end{proof}

Next we define the curvature 2-form of a connection on a principal Hamiltonian $\T$-space $\mu:(S,\omega)\to M$. Pulling back the canonical flat connection $\nabla$ on $T^*M$ along the projection $\mu:S\to M$ we obtain a flat connection $\nabla$ on $\mu^*T^*M$ and this gives a differential $\d^\nabla:\Omega^\bullet(S,T^*M)\to \Omega^{\bullet+1}(S,T^*M)$. We let $\Omega^\bullet_\partial(M,T^*M)$ denote the kernel of the anti-symmetrization map
\begin{equation} 
\label{eqn:antisymmap}
\partial:\Omega^\bullet(M,T^*M)\to \Omega^{\bullet+1}(M),
\end{equation} 
which is a subcomplex of $(\Omega^\bullet(M,T^*M),\d^\nabla)$.

\begin{proposition}
\label{prop:curvature}
    Let $D$ be a connection on $\mu:(S,\omega)\to M$ with connection 1-form $\theta\in\Omega^1(S,T^*M)$. There is a unique 2-form $K_{\theta}\in \Omega^2(M,T^*M)$ such that $\mu^*K_{\theta}=\d^\nabla\theta$.
    This satisfies the following properties:
    \begin{enumerate}[(i)]
        \item $K_{\theta}$ is $\d^\nabla$-closed;
        \item $K_{\theta}=0$ if and only if $D$ is involutive.
    \end{enumerate}
    Moreover, if $D$ is a Lagrangian connection then $K_{\theta}\in \Omega^2_\partial(M,T^*M)$, meaning
    \[\partial K_{\theta}(v_1,v_2,v_3):=\sum_{\sigma\in S_3}(-1)^{|\sigma|}\langle K _{\theta}(v_{\sigma(1)},v_{\sigma(2)}),v_{\sigma(3)}\rangle=0. \]
\end{proposition}

\begin{proof} The existence of $K_\theta$ follows from the remark that $\d^\nabla\theta$ is horizontal and $\T$-invariant, and its uniqueness is clear. Properties (i) and (ii) are immediate, whilst the last property follows by observing that \[ \d\omega_\theta=\frac{1}{2}\mu^*(\partial K_\theta), \] 
    so that (by Lemma \ref{lem:Lagrangian:connection}) $\partial K_\theta$ must vanish if $\theta$ is Lagrangian.
\end{proof}

According to the previous proposition, a choice of Lagrangian connection on $\mu:(S,\omega)\to M$ gives rise to a cohomology class
\[ [K_\theta]\in H^2(\Omega^\bullet_\partial(M,T^*M),\d^\nabla). \]
This class is independent of the choice of Lagrangian connection and, as for principal $\mathbb{S}^1$-bundles, it is related to the Lagrangian Chern class by a map that `passes from integer to real coefficients'. To make this precise, consider the composite
 \begin{equation}
        \label{eq:map:Lag:chern:class}
        \xymatrix@C=15pt{
        \check{H}^1(M,\TLag)\ar[r]^\simeq& \check{H}^2(M,\O_\Lambda)\ar[r] & \check{H}^2(M,\O_{\Aff}) \ar[r]^---\simeq &  H^2(\Omega^\bullet_\partial(M,T^*M),\d^\nabla)},
    \end{equation} where $\O_{\Aff}$ and $\O_\Lambda$ denote respectively the sheaves of locally defined functions on $M$ that are affine and integral affine with respect to $\Lambda$. The first map in this sequence is the connecting homomorphism induced by the short exact sequence
    \[ 
    \xymatrix{
    0\ar[r] & \O_\Lambda\ar[r] &  \mathcal{C}^\infty\ar[r]^--{\d}& \TLag\ar[r] & 0,}
    \]
    while the last map is the isomorphism induced by the fine resolution
    \begin{equation}\label{eqn:fineresafffnc} 
    \xymatrix{
    0\ar[r] & \O_{\Aff}\ar[r] &\mathcal{C}^\infty\ar[r]^---{\d^\nabla\circ \d} & \Omega^1_\partial(-,T^*M)\ar[r]^{\d^\nabla}& \Omega^2_\partial(-,T^*M)\ar[r]^--{\d^\nabla}& \dots}
    \end{equation}
\begin{theorem}
    \label{thm:real:Lag-Chern:class} 
    If $\theta$ is a Lagrangian connection for a principal Hamiltonian $\T$-space $\mu:(S,\omega)\to M$, then \eqref{eq:map:Lag:chern:class} maps the Lagrangian Chern class $c_1(S,\omega)$ to the class $[K_\theta]$. 
\end{theorem}

\begin{proof} Let $\mathcal{U}=\{U_i\}_{i\in I}$ be a good open cover of $M$ such that $\mu$ admits a Lagrangian section $\sigma_i$ over $U_i$ for each $i\in I$. The class $c_1(S,\omega)\in \check{H}^1(M,\TLag)$ is represented by the $1$-cocycle $\tau=[\tau_{ij}]\in \check{C}^1(\mathcal{U},\TLag)$, defined by the relation:
\[
    \tau_{ij}(x)\cdot \sigma_i(x)=\sigma_j(x), \quad x\in U_{ij}. 
\]
The connecting homomorphism in \eqref{eq:map:Lag:chern:class} sends $c_1(S,\omega)$ to the class represented by the $2$-cocycle $\check{\d}f\in \check{C}^2(\mathcal{U},\O_\Lambda)$, where $f=[f_{ij}]\in \check{C}^1(\mathcal{U},\mathcal{C}^\infty)$ is any $1$-cochain lifting $\tau$ (which exists since $U_{ij}$ is contractible for all $i$ and $j$). The last map in \eqref{eq:map:Lag:chern:class} is induced by the `\v{C}ech-de Rham'-type double complex
\[
\xymatrix{ \vdots & \vdots & \vdots & 
\\
\check{C}^2(\mathcal{U},\mathcal{C}^\infty)\ar[r]\ar[u] & \check{C}^2(\mathcal{U},\Omega^1_\partial(-,T^*M))\ar[r]\ar[u] & \check{C}^2(\mathcal{U},\Omega^2_\partial(-,T^*M)) \ar[r]\ar[u] & \dots \\
\check{C}^1(\mathcal{U},\mathcal{C}^\infty)\ar[r]\ar[u] & \check{C}^1(\mathcal{U},\Omega^1_\partial(-,T^*M))\ar[r]\ar[u] & \check{C}^1(\mathcal{U},\Omega^2_\partial(-,T^*M))\ar[r]\ar[u] & \dots \\
\check{C}^0(\mathcal{U},\mathcal{C}^\infty)\ar[r]\ar[u] & \check{C}^0(\mathcal{U},\Omega^1_\partial(-,T^*M))\ar[r]\ar[u] & \check{C}^0(\mathcal{U},\Omega^2_\partial(-,T^*M))\ar[r]\ar[u] & \dots 
}
\] 
In this complex the vertical differentials are those of the \v{C}ech complex and the horizontal differentials are induced by \eqref{eqn:fineresafffnc}. Notice that both the rows and columns are exact at the $(p,q)^\textrm{th}$ entry whenever $p,q>0$. We need to show that $(\check{\d}f,0,0)$ and $(0,0,K_\theta)$ represent the same class in the cohomology of the total complex. This follows because 
\[(\check{\d}f,0,0)-(0,0,K_\theta)=\d_\textrm{Tot}(f,\sigma^*\theta),
\] 
where $\sigma^*\theta\in \check{C}^0(\mathcal{U},\Omega^1_\partial(M,T^*M))$ is given by $(\sigma^*\theta)_i:=\sigma_i^*\theta$. Note that $\sigma_i^*\theta$  belongs to the kernel of $\partial$ because $\sigma_i$ and $\theta$ are Lagrangian and, by Lemma \ref{lem:Lagrangian:connection}, one has
\[ 
\partial(\sigma_i^*\theta)=\sigma_i^*\omega_\theta=\sigma_i^*\omega=0. 
\]
\end{proof}

\subsection{Flat Lagrangian connections} 
\label{sec:Lag:connections:flat}

Flat Lagrangian connections will play a key role in the study of invariant K\"ahler metrics. In this subsection we show how Lagrangian fibrations with such connections can be classified in terms of a sheaf cohomology group, much like principal $\Tt^n$-bundles with flat connection. 

\begin{definition}
A \textbf{flat Lagrangian fibration} is a Lagrangian fibration equipped with a flat Lagrangian connection.
\end{definition}

Note that, by Proposition \ref{prop:curvature}, a flat Lagrangian connection is the same thing as an invariant Lagrangian foliation transverse to the fibers.

\begin{remark}
    In the literature on T-duality and mirror symmetry flat Lagrangian fibrations in our sense are sometimes called \emph{semi-flat} (see, e.g., \cite{ABCDGMSSW}).
\end{remark}

Given an integral affine manifold $(M,\Lambda)$ we denote by $\TFlat$ the sheaf of $\nabla$-flat sections of $\T_\Lambda$. If $(U,x^i)$ an integral affine chart then the flat sections of $\T_\Lambda|_U$ take the form
\[ s=\sum_{i=1}^n c_i\d x^i \pmod{\Lambda}, \]
where the coefficients $c_i$ are locally constant. This shows that every flat section is automatically a Lagrangian section. Hence, $\TFlat$ is a subsheaf of $\TLag$ and we have a map in cohomology
\[
\Check{H}^1(M,\TFlat)\to \Check{H}^1(M,\TLag).
\]
This map is part of an exact sequence, which will allow us to express the first obstruction.

\begin{proposition}
\label{prop:L:flat}
Given a symplectic torus bundle $(\T,\Omega)\to M$ the sequence
\begin{equation}
\label{eq:image:flat:cohomology}
\xymatrix{\Check{H}^1(M,\TFlat)\ar[r] & \Check{H}^1(M,\TLag) \ar[r] & H^2(\Omega^\bullet_\partial(M,T^*M),\d^\nabla),}
\end{equation}
where the second map is given by {\eqref{eq:map:Lag:chern:class}}, is exact.
\end{proposition}
\begin{proof}[Proof of Proposition \ref{prop:L:flat}] Consider the map of short exact sequences:

\[ 
\xymatrix{ 
0\ar[r] & \O_\Lambda\ar[r]\ar[d] & \mathcal{C}^\infty\ar[r]^{\d}\ar@{=}[d] & \TLag\ar[r]\ar[d] & 0\\
0 \ar[r]& \O_{\Aff}\ar[r] & \mathcal{C}^\infty\ar[r] & \TLag/\TFlat\ar[r] & 0 
}
\] By naturality of the connecting homomorphism, the induced square:
\[ 
\xymatrix{ 
\check{H}^1(M,\TLag)\ar[r]^{\sim}\ar[d] & \check{H}^2(M,\O_\Lambda)\ar[d] \\
 \check{H}^1(M,\TLag/\TFlat) \ar[r]^{\sim}& \check{H}^2(M,\O_{\Aff})  
}
\] commutes. So, up to isomorphism, \eqref{eq:image:flat:cohomology} is  part of the long exact sequence of:
\[ 0\to \TFlat\to \TLag\to \TLag/\TFlat\to 0. 
\]


\end{proof} 
Combining this result, Proposition \ref{prop:curvature} and Theorem \ref{thm:real:Lag-Chern:class} we obtain the following corollary.

\begin{corollary}
    \label{cor:L:flat}
    Fix an integral affine manifold $(M,\Lambda)$. Given a Lagrangian fibration $\mu:(S,\omega)\to M$ inducing $\Lambda$ the following conditions are equivalent:
    \begin{enumerate}[(i)]
        \item $\mu:(S,\omega)\to M$ admits a flat Lagrangian connection;
        \item the Lagrangian Chern class $c_1(S,\omega)\in\Check{H}^1(M,\TLag)$ lies in the image of \eqref{eq:image:flat:cohomology}.
    \end{enumerate}
\end{corollary}

Flat Lagrangian fibrations are classified by their \textbf{flat Chern class}. This is the element in $\check{H}^1(M,\TFlat)$ constructed as the usual Chern class but using only flat local sections.
More precisely, two flat Lagrangian fibrations $\mu_i:(S_i,\omega_i,\theta_i)\to (M,\Lambda)$ are said to be {\bf equivalent} if there exists an equivalence of Lagrangian fibrations preserving the connections:
\[
\vcenter{\xymatrix{
S_1\ar[r]^\Psi\ar[d]_{\mu_1} & S_2\ar[d]^{\mu_2}\\
M\ar@{=}[r]_{\id} & M}}\qquad \Psi^*\theta_2=\theta_1.
\]
Then one has the following classification result.

\begin{theorem}
\label{thm:bi-lagrangian:classification}
Fix an integral affine manifold $(M,\Lambda)$. Assigning to a flat Lagrangian fibration its flat Chern class yields a canonical 1:1 correspondence
\[ 
\left\{\\ \txt{flat Lagrangian fibrations\\ $\mu:(S,\omega,\theta)\to M$\\ up to equivalence \,}\right\}\ 
\tilde{\longleftrightarrow}\ 
\check{H}^1(M,\TFlat)
\]
Moreover, the flat Chern class is mapped to the Lagrangian Chern class by \eqref{eq:image:flat:cohomology}.
\end{theorem}

\begin{proof}
Using flat sections, the proof of the theorem is entirely similar to the proof of Theorem \ref{thm:Duistermaat} (see, e.g.,\cite{Dui80} or \cite{CFT19}).
\end{proof}

\begin{remark}\label{rem:cohomology:TFlat}
Note that $\TFlat$ is the locally constant sheaf associated to the representation of $\pi_1(M)$ on $\T_\Lambda$ given by the holonomy of the integral affine manifold $(M,\Lambda)$. In particular, if the holonomy representation if trivial then the cohomology of $\TFlat$ is that of $M$ with values in the abelian group $\mathbb{T}^m$, where $m=\dim(M)$. For example, when $M$ is simply-connected one obtains that the cohomology vanishes in degree one.
\end{remark}

Finally, we observe that one can specialize the correspondence given in Proposition \ref{prop:classification:complete} to obtain the following useful description of flat Lagrangian fibrations over complete integral affine manifolds.

\begin{corollary}
    \label{cor:classification:complete:flat}
    Let $(M,\Lambda)$ be a complete integral affine manifold (meaning that $M\simeq\R^n/\Gamma$) and let $H$ be the canonical horizontal distribution of $\Tt^n\times\R^n\to\R^n$. There is a 1:1 correspondence
    \[ 
    \left\{\\ \txt{flat Lagrangian fibrations\\ $\mu:(S,\omega,\theta)\to M$\\ up to equivalence \,}\right\}\quad 
    \tilde{\longleftrightarrow}\quad
    \left\{\\ \txt{symplectic actions\\ $\Gamma\ract (\Tt^n\times\R^n,\omega_\can)$\\ covering $\Gamma\ract\R^n$\\
    preserving $H$ \,}\right\}
    \]
\end{corollary}
\smallskip

\begin{example} 
\label{ex:complextori1}
A simple example illustrating the above correspondences can be obtained by considering the integral affine circle 
$(M,\Lambda)=(\mathbb{S}^1,b\,\Z\,\d x)$, where $b>0$ is a parameter. Then $H^1(M,\TLag)=0$, so any Lagrangian fibration inducing $(M,\Lambda)$ takes the form
\[ 
\pr_{2}:(\mathbb{S}^1\times\mathbb{S}^1,b\,\d y\wedge\d x)\to \mathbb{S}^1.
\] 
By the remark above, we have $H^1(M,\TFlat)=\mathbb{S}^1$. A class $e^{2\pi i a}\in\Ss^1$ can be realized as the Lagrangian connection
\[ D_a=\langle\partial_x-a\partial_y\rangle. \]
\end{example}

\begin{example} 
\label{ex:trivial:chern:class:2}
Consider the second Lagrangian fibration 
\[ (\Tt^2\times\Tt^2,\omega_a)\to \Tt^2 \] 
in Example \ref{ex:trivial:chern:class}. We saw there that for $a\in]0,1[$ this fibration has non-trivial Lagrangian Chern class. On the other hand, the canonical horizontal distribution of $\Tt^2\times\R^2\to\R^2$ is invariant under the action \eqref{eq:action:funny}, so it follows from Corollary \ref{cor:classification:complete:flat} that this Lagrangian fibration is flat.
\end{example}

\section{Toric Hamiltonian $\T$-spaces a.k.a.~toric Lagrangian fibrations}
\label{sec:toricspaces}

The main class of Hamiltonian $\T$-spaces that we study in this paper is the following one.

\begin{definition} 
\label{def:toric:Hamiltonian:space}
A Hamiltonian $\T$-space $\mu:(S,\omega)\to M$ is called \textbf{toric} if it satisfies the following conditions:
\begin{enumerate}[(i)]
\item The $\T$-action is free on a dense subset of $S$;
\item $\dim(S)=2\dim(M)$;
\item The map $\mu$ has connected fibers and it is proper as map onto its image. 
\end{enumerate} 
\end{definition}

{ Toric Hamiltonian $\T$-spaces include as special cases 
both principal Hamiltonian $\T$-spaces (when the action is free everywhere and $\mu$ is surjective) and classical symplectic toric manifolds (when $\T=\Tt^n\times \R^n$ -- see Example \ref{ex:classical:toric}).}

In the rest of this section we will see that the theory of principal Hamiltonian $\T$-spaces, studied in the previous section, extends to the toric setting. In particular, in section \ref{subsec:toric:lagrangian:fibrations}, we extend the correspondence between principal Hamiltonian $\T$-spaces and Lagrangian fibrations by showing that toric Hamiltonian $\T$-spaces are essentially the same thing as toric Lagrangian fibrations, i.e., singular Lagrangian fibrations with only elliptic singularities.


\subsection{Toric Lagrangian fibrations}
\label{subsec:toric:lagrangian:fibrations}

We consider the following special class of singular Lagrangian fibrations.

\begin{definition}
By a {\bf singular Lagrangian fibration} we mean a symplectic manifold $(S,\omega)$ together with a smooth map $\mu:S\to M$ satisfying:
\begin{enumerate}[(i)]
\item $\dim(S)=2\dim(M)$;
\item $\mu:(S,\omega)\to M$ is a Poisson map, where $M$ is equipped with the zero Poisson structure;
\item $\mu$ has connected fibers, is proper as a map onto its image and its regular points form a dense open set in $S$.
\end{enumerate}
A {\bf toric Lagrangian fibration} is a singular Lagrangian fibration whose only singularities are of elliptic type.
\end{definition}

A \textit{singularity} or \textit{singular point} of a Lagrangian fibration is a point $p_0\in S$ at which its differential is not surjective. An \emph{elliptic type} singularity is the simplest non-degenerate singularity that a singular Lagrangian fibration can possess -- see, e.g., \cite{DM90,Zung96}. Instead of giving its definition, we recall the following characterization.

\begin{theorem}[Dufour and Molino \cite{DM90}]
\label{thm:Dufour-Molino}
Let $\mu:(S,\omega)\to M$ be a singular Lagrangian fibration. A singular point $p_0\in S$ of $\mu$ is elliptic if and only if there is a neighborhood $U$ of $x_0:=\mu(p_0)$ and an embedding $\phi:U\to\R^n$, such that $\phi\circ\mu:(\mu^{-1}(U),\omega)\to \R^n$ is the moment map of an effective $\Tt^n$-action. 
\end{theorem}

Note that the \emph{regular fibers} of a singular Lagrangian fibration are indeed Lagrangian submanifolds. For a general singular Lagrangian fibration a fiber of $\mu:S\to M$ can contain both regular and singular points. However, the previous result implies that for a \emph{toric} Lagrangian fibration $\mu:S\to M$ this cannot happen {and:
\begin{enumerate}[(i)]
\item the rank of $\d\mu$ is the same for all points in the same fiber;
\item the image $\Delta:=\mu(S)$ is a codimension 0 submanifold of $M$ with corners;
\item the set of regular points $\mathring{S}\subset S$ is the pre-image of the interior $\mathring{\Delta}$ of $\Delta$.
\end{enumerate} }
Therefore, the restriction $\mu:(\mathring{S},\omega)\to \mathring{\Delta}$ is a (regular) Lagrangian fibration so, as we saw in Section \ref{subsec:lagrangian:fibrations}, there is an induced integral affine structure $\mathring{\Lambda}$ on $\mathring{\Delta}$ and $\mu:(\mathring{S},\omega)\to \mathring{\Delta}$ is a principal Hamiltonian $\T_{\mathring{\Lambda}}$-space. 

The following result shows that the integral affine structure on $\mathring{\Delta}$ extends essentially uniquely to an open set $M'\subset M$ containing $\Delta$, so $\mu:(S,\omega)\to M'$ becomes a toric Hamiltonian $\T$-space. { In the proof we use integral affine structures on manifolds with corners, which can be defined either by an atlas whose transition functions are (restrictions of) integral affine transformations or by a Lagrangian lattice in the cotangent bundle that around boundary points is generated by coordinate coframes.}

\begin{theorem}
\label{thm:equiv:toric:Lagrangian}
Let $\mu:(S,\omega)\to M$ be a toric Lagrangian fibration. Then there is an open set $M'\subset M$ containing 
 $\mu(S)$ and an integral affine structure $\Lambda$ on $M'$ such that $\mu:(S,\omega)\to M'$ is a toric Hamiltonian $\T_\Lambda$-space. The germ of $(M',\Lambda)$ is unique up to an isomorphism that is the identity on $\mu(S)$. Conversely, every toric Hamiltonian $\T$-space is a toric Lagrangian fibration.
\end{theorem}

\begin{proof}
{Let $\mu:(S,\omega)\to M$ be a toric Lagrangian fibration. In order to construct an integral affine structure $\Lambda$ on an open neighborhood $M'$ of $\mu(S)$ in $M$ we first observe that the integral affine structure on $\mathring{\Delta}$ extends to the boundary $\partial\Delta$, so $\Delta$ is an integral affine manifold with corners. Around a boundary point $x_0\in\partial\Delta$ an integral affine chart is provided by the restriction $\phi:V\cap\Delta\to\R^n$ of the embedding given by Theorem \ref{thm:Dufour-Molino}.}

Next, we recall that any manifold with corners embeds as a codimension 0 submanifold of a manifold without corners and the germ of such embeddings is unique. Similarly, for embeddings of integral affine manifolds with corners we have the following lemma, which proves the existence of $(M',\Lambda)$ and ensures the uniqueness of its germ.

\begin{lemma}
\label{lem:extend:manifolds:corners}
    If $(\Delta,\Lambda)$ is an integral affine manifold with corners:
        \begin{enumerate}[(i)]
        \item Given a codimension 0 embedding of $\Delta$ in a manifold without corners $M$, the integral affine structure $\Lambda$ extends to a neighborhood of $\Delta$.
        \item The germ of codimension 0 integral affine embeddings of $(\Delta,\Lambda)$ in integral affine manifolds without corners is unique. 
    \end{enumerate}
\end{lemma}

\begin{proof}[Proof of the Lemma \ref{lem:extend:manifolds:corners}]
    Any integral affine manifold with corners can be embedded as a codimension 0 integral affine submanifold of an integral affine manifold without corners. The proof is entirely similar to the proof given in \cite{WB59} of the fact that any real analytic manifold admits a complexification (see also \cite{DH73}). Fix such an embedding $\phi:(\Delta,\Lambda)\to (N,\Lambda_N)$ and let $\psi:\Delta\to M$ be any codimension 0 embedding in a manifold without corners. Using Whitney's embedding theorem we can assume that $N\subset \R^p$ and, by a partition of unity argument, that there is a smooth map $F:W\to \R^p$ extending the identity on $\Delta$, defined on some open neighborhood of $\Delta$ in $M$. Choosing a retraction $r:U\to N$ of some neighborhood of $N$ in $\R^p$, the composition $r\circ F:F^{-1}(U)\to N$ restricts to the identity on $\Delta$. Hence, possibly after shrinking $W$, we obtain a local diffeomorphism $r\circ F:W\to N$, that extends the identity on $\Delta$. Pulling back $\Lambda_N$ by this local diffeomorphism, we obtain an extension of $\Lambda$ from $\Delta$ to an integral affine structure on the neighborhood $W$. This proves item (i). 
    
    The proof of item (ii) is along the same lines as that of the analogous statement in \cite{WB59}.
\end{proof}

It remains to prove the converse statement in the theorem. Given a toric Hamiltonian $\T$-space $\mu:(S,\omega)\to M$ and a point $p_0\in S$, fixing an integral affine chart $(U,\phi)$ around $x_0:=\mu(p_0)$ we have that $\phi\circ\mu:(\mu^{-1}(U),\omega)\to \R^n$ is the moment map of an effective $\Tt^n$-action.
\end{proof}

It follows from the previous results that the standard local model for toric symplectic manifolds is also a local model for toric Hamiltonian $\T$-spaces. Let $\mu:(S,\omega)\to M$ be a toric $(\T,\Omega)$-space with $\Delta:=\mu(S)$ and let $x\in \Delta$. As we saw before, a choice of integral affine chart $(U,\phi=(x^1,\dots,x^n))$ centered at $x$ gives a trivialization of the symplectic torus bundle
\[
\Phi:(\T,\Omega)\vert_U\simeq (\mathbb{T}^n\times U,\sum_{i=1}^n\d\theta_i\wedge \d x^i). 
\]
Via this local isomorphism the toric $\T$-action corresponds to a toric $\Tt^n$-action with moment map 
\[ \phi\circ\mu:(\mu^{-1}(U),\omega)\to \R^n. \] 
Since $\mu$ has connected fibers and is proper onto its image, it follows from 
the standard normal form for toric manifolds (see, e.g., \cite[Lemma B.3]{KL15}) that, around $x$, there is a local isomorphism between this Hamiltonian $\mathbb{T}^n$-space and the standard local model. The latter is the Hamiltonian $\Tt^n$-space consisting of: 
\begin{itemize}
\item the symplectic manifold
\[ \left(S_{k,n},\omega_{k,n}\right):=(\C^k\times \mathbb{T}^{n-k}\times \R^{n-k},\sum_{j=1}^k\frac{1}{2\pi i}\d z_j\wedge \d \overline{z}_j+\sum_{j=1}^{n-k}\d\theta_j\wedge \d x^j);
\]
\item the moment map is given by
\[
    \mu_{k,n}:\C^k\times \mathbb{T}^{n-k}\times \R^{n-k}\to \R^n, \quad (z,\theta,x)\mapsto (|z_1|^2,\dots,|z_k|^2,x),
\] with image $\Delta=\R^n_k:=[0,\infty[^k\times \R^{n-k}$;
\item the $\mathbb{T}^n$-action given by
\begin{equation}\label{eqn:torusaction:localmodel}
    (\tau_1,\tau_2)\cdot (z,\theta,x):=(e^{2\pi i\tau_1} z,\tau_2+\theta,x), \quad (\tau_1,\tau_2)\in \mathbb{T}^k\times \mathbb{T}^{n-k}=\mathbb{T}^n.
\end{equation}
\end{itemize}
So, after possibly changing the integral affine chart $\phi$, we obtain an isomorphism $(\phi,\Phi,\Psi)$ from the Hamiltonian $\T\vert_U$-space $\mu:(\mu^{-1}(U),\omega)\to U$ to the Hamiltonian $\Tt^n$-space $\mu:(\mu_{k,n}^{-1}(V),\omega_{k,n})\to V$, where $V:=\phi(U)\subset \R^n$. Henceforth, we will refer to such an isomorphism, or to $(z,\theta,x)$, as \textbf{standard toric coordinates} for $\mu:(S,\omega)\to M$ around $x$. {Notice that the case $k=n$ corresponds to the fiber over $x$ being a fixed point of the $\Tt^n$-action, while $k=0$ corresponds to the $\Tt^n$-action on the fiber being free. For general $k$, the fiber over $x$ is a torus of dimension $n-k$.}


For future reference, we state the following result concerning invariant functions on the local model.

\begin{proposition}\label{prop:invsmoothfunctions:localmodel}
    Consider $S_{k,n}$ equipped with the $\mathbb{T}^n$-action (\ref{eqn:torusaction:localmodel}), and let $U$ be an open set in $\Delta$. For any $\mathbb{T}^n$-invariant smooth function  $f$ on $\mu^{-1}(U)$, the following hold:
    \begin{enumerate}[(a)]
        \item The unique function $\underline{f}$ on $U$ such that $f=\mu^*\underline{f}$ is smooth, i.e., it is the restriction of a smooth function defined on some open set in $\R^n$;
        \item If $f$ vanishes on $\mu^{-1}(U)\cap \{z_j=0\}$ for some $j\in \{1,\dots,k\}$, then it is of the form $f=|z_j|^2\cdot \widehat{f}$ for a unique $\mathbb{T}^n$-invariant smooth function $\widehat{f}$ on $\mu^{-1}(U)$.
    \end{enumerate}
\end{proposition}
    
\begin{proof}
Part (a) is well-known: it readily follows from Schwartz' theorem. Part (b) is immediate from part (a) and the fact that if $h\in C^\infty(U)$ vanishes on $U\cap \{x^j=0\}$, then $h=x^j\widehat{h}$, for some $\widehat{h}\in C^\infty(U)$.
\end{proof}

\begin{remark}
\label{rem:other:settings}
    In the literature singular Lagrangian fibrations admitting local effective $\Tt^n$-actions have been study by several authors \cite{Boucetta91,Yoshida11,Zung96,Zung03}. The result above shows that this form of symmetry can be encoded globally by toric Hamiltonian $\T$-actions. Moreover, the results to follow will show that this approach gives an effective, coordinate free way of studying the geometry of such fibrations.
\end{remark}

\subsection{Classification of toric Hamiltonian $\T$-spaces} 
Delzant's classification of compact symplectic toric manifolds in terms of polytopes (\cite{De88}) was extended to toric Hamiltonian $\T$-spaces in \cite{Maarten}. In this more general setup, Delzant's classification and Duistermaat's classification of Lagrangian fibrations (cf. Theorem \ref{thm:Duistermaat}) are combined, as we will now explain.

In Delzant's classification, the polytope corresponding to a symplectic toric manifold is the image of the moment map and lies in the integral affine manifold $(\R^n,\Z^n)$. These Delzant polytopes satisfy certain integrality conditions. For toric Hamiltonian $\T$-spaces, the image of $\mu$ is a more general object called a \textit{Delzant domain} of the integral affine manifold $(M,\Lambda)$, the base of the symplectic torus bundle $(\T,\Omega)$.

\begin{definition} 
A \textbf{Delzant domain} $\Delta$ of an $m$-dimensional integral affine manifold $(M,\Lambda)$ is a subset of $M$ with the property that for each $x\in \Delta$ there is an integral affine chart $(U,\phi)$ for $(M,\Lambda)$ centered at $x$ such that $\phi(U\cap \Delta)$ is an open set in $\R^n_k$ for some $k\in \{0,\dots,n\}$ (depending on $x$). 
\end{definition}

In other words, a Delzant domain of $(M,\Lambda)$ is a codimension zero integral affine submanifold with corners. It follows from the definition that each open face of a Delzant domain inherits an integral affine structure. 


\begin{example}[Compact symplectic toric manifolds]
For any compact symplectic toric manifold $\mu_0:(S,\omega)\to \R^n$, the image of the moment map is a Delzant polytope in $\R^n$. Delzant polytopes are precisely the compact and connected Delzant domains of $(\R^n,\Z^n)$. {Notice that this does not require the vertices to be integral points. Moreover, there is no notion of ``integral point" in a general integral affine manifold.}

\end{example}

\begin{example}[Non-compact symplectic toric manifolds]
\label{ex:non-compact:toric}
Non-compact symplectic toric manifolds $\mu_0:(S,\omega)\to \R^n$ also fit into our framework. It was observed in \cite{KL15} that the orbit space $S/\Tt^n$ is a manifold with corners and the induced map
\[ \phi:S/\Tt^n\to \R^n \]
is a local embedding. The pullback along $\phi$ of the standard integral affine structure in $\R^n$ makes $\Delta:=S/\Tt^n$ an integral affine manifold with corners. As in the proof of Lemma \ref{lem:extend:manifolds:corners}, $\Delta$ embeds as a Delzant domain of an integral affine manifold $(M,\Lambda)$ without corners. Note that the composition of the quotient map with this embedding
\[ \mu: S\to \Delta\hookrightarrow M, \]
yields a toric Hamiltonian $\T_\Lambda$-space $\mu:(S,\omega)\to M$. When $\phi$ is an embedding we can take $(M,\Lambda)=(\R^n,\Z^n)$. 
\end{example}

\begin{example}
The Hamiltonian $\T$-spaces in examples \ref{ex:trivspherebun:stdIA:cylinder} and \ref{ex:nontrivspherebun:nonstdIA:cylinder} are toric and the image of their moment maps are Delzant spaces. In fact:
\begin{itemize}
    \item For the standard integral affine cylinder $\R\times \Ss^1$, appearing in Example \ref{ex:trivspherebun:stdIA:cylinder}, $\Delta=[a,b]\times \Ss^1$ is a Delzant domain, for any  $a<b$. 
    \item Similarly, for the exotic integral affine cylinder $]-2,\infty[\times \Ss^1$ of Example \ref{ex:nontrivspherebun:nonstdIA:cylinder}, $\Delta=[a,b]\times \Ss^1$ is a Delzant domain, for any $-2<a<b$.  
\end{itemize}
\end{example}

A version of the classification theorem for toric Hamiltonian $\T$-spaces can be stated as follows. For a complete proof and details we refer to \cite{Maarten}.

\begin{theorem}[\cite{Maarten}]
\label{thm:toric:classification}
Let $(\T,\Omega)\to (M,\Lambda)$ be a symplectic torus bundle. For any toric Hamiltonian $\T$-space $\mu:(S,\omega)\to M$, the image $\Delta:=\mu(S)$ is a Delzant domain of $(M,\Lambda)$. Moreover, such Hamiltonian $\T$-spaces are classified by their:
\begin{enumerate}[(i)]
    \item Delzant domain $\Delta:=\mu(S)$, and
    \item Lagrangian Chern class $c_1(S,\omega)\in \check{H}^1(\Delta,\TLag)$.
\end{enumerate}
\end{theorem}

Here $\TLag$ denotes the sheaf of Lagrangian sections of $\T|_\Delta$ and the Lagrangian Chern class is defined similarly to the regular case (see \cite{Maarten} and the proof of Theorem \ref{thm:real:Lag-Chern:class:2}). For a Delzant polytope the cohomology group in item (ii) always vanishes, because the polytope is contractible. 

\begin{remark} As for classical toric symplectic manifolds \cite{De88}, the map 
\begin{equation}
\label{eqn:transversemommap}
\underline{\mu}:S/\T\to M
\end{equation}
is a homeomorphism onto $\Delta:=\mu(S)$. Moreover, it respects the natural stratifications on $S/\T$ and $\Delta$, namely:
\begin{itemize}
\item $\Delta$ is naturally stratified, being a manifold with corners;
\item $S/\T$ is naturally stratified, being the orbit space of a proper Lie groupoid action \cite{PflaumPosthumaTang}.
\end{itemize} 
We refer to \cite{Maarten} for details.
\end{remark} 

\begin{remark}
    The classification result above, beside generalizing Duistermaat and Delzant classification results, is also closely related to classification results in various other settings, as in Remark \ref{rem:other:settings}. For example, the classification of non-compact symplectic toric manifolds due to Karshon and Lerman \cite{KL15} follows from this theorem.
\end{remark}

\subsection{Classical symplectic toric manifolds}
\label{sec:classical:toric}

Motivated by Example \ref{ex:non-compact:toric} we introduce the following definition.

\begin{definition}
    A toric Hamiltonian $\T$-space $\mu:(S^{2n},\omega)\to M$ is called a \textbf{classical symplectic toric manifold} if there exists an effective Hamiltonian $\Tt^n$-action on $(S,\omega)$ with moment map $\mu_0:S\to \R^n$ and a local embedding $\phi:\Delta\to \R^n$, with $\Delta:=\mu(S)$, such that
    \[ \mu_0=\phi\circ\mu. \]
\end{definition}

The toric Hamiltonian $\T$-spaces that are classical symplectic toric manifolds can be easily characterized in terms of the integral affine geometry of their Delzant domains.

\begin{theorem}
    \label{thm:classical:toric:spaces}
    A toric $\T$-space $\mu:(S^{2n},\omega)\to M$ is a classical symplectic toric manifold if and only if the associated Delzant space $\Delta=\mu(S)$ has trivial affine holonomy.
\end{theorem}

\begin{proof}
    The Delzant space $\Delta=\mu(S)$ has trivial affine holonomy if and only if the developing map $\dev_{x_0}:\widetilde{\Delta}\to \R^n$  factors through the covering map to an integral affine local embedding $\phi:\Delta\to \R^n$ (see, e.g., \cite[Section 2.4]{HSSS18}).

    On the other hand, a map $\mu_0:S\to \R^n$ which factors through a local embedding $\phi:\Delta\to \R^n$ such that
    \[ \mu_0=\phi\circ\mu, \]
    is the moment map of an effective Hamiltonian $\Tt^n$-action if and only if $\phi$ is integral affine. Hence, the result follows.
\end{proof}

Recall that a compact toric symplectic manifold is always 1-connected. In this respect we also have the following result.

\begin{corollary}
    A toric $\T$-space $\mu:(S,\omega)\to M$ whose associated Delzant space $\Delta=\mu(S)$ is 1-connected is a classical symplectic toric manifold. In particular this holds whenever $S$ is 1-connected.
\end{corollary}

\begin{proof}
    The first part of the statement follows immediately from the theorem. For the second part, consider the inclusions
     \[ \xymatrix{ S \ar[r]^{\mu} & \Delta\\
    \mathring{S} \ar@{^{(}->}[u] \ar[r]_{\mu|_{\mathring{S}}} & \mathring{\Delta}\ \ar@{^{(}->}[u]}\]
    The map $\mu:\mathring{S}\to \mathring{\Delta}$ is a principal $\Tt^n$-bundle while $\mathring{\Delta}$ and $\Delta$ have the same homotopy type. Therefore, passing to fundamental groups we obtain a commutative diagram
    \[ 
    \xymatrix{
    \pi_1(S)\ar[r]^{\mu_*}& \pi_1(\Delta)\\
    \pi_1(\mathring{S})\ar[u]\ar[r]& 
    \pi_1(\mathring{\Delta})\ar[u]_{\simeq}
    }
    \]
    where the bottom horizontal arrow is surjective. It follows that $\mu_*:\pi_1(S)\to \pi_1(\Delta)$
    is also surjective. Hence, if $S$ is 1-connected, so is $\Delta$.
\end{proof}

\begin{corollary}
    A toric $\T$-space $\mu:(S,\omega)\to M$ with $S$ compact is a classical symplectic toric manifold if and only if $S$ is 1-connected.
\end{corollary}

Note that non-compact symplectic toric manifolds may fail to be 1-connected (e.g., $S=T^*\Tt^n$). Hence, the last corollary does not hold if one removes the compactness assumption.

\newpage
\subsection{The elliptic and $b$-tangent bundle}

\subsubsection{Background on the $b$-tangent bundle}
\label{sec:btangent}
Let $\Delta$ be a manifold with corners. Recall that a \textbf{$b$-vector field} in $\Delta$ is a smooth vector field that is tangent to all open faces (or strata) of $\Delta$ \cite{Mel93}. The sheaf of $b$-vector fields, denoted $\tensor[^b]{\X}{}:=\tensor[^b]{\X}{_\Delta}$ is a locally free sheaf of $\mathcal{C}^\infty$-modules, with local bases 
\begin{equation}
\label{eqn:basis:btngtbun} 
x^1{\partial}_{x^1},\dots,x^k{\partial}_{x^k},{\partial}_{x^{k+1}},\dots,{\partial}_{x^n},
\end{equation} 
for any coordinates $(U,x^i)$ that map onto an open set in $\R^n_k$. Hence, the vector spaces
\[
    \tensor[^b]{T_x\Delta}{}:=\frac{\tensor[^b]{\X}{_x}}
    {\mathfrak{m}_x\cdot \tensor[^b]{\X}{_x}} 
    \quad (x\in \Delta),
\] 
fit together into a rank $n$ vector bundle $\tensor[^b]{T\Delta}{}$ -- the \textbf{$b$-tangent bundle} or \textbf{log tangent bundle} of $\Delta$ --  with smooth local frames given by (\ref{eqn:basis:btngtbun}). Here $\tensor[^b]{\X}{_x}$ denotes the stalk of $\tensor[^b]{\X}{}$ at $x$ and $\mathfrak{m}_x$ is the ideal in $\mathcal{C}^\infty_x$ of germs of functions that vanish at $x$. The map
\begin{equation}
\label{eqn:anchor:btngntbun} 
\rho:\tensor[^b]{T\Delta}{}\to T\Delta
\end{equation} 
induced by the evaluation maps $\X_x\to T_x\Delta$, identifies the sheaf of smooth sections of $\tensor[^b]{T\Delta}{}$ with $\tensor[^b]{\X}{_\Delta}$. The $b$-tangent bundle has the structure of a Lie algebroid over $\Delta$, with bracket the usual Lie bracket of vector fields and anchor map (\ref{eqn:anchor:btngntbun}). 

The smooth sections of the \textbf{$b$-cotangent bundle}
\[ 
\tensor[^b]{T^*\Delta}{}:=(\tensor[^b]{T\Delta}{})^*
\] 
are determined by their restriction to the interior $\mathring{\Delta}$, which are just ordinary
smooth $1$-forms. In view of this, sections of the dual bundle are often simply denoted by their underlying $1$-forms on $\mathring{\Delta}$. The local frame of $\tensor[^b]{T^*\Delta}{}$ dual to (\ref{eqn:basis:btngtbun}) consists of those smooth sections that restrict to the $1$-forms
\[
    \frac{1}{x^1}\d x^1,\dots,\frac{1}{x^k}\d x^k,\d x^{k+1},\dots,\d x^n, 
\] on $\mathring{\Delta}$. In summary, smooth sections of the $b$-cotangent bundle can be represented by smooth $1$-forms on $\mathring{\Delta}$ with poles of order at most $1$ at the boundary.

Next, we turn to residues. The open faces of $\Delta$ are the leaves of the Lie algebroid $\tensor[^b]{T\Delta}{}$. Accordingly, for each open face $F$ of $\Delta$ we have a short exact sequence
\[ 
\xymatrix{ 0\ar[r] & \ker(\rho)\vert_F\ar[r] &\tensor[^b]{T^*\Delta}{}\vert_F\ar[r]^{\rho} & TF\ar[r] & 0.}
\] 
When $F$ is a facet, i.e., a codimension $1$ face in $\Delta$, $\ker(\rho)\vert_F$ is a line bundle that admits a canonical non-vanishing section. 

\begin{proposition}
\label{prop:eulersection:manwithcorners}
Let $F$ be an open facet of a manifold with corners $\Delta$. There is a unique non-vanishing section
\[
    \mathcal{E}_F\in \Gamma^\infty(\ker(\rho)\vert_F)
    \] 
    with the property that any $X\in \tensor[^b]{\X}{}(U)$ that is Euler-like with respect to $F$ extends $\mathcal{E}_F\vert_{U\cap F}$. 
\end{proposition} 

Here by an \emph{Euler-like vector field} $X\in\X(\Delta)$ relative to the submanifold $F\subset \Delta$ we mean a vector field such that $X\vert_{F}=0$ and the linearization of $X$ along $F$ induces the identity map on the normal bundle to $F$ in $\Delta$. This proposition is well-known (see, e.g., \cite{CaGu,GuiMirPires}) so we omit the proof.
 
\begin{remark}
\label{rem:eulerlikevf:manwithcorncoord}
Let $(U,x^1,\dots,x^n)$ be coordinates for $\Delta$ onto an open set in $\R^n_k$. Writing a given $b$-vector field $X\in \tensor[^b]{\X}{}(U)$ as
\[ 
X=\sum_{i=1}^k f_i\,x^i\, {\partial}_{x^i} + \sum_{i=k+1}^n f_i\,{\partial}_{x^i}, \quad f_i\in C^\infty_\Delta(U),
\] 
then $X$ is Euler-like with respect to $F$ if and only if for each $x\in U\cap F$
\[ f_i(x)=\begin{cases}1 \quad&\textrm{ if }x^i=0 \text{ and }i\leq k,\\
0 \quad &\textrm{ otherwise}.
\end{cases}
\] In particular, if $k=1$, then $x^1\partial_{x^1}\in \tensor[^b]{\X}{}(U)$ is an extension of $\mathcal{E}_F\vert_{U\cap F}$.
\end{remark}

Using the previous proposition the residue of a vector-valued 1-form can be defined as follows.

\begin{definition}
\label{def:residue:form}
   Let $\Delta$ be a manifold with corners, $F$ an open facet of $\Delta$ and $E\to\Delta$ a vector bundle. The \textbf{residue along $F$} of a vector-valued $1$-form $\alpha\in \Omega^1(\tensor[^b]{T\Delta}{},E)$ is the section 
   \[
   \mathrm{res}_F(\alpha):=\alpha(\mathcal{E}_F)\in \Gamma^\infty(E|_F).
   \]
\end{definition}

\subsubsection{Background on the elliptic tangent bundle}

The elliptic tangent bundle $\tensor[^\mu]{TS}{}$ of a toric $(\T,\Omega)$-space $\mu:(S,\omega)\to M$ is a Lie algebroid over $S$, similar to the $b$-tangent bundle introduced in the previous subsection. This Lie algebroid was first introduced in \cite{CaGu} in the case of a single elliptic divisor and then generalized for several divisors, including the case of a classical symplectic toric action, in \cite{CKW22,CKW23}.  In our case, we can define the elliptic tangent bundle more directly in terms of its sheaf of sections as follows. 

Recall that $\Delta:=\mu(S)$ is a submanifold with corners of $M$, so that it comes with a natural sheaf of smooth functions $\mathcal{C}^\infty_\Delta$. Consider the sheaf $\mu_*(\X_S)^\T$ on $\Delta$ that assigns to an open set $V$ the vector space of $\T$-invariant vector fields on the invariant open set $\mu^{-1}(V)$ in $S$. This is naturally a sheaf of $\mathcal{C}^\infty_\Delta$-modules, which is locally free of rank $2n$, where $n=\dim(\Delta)=\frac{1}{2}\dim(S)$. Indeed, in the standard toric coordinates  $(u_j+iv_j,\theta_l,x^l)$ around a point $x\in\Delta$ of depth $k$, i.e, in a open face of codimension $k$, as a consequence of Proposition \ref{prop:invsmoothfunctions:localmodel} a basis for $\mu_*(\X_S)^\T$ around $x$ is given by the vector fields
\begin{equation}\label{eqn:basis:elltngtbun}
    \partial_{\phi_1},\dots,\partial_{\phi_k},r_1\partial_{r_1},\dots,r_k\partial_{r_k},\dots,\partial_{\theta_1},\dots,\partial_{\theta_{n-k}},\partial_{x^1},\dots,\partial_{x^{n-k}}.
\end{equation} 
Here our convention for the angular and radial vector fields on $\mathbb{C}^k$ is so that
\begin{align*}
    \tfrac{1}{2\pi}\partial_{\phi_j}&=u_j\partial_{v_j}-v_j\partial_{u_j}, \\  r_j\partial_{r_j}&=u_j\partial_{u_j}+v_j\partial_{v_j}.
\end{align*}
The sheaf of $\mathcal{C}^\infty_\Delta$-modules $\mu_*(\X_S)^\T$ pulls back along $\mu$ to a sheaf $\tensor[^\mu]{\X}{}$ of $\mathcal{C}^\infty_S$-modules, which is locally free of the same rank. More concretely, $\tensor[^\mu]{\X}{}$ can be viewed as the subsheaf of $\X_S$ consisting of those vector fields on $S$ locally generated by $\T$-invariant vector fields. We will call $\tensor[^\mu]{\X}{}$ the \textbf{sheaf of elliptic vector fields} on $S$ relative to $\mu$. The \textbf{elliptic tangent bundle} $\tensor[^\mu]{TS}{}$ is the corresponding vector bundle over $S$. Like the $b$-tangent bundle, $\tensor[^\mu]{TS}{}$ has a natural Lie algebroid structure. The anchor is the canonical map
\[
    \rho: \tensor[^\mu]{TS}{}\to TS,
\] 
which identifies its sheaf of sections with $\tensor[^\mu]{\X}{}$, and the Lie bracket is the usual Lie bracket of vector fields.  By construction, the vector fields (\ref{eqn:basis:elltngtbun}) form a smooth local frame of $\tensor[^\mu]{TS}{}$. The elliptic tangent bundle allows to resolve the singularities of the differential of $\mu$ and of the infinitesimal action $\act$ (the ranks of which drop at points outside of $\mathring{S}$), in the following sense.

\begin{proposition}
\label{prop:elliptictngbun:liftmomentummap} 
The differential of $\mu$ lifts uniquely to a fiberwise surjective Lie algebroid map $\mu_*$
\[
\xymatrix{
\tensor[^\mu]{TS}{}\ar@{-->}[r]^{\mu_*}\ar[d]_{\rho} & \tensor[^b]{T\Delta}{}\ar[d]^{\rho}\\
TS\ar[r]_{\d\mu} & T\Delta}
\]
Moreover, the infinitesimal action lifts uniquely to a vector bundle map $\act_*$
\[
\xymatrix{
 & \tensor[^\mu]{TS}{} \ar[d]^{\rho}\\
\mu^*(T^*M)\ar@{-->}[ru]^{\act_*}\ar[r]_{\text{ }\text{ }\text{ }\act}  & TS}
\] 
that is fiberwise injective and has image $\ker(\mu_*)$.
\end{proposition}
\begin{proof}
    The first statement follows from the remark that any invariant vector field on $S$ is $\mu$-related to a unique $b$-vector field on $\Delta$, which is readily verified using the bases (\ref{eqn:basis:elltngtbun}). The second statement follows from the fact that $\act(\alpha)$ is an invariant vector field tangent to $\ker(\d \mu)$ for every $1$-form $\alpha$.  
\end{proof}
Finally, we recall the notion of radial residue. As for classical symplectic toric manifolds, for a toric $(\T,\Omega)$-space $\mu:(S,\omega)\to M$ the stratification of $S$ induced by the $\T$-action consists of the submanifolds of $S$ of the form $S_F:=\mu^{-1}(F)$, with $F$ an open face of $\Delta:=\mu(S)$. These strata are the leaves of the Lie algebroid $\tensor[^\mu]{TS}{}$. In particular, for any open face $F$ of $\Delta$, we have a short exact sequence
\[
\xymatrix{ 0\ar[r] & \ker(\rho)\vert_{S_F}\ar[r] & \tensor[^\mu]{TS}{}\vert_{S_F}\ar[r]^{\rho} & TS_F\ar[r] & 0.}
\] 
If $F$ is an open facet of $\Delta$, the toric local model shows that $\ker(\rho)\vert_{S_F}$ is a rank $2$ vector bundle that admits a canonical non-vanishing section, which can characterized as follows.
\begin{proposition} Let $\mu:(S,\omega)\to M$ be a toric $(\T,\Omega)$-space and let $F$ be an open facet of $\Delta:=\mu(S)$. There is a unique non-vanishing section
\[ 
\mathcal{E}_{S_F}\in \Gamma^\infty(\ker(\rho)\vert_{S_F})
    \] 
    such that any $\T$-invariant vector field on an invariant open set $U$ in $S$ that is Euler-like with respect to $S_F$ extends $\mathcal{E}_{S_F}\vert_{U\cap S_F}$. Moreover, $\mathcal{E}_{S_F}$ is $\mu_*$-related to $2\mathcal{E}_{F}$.
\end{proposition}
\begin{remark}
\label{rem:eulerlikevf:toriccoord}
In standard toric coordinates, as in \eqref{eqn:basis:elltngtbun}, a $\T$-invariant vector field $X$ on $U$ can be written as
\[
    X=\sum_{i=1}^kf^r_ir_i\partial_{r_i}+f^\phi_i\partial_{\phi_i}+\sum_{j=1}^{n-k}f^\theta_j\partial_{\theta_j}+f^x_j{\partial}_{x^j},
\] 
for some invariant functions $f^r_i,f^\phi_i,f^\theta_j,f^x_j\in \mathcal{C}^\infty(U)^{\T}$. Then $X$ is Euler-like with respect to $S_F$ if and only if for each $p=(z,t,x)\in U\cap S_F$ one has
\[
   f_i^\phi(p)=0,\quad f_j^\theta(p)=0, \quad f_j^x(p)=0\quad\text{and}\quad f_i^r(p)=\begin{cases}1 \quad&\textrm{ if } z_i=0,\\
0 \quad &\textrm{ otherwise},
\end{cases}
\] 
for all $i\in \{1,\dots,k\}$ and $j\in \{1,\dots,n-k\}$. In particular, if $k=1$, then $r\partial_{r}$ is an extension of $\mathcal{E}_{S_F}\vert_{U\cap S_F}$.
\end{remark}

Similar to the definition of residues for vector-valued $b$-forms, we can define residues of vector-valued elliptic 1-forms as follows.

\begin{definition} Let $\mu:(S,\omega)\to M$ be a toric Hamiltonian $\T$-space, $F$ an open facet of $\Delta:=\mu(S)$ and $E$ a vector bundle over $\Delta$. The \textbf{radial residue along $F$} of a vector-valued $1$-form $\alpha\in \Omega^1(\tensor[^\mu]{TS}{},E)$ is the section
\[
\mathrm{res}_{F,\mathrm{Rad}}(\alpha):=\alpha(\mathcal{E}_{S_F})\in \Gamma^\infty(E|_{S_F}).
\]
\end{definition}
\subsection{Elliptic Lagrangian connections}\label{sec:elliptic:connections}
We will now show that Theorem \ref{thm:real:Lag-Chern:class} extends to toric Hamiltonian $\T$-spaces. For that, we need an appropriate notion of connection $1$-form on a toric $\T$-space $\mu:(S,\omega)\to M$. First, notice that there is a canonical fiberwise linear $\T$-action on $\tensor[^\mu]{TS}{}$ along the composition of the bundle projection with $\mu$, because $\tensor[^\mu]{TS}{}$ is the pull-back along $\mu$ of a vector bundle on $\Delta:=\mu(S)$. For any $p\in S$ the action by an element $g\in\T_{\mu(p)}$ is the canonical isomorphism $\tensor[^\mu]{T_{p}S}{}\to \tensor[^\mu]{T_{g\cdot p}S}{}$. The anchor map of the elliptic tangent bundle is $\T$-equivariant with respect to this action and the tangent action of $\T$. 


\begin{definition}
    An \textbf{elliptic connection $1$-form} for a toric Hamiltonian $\T$-space $\mu:(S,\omega)\to M$ is a Lie algebroid form $\theta\in \Omega^1(\tensor[^\mu]{TS}{},T^*M)$ which satisfies:
    \begin{itemize}
\item[(i)] $\T$-invariance: $\theta_{gp}(g\cdot v)=\theta_p(v)$, for all $v\in \tensor[^\mu]{T_pS}{}$, $g\in\T_{\mu(p)}$;
\item[(ii)] $\theta_p((\act_*)_p(\alpha))=\alpha$, for all $\alpha\in T_{\mu(p)}^*M$ and $p\in S$.
\end{itemize} We call such a connection $1$-form \textbf{Lagrangian} if $\ker(\theta)_p\subset (T_pS,\omega_p)$ is Lagrangian for every $p\in \mathring{S}:=\mu^{-1}(\mathring{\Delta})$. 
\end{definition}

\begin{remark}
   In the case of torus actions the notion of elliptic connection already appears in \cite{CW22}. 
\end{remark}

\begin{example}\label{example:ellconnection:stdlocmodel}
An elliptic Lagrangian connection for the standard toric local model $\mu_{k,n}:(S_{k,n},\omega_{k,n})\to \R^n$ is given by 
\[
\theta(r_i\partial_{r_i})=0,\quad \theta(\partial_{\phi_i})=\d x^i,\quad \theta(\partial_{x^i})=0, \quad \theta(\partial_{\theta_i})=\d x^{i+k}.
\] 
Note that this has zero radial residue along all open facets of $\R^n_k$. 
\end{example}

\begin{proposition} 
\label{prop:existence:conn1form:toric}
Every toric Hamiltonian $\T$-space admits an elliptic Lagrangian connection $1$-form with zero radial residue along all open facets. 
\end{proposition}

\begin{proof} 
In view of the above example such connection $1$-forms exist locally. A global such $1$-form can be constructed from these using a partition of unity for the manifold with corners $\Delta:=\mu(S)$, as in the proof of Proposition \ref{prop:existenceLagrconn:principalcase}.
\end{proof}

Next, we turn to the curvature of such connection $1$-forms. Let $\mu:(S,\omega)\to M$ be a toric Hamiltonian $\T$-space. There is a flat $\tensor[^\mu]{TS}{}$-connection on the bundle $\mu^*T^*M$ obtained by first pulling-back the canonical flat connection $\nabla$ on $T^*M$ along $\mu$ to a $TS$-connection on $\mu^*T^*M$ and then composing with the anchor of $\tensor[^\mu]{TS}{}$. This gives a differential 
\[
\d^\nabla:\Omega^\bullet(\tensor[^\mu]{TS}{},T^*M)\to \Omega^{\bullet+1}(\tensor[^\mu]{TS}{},T^*M).
\] 
It follows from Proposition \ref{prop:invsmoothfunctions:localmodel}, the standard local normal form and the properties of a connection 1-form $\theta$ that there is a  unique $2$-form 
\[
K_\theta\in \Omega^2(\tensor[^b]{T\Delta}{},T^*\Delta)
\] such that
\[
\mu^*K_\theta=\d^\nabla\theta \in \Omega^2(\tensor[^\mu]{TS}{},T^*M).
\]
We call $K_\theta$ the \textbf{curvature of the elliptic connection} $\theta$.

Notice that $\nabla$ also induces a flat $\tensor[^b]{T\Delta}{}$-connection on $T^*\Delta$, with corresponding differential
\[
\d^\nabla:\Omega^\bullet(\tensor[^b]{T\Delta}{},T^*\Delta)\to \Omega^{\bullet+1}(\tensor[^b]{T\Delta}{},T^*\Delta).
\]
The curvature $K_\theta$ is $\d^\nabla$-closed and its cohomology class in $H^2(\tensor[^b]{T\Delta}{},T^*\Delta)$ does not depend on the choice of $\theta$. If $\theta$ has zero radial residues, then by the proposition below the curvature is an honest smooth $2$-form $K_\theta\in \Omega^2(\Delta,T^*\Delta)$. So, in view of Proposition \ref{prop:existence:conn1form:toric}, the connection-curvature construction actually defines a cohomology class $H^2(\Delta,T^*\Delta)$. In the case of torus actions this was also observed in \cite{CW22}.

\begin{proposition}
\label{prop:curvature2}
    Let $\theta\in \Omega^1(\tensor[^\mu]{TS}{},T^*M)$ be an elliptic connection $1$-form for a toric Hamiltonian $\T$-space $\mu:(S,\omega)\to M$ and suppose that $\theta$ has zero radial residues. Then $\d^\nabla\theta\in \Omega^2(S,T^*M)$ and $K_{\theta}\in \Omega^2(\Delta,T^*\Delta)$.
    Moreover, if $\theta$ is Lagrangian, then $K_{\theta}\in \Omega^2_\partial(\Delta,T^*\Delta)$, meaning that
    \[\partial K_{\theta}(v_1,v_2,v_3):=\sum_{\sigma\in S_3}(-1)^{|\sigma|}\langle K _{\theta}(v_{\sigma(1)},v_{\sigma(2)}),v_{\sigma(3)}\rangle=0. \]
\end{proposition}

\begin{proof} 
Note that $\iota_{\act(\alpha)}(\d^\nabla\theta)=0$ for all $1$-forms $\alpha$, since $\theta$ is a connection $1$-form. Also $\iota_{\mathcal{E}_{S_F}}(\d^\nabla\theta)=0$ for each open facet $F$, since $\theta$ has zero radial residues. Hence, $\d^\nabla\theta$ is a smooth $2$-form on $S$. Since $\mu_*(\mathcal{E}_{S_F})=2\,\mathcal{E}_F$, it follows that $K_\theta$ is a smooth $2$-form as well. Finally, $\partial K_\theta=0$ if $\theta$ is Lagrangian, since this holds on $\mathring{\Delta}$ by Proposition \ref{prop:curvature}, so it holds on all of $\Delta$ by density of $\mathring{\Delta}$.
\end{proof}

This proposition shows that a choice of elliptic Lagrangian connection with zero radial residue gives rise to a cohomology class
\[
[K_\theta]\in H^2(\Omega_\partial^\bullet(\Delta,T^*\Delta),\d^\nabla), 
\] 
with the complex $(\Omega_\partial^\bullet(\Delta,T^*\Delta),\d^\nabla)$ defined like in Section \ref{sec:Lag:connections}. In order to extend Theorem \ref{thm:real:Lag-Chern:class}, consider the composite
 \begin{equation}  \label{eq:map:Lag:chern:class:2}
        \xymatrix@C=15pt{
        \check{H}^1(\Delta,\TLag)\ar[r]^\simeq& \check{H}^2(\Delta,\O_\Lambda)\ar[r] & \check{H}^2(\Delta,\O_{\Aff}) \ar[r]^---\simeq &  H^2(\Omega^\bullet_\partial(\Delta,T^*\Delta),\d^\nabla)},
    \end{equation} 
where $\O_{\Aff}$ and $\O_\Lambda$ denote respectively the sheaves of locally defined functions on $\Delta$ that are affine and integral affine with respect to $\Lambda$. The first map in this sequence is the connecting homomorphism induced by the short exact sequence
    \[ 
    \xymatrix{
    0\ar[r] & \O_\Lambda\ar[r] &  \mathcal{C}^\infty\ar[r]^--{\d}& \TLag\ar[r] & 0,}
    \]
    while the last map is the isomorphism induced by the fine resolution
    \[
    \xymatrix{
    0\ar[r] & \O_{\Aff}\ar[r] &\mathcal{C}^\infty\ar[r]^---{\d^\nabla\circ \d} & \Omega^1_\partial(-,T^*\Delta)\ar[r]^{\d^\nabla}& \Omega^2_\partial(-,T^*\Delta)\ar[r]^--{\d^\nabla}& \dots}
    \]
\begin{theorem}
    \label{thm:real:Lag-Chern:class:2} 
   Let $\mu:(S,\omega)\to M$ be a toric Hamiltonian $\T$-space. If $\theta$ is an elliptic Lagrangian connection for $\mu$ with zero radial residue, then \eqref{eq:map:Lag:chern:class:2} maps the Lagrangian Chern class $c_1(S,\omega)$ to the class $[K_\theta]$. 
\end{theorem}

For the proof it will be useful to recall from \cite{Maarten} how, given a Delzant domain $\Delta$ of an integral affine manifold $(M,\Lambda)$, one constructs the \textbf{canonical toric $\T$-space} $\mu_\Delta:(S_\Delta,\omega_\Delta)\to M$ with momentum map image $\Delta$ and $c_1(S_\Delta,\omega_\Delta)=0$. This construction consists of the following steps:
\begin{enumerate}[(i)]
\item As a topological space, $S_\Delta$ is the quotient of $\T_\Lambda\vert_\Delta$ by the normal subgroupoid with isotropy group at $x\in \Delta$ given by the subtorus with Lie algebra $(T_xF)^0$, the annihilator in  $T_x^*M$ of the tangent space to the open face $F$ through $x$. 
\item The bundle projection of $\T_\Lambda$ descends to a continuous map $\mu_\Delta:S_\Delta\to M$ and the action of $\T_\Lambda$ along its bundle projection descends to a continuous action of $\T_\Lambda$ along $\mu_\Delta$. 
\item There are a unique smooth and symplectic structures on $S_\Delta$ with the following property: for any integral affine chart $(U,\phi=(x^1,...,x^n))$ onto a connected open set $\phi(U)$ around the origin in $\R^n$ such that $\phi(U\cap\Delta)=\phi(U)\cap \R^n_k$, the induced homeomorphism
\begin{align}
\label{eqn:toriccoord:S_Delta}
&\xymatrix{
\mu_\Delta^{-1}(U)\ar[r]^-{\sim} & \mu_{k,n}^{-1}(\phi(U)) 
}\\
\Big[\sum_{j=1}^n \theta_j\d x^j\Big]& \mapsto \left(e^{2\pi i\theta_1}\sqrt{x^1},\dots,e^{2\pi i\theta_k}\sqrt{x^k},\theta_{k+1},\dots, \theta_n,x^{k+1},\dots,x^n\right) \notag
\end{align}
is a symplectomorphism with respect to $\omega_{k,n}$. 
\end{enumerate}
All together, this gives a toric $\T_\Lambda$-space
\[ 
\mu_\Delta:(S_\Delta,\omega_\Delta)\to M
\] 
with moment map image $\Delta$ and trivial Lagrangian Chern class (by definition). Note that this comes with a canonical symplectic open embedding
\[
\T_\Lambda\vert_{\mathring{\Delta}}\hookrightarrow S_\Delta
\]
onto the dense subset $\mathring{S}_\Delta=\mu_\Delta^{-1}(\mathring{\Delta})$. 
\begin{proof}[Proof of Theorem \ref{thm:real:Lag-Chern:class:2}] Consider $\mathcal{U}=\{U_i\}_{i\in I}$ a good open cover of $\Delta$ together with equivalences of Hamiltonian $\T$-spaces 
\[
\xymatrix{
(\mu^{-1}(U_i),\omega)\ar[rr]_{\sim}^{\Psi_i}\ar[dr]_{\mu} & & (\mu_\Delta^{-1}(U_i),\omega_\Delta) \ar[dl]^{\mu_\Delta}\\
& M}
\] 
For any two $i,j\in I$ there is a unique Lagrangian section $\tau_{ij}\in \TLag(U_{ij})$ such that
\[
\tau_{ij}(\mu(p))\cdot \Psi_j(p)=\Psi_i(p),\quad p\in \mu^{-1}(U_{ij}).
\] 
The class $c_1(S,\omega)$ is represented by the \v{C}ech $1$-cocycle $\tau=[\tau_{ij}]\in \check{C}^1(\mathcal{U},\TLag)$. The zero-section of $\T_\Lambda$ induces a continuous section $\sigma_\Delta:\Delta\to S_\Delta$ of $\mu_\Delta$, which is smooth and Lagrangian on $\mathring{\Delta}$. We set 
\begin{equation}
\label{eq:sections:singular}
    \sigma_i:=\Psi_i^{-1}\circ \sigma_\Delta:U_i\to S.
\end{equation}
The fact that $c_1(S,\omega)$ is mapped to $[K_\theta]$ by \eqref{eq:map:Lag:chern:class:2} follows from an argument identical to that in the proof of  Theorem \ref{thm:real:Lag-Chern:class}, where the Lagrangian sections $\sigma_i$ are now given by \eqref{eq:sections:singular}. There is however one technical issue to overcome: now the sections $\sigma_i$ are not smooth on all of $U_i$, only on  $\mathring{U}_i=U_i\cap\mathring{\Delta}$, so \emph{a priori} we can only consider 
\[ \sigma_i^*\theta\in \Omega_\partial^1(\mathring{U}_i,T^*\Delta).\] 
We need to show that the $1$-forms $\sigma_i^*\theta$ extend smoothly to all of $U_i$ for each $i\in I$. Since this is a local property, it can be checked in the standard toric coordinates for $\mu_\Delta$ induced by integral affine coordinates $(U,x^i)$ as in \eqref{eqn:toriccoord:S_Delta}. In such toric coordinates $(\Psi_i^{-1})^*\theta$ is of the form
\[ 
(\Psi_i^{-1})^*\theta=\sum_{j=1}^k\d\phi_j\otimes \alpha_j+\mu_{k,n}^*(\d x^j)\otimes \beta_j+\sum_{j=1}^{n-k}\d\theta_{j}\otimes\alpha_{j+k}+\mu_{k,n}^*(\d x^{j+k})\otimes\beta_{j+k} 
\]
for $\alpha_j,\beta_j\in \Omega^1(U_i)$, because $\theta$ has zero radial residues. Moreover, $\sigma_\Delta$ is given by
\[
\sigma_\Delta(x)=\left(\sqrt{x^1},\dots,\sqrt{x^k},0,\dots,0,x^{k+1},\dots,x^n\right).
\]
Therefore, $\sigma_i^*\theta=\sum_{j=1}^n\d x^j\otimes \beta_j$ on $U\cap \mathring{U}_i$, so it extends smoothly to $U_i$. 
\end{proof}

\subsection{Flat elliptic Lagrangian connections} 
\label{sec:Lag:connections:flat2}
Flat Lagrangian elliptic connections with zero radial residue will play a key role in the study of invariant K\"ahler metrics. In this subsection we extend the results of Section \ref{sec:Lag:connections:flat} to such connections.  

\begin{definition}
A \textbf{flat toric Hamiltonian $\T$-space} is a  toric Hamiltonian $\T$-space equipped with a flat Lagrangian elliptic connection that has zero radial residues. 
\end{definition}

Given an integral affine manifold $(M,\Lambda)$ and a Delzant domain $\Delta$, we denote by $\TFlat$ the sheaf on $\Delta$ of $\nabla$-flat sections of $\T_\Lambda\vert_\Delta$. This coincides with the restriction of the sheaf of flat sections of $\T$ to $\Delta$ and, as before, it is a subsheaf of $\TLag$. Also, the induced map in cohomology fits in an exact sequence
\begin{equation}
\label{eq:image:flat:cohomology2}
\xymatrix{\Check{H}^1(\Delta,\TFlat)\ar[r] & \Check{H}^1(\Delta,\TLag) \ar[r] & H^2(\Omega^\bullet_\partial(\Delta,T^*\Delta),\d^\nabla),}
\end{equation}
where the second map is given by {\eqref{eq:map:Lag:chern:class:2}}. Therefore, Corollary \ref{cor:L:flat} extends. The same goes for Theorem \ref{thm:bi-lagrangian:classification}. 

\begin{theorem}
\label{thm:bi-lagrangian:classification2}
Fix a Delzant domain $\Delta$ of an integral affine manifold $(M,\Lambda)$. There is a canonical 1:1 correspondence
\[ 
\left\{\\ \txt{flat toric Hamiltonian $\T_\Lambda$-spaces\\ $\mu:(S,\omega,\theta)\to M$\\ with $\mu(S)=\Delta$, up to equivalence \,}\right\}\ 
\tilde{\longleftrightarrow}\ 
\check{H}^1(\Delta,\TFlat)
\] 
The cohomology class corresponding to a flat toric Hamiltonian $\T$-space is mapped to its Lagrangian Chern class by the first map in \eqref{eq:image:flat:cohomology2}.
\end{theorem} 

For the proof of this, we use the following observation. 

\begin{proposition} 
The canonical toric $\T$-space $\mu_\Delta:(S_\Delta,\omega_\Delta)\to M$ admits a unique elliptic connection $\theta_\Delta$ that restricts to the canonical flat Lagrangian connection on $\T_\Lambda$ via the canonical open embedding
\begin{equation}\label{eqn:canopenemb:S_Delta}
\T_\Lambda\vert_{\mathring{\Delta}}\hookrightarrow S_\Delta.
\end{equation}
The connection $\theta_\Delta$ is flat, Lagrangian and has zero radial residues. 
\end{proposition}

\begin{proof} 
Uniqueness holds since the image of \eqref{eqn:canopenemb:S_Delta} is dense in $S_\Delta$. This also implies that for existence, it is enough to show that locally the canonical flat connection on $\T_\Lambda\vert_{\mathring{\Delta}}$ extends to an elliptic connection $1$-form for $\mu_\Delta$ with zero radial residues. So we can work in the toric coordinates induced by an integral affine chart $(U,x^i)$, as in \eqref{eqn:toriccoord:S_Delta}, over which we also have an induced trivialization of the symplectic torus bundle $\T_\Lambda\vert_{\mathring{\Delta}}$. The embedding \eqref{eqn:canopenemb:S_Delta} becomes the map $\mathbb{T}^n\times\mathring{\R}^n_k\to S_{k,n}$ given by
\[
\left(\theta_1,...,\theta_n,x^1,...,x^n\right)\mapsto \left(e^{2\pi i \theta_1}\sqrt{x^1},\dots,e^{2\pi i \theta_k}\sqrt{x^k},\theta_{k+1},...,\theta_n,x^{k+1},\dots,x^n\right). 
\] 
On the other hand, the canonical flat connection on $\T_\Lambda|_{\mathring{\Delta}}$ becomes $\sum_{i=1}^n\d\theta_i\otimes \d x^i$ so corresponds under \eqref{eqn:canopenemb:S_Delta} with the restriction to $\mathring{S}_\Delta$ of the connection given in Example \ref{example:ellconnection:stdlocmodel}. Therefore, the canonical flat connection on $\T_\Lambda\vert_{\mathring{\Delta}}$ indeed extends to an elliptic connection $1$-form for $\mu_\Delta$ with zero radial residues. Since the canonical connection on $\T_\Lambda$ is flat and Lagrangian and $\mathring{S}_\Delta$ is dense in $S_\Delta$, it follows that $\theta_\Delta$ is also flat and Lagrangian. 
\end{proof}

\begin{proof}[Proof of Theorem \ref{thm:bi-lagrangian:classification2}] The proof is virtually the same as that in the case of Lagrangian fibrations, where we now take the equivalence class of the flat toric Hamiltonian $\T_\Lambda$-space of the above proposition as the equivalence class corresponding to the zero element of the group $\check{H}^1(\Delta,\TFlat)$. 
\end{proof}

\begin{example}
\label{ex:trivspherebun:stdIA:cylinder:1}
Consider the Delzant domain $\Delta=\Ss^1\times [-1,1]$ of the standard cylinder $(\Ss^1\times\R,\Z\d x\oplus \Z\d h)$. Notice that
\[ \check{H}^1(\Delta,\TLag)=\check{H}^2(\Delta,\mathcal{O}_\Lambda)=0, \]
so every toric $\T_\Lambda$-space is isomorphic to the space $\mu:(\Tt^2\times\Ss^2,\omega)\to \Ss^1\times\R$ of Example \ref{ex:trivspherebun:stdIA:cylinder}. 

Since $\Lambda$ has trivial holonomy representation it follows -- see Remark \ref{rem:cohomology:TFlat} -- that the equivalence classes of flat $\T_\Lambda$-toric fibrations are in 1-1 correspondence with
\[ \check{H}^1(\Delta,\TFlat)=\check{H}^1(\Delta,\Tt^2)=\Tt^2. \]
For $a,b\in \R$, the classes $(e^{2\pi i a},e^{2\pi i b})\in \mathbb{T}^2$ can be realized by the elliptic connections on $\mu:\Tt^2\times\Ss^2\to \Delta$ given by the connection $1$-form
\[
\theta_{a,b}=\left(-a\d x-b\d h+\d y\right)\otimes \d 
x+\left(-b\d x+\d\phi\right)\otimes \d h.
\] 
This corresponds to the Lagrangian distribution on $\mu^{-1}(]-1,1[)$ given by
\[
D_{a,b}:=\left\langle \partial_{x}+a\partial_{y}+b\partial_\phi,\partial_h+b\partial_y \right\rangle. 
\]
\end{example}

\begin{example}\label{ex:nontrivspherebun:nonstdIA:cylinder:1}
Consider the Delzant domain $\Delta=\Ss^1\times [-1,1]$ of the non-standard cylinder $(\Ss^1\times ]-2,\infty[,\Z((h+2)\d x+x\d h)\oplus \Z\d h)$, as in Example \ref{ex:nontrivspherebun:nonstdIA:cylinder}. We still have $\check{H}^1(\Delta,\TLag)=0$ so it follows that every toric $\T_\Lambda$-space is isomorphic to the space $\mu:(\Tt^2\twprod\Ss^2,\omega)\to \Ss^1\times ]-2,\infty[$ appearing in Example \ref{ex:nontrivspherebun:nonstdIA:cylinder}. 

A computation using the Mayer-Vietoris sequence shows that 
\[ \check{H}^1(\Delta,\TFlat)=\Ss^1. \]
So, we conclude that equivalence classes of flat $\T_\Lambda$-toric fibrations are parameterized by $e^{2\pi i a}\in \Ss^1$, with $a\in \R$. These can be realized by the elliptic connections on $\mu:(\Tt^2\twprod\Ss^2,\omega)\to \Ss^1\times]-2,\infty[$ given by the connection $1$-form 
\[
\theta_{a}=(h+2)\left(-a\d x+\d y\right)\otimes \d x+\left(\d\phi+x\d y\right)\otimes \d h.
\]  
This corresponds to the Lagrangian distribution on $\mu^{-1}(]-1,1[)$ given by
\[
D_{a}:=\left\langle \partial_{x}+a(\partial_{y}-x\partial_\phi),\partial_h \right\rangle. 
\]
\end{example}

\section{Delzant construction}
\label{sec:delzant}

The proof of the classification theorem of toric Hamiltonian $\T$-spaces (Theorem \ref{thm:toric:classification}) requires the construction of a toric Hamiltonian $\T$-space $\mu:(S,\omega)\to M$ with moment map image $\mu(S)$ a given Delzant domain $\Delta\subset (M,\Lambda)$. 

In the case of a Delzant polytope $\Delta\subset (\R^n,\Z^n)$ Delzant gave a construction of a compact symplectic toric manifold with moment map image $\Delta$ by symplectic reduction of a linear Hamiltonian torus action on $(\C^d,\omega_\st)$, where $d$ is the number of facets. For a general Delzant domain, as we recalled in Section \ref{sec:elliptic:connections}, the construction of canonical toric $\T$-space uses a different method. In this section we will give a generalization of Delzant's construction for a large class of Delzant domains. Although this construction does not apply to all Delzant domains, it does cover many natural examples and has the advantage of realizing the desired toric space via symplectic reduction of a rather simple symplectic manifold.

\subsection{The primitive boundary defining functions}
\label{subsec:Delzant:normal:vectors}

The aim of this section is to show for a Delzant domain $\Delta$ in a 1-connected integral affine manifold each facet of $\Delta$ admits a natural defining integral affine function. 

In order to state a precise result, we recall that for a manifold with corners $\Delta$, given $x\in\partial \Delta$ the cone $C_x(\Delta)$ consists of all inward-pointing tangent vectors at $x$. We further denote by $\mathcal{S}^1(\Delta)$ the collection of codimension $1$ strata of $\Delta$ (i.e., the open facets of $\Delta$).

\begin{proposition}
\label{prop:Delzant:subspace:faces:functions} 
Let $(M,\Lambda)$ be a $1$-connected integral affine manifold and let $\Delta$ be a Delzant domain. For each open facet $\Sigma$ of $\Delta$, there is a unique integral affine function $\ell=\ell_\Sigma\in \O_\Lambda(M)$ satisfying the following properties:
\begin{itemize}
    \item[(i)] $\ell\vert_\Sigma=0$ and $\ker(\d \ell)\vert_{\Sigma}=T\Sigma$;
    \item[(ii)] $\ell$ is outward-pointing, i.e., $\d_x \ell(v)\leq 0$ for all $v\in C_x(\Delta)$ and all $x\in \Sigma$;
    \item[(iii)] for any $f\in \O_\Lambda(M)$ with property (i), there is a $k\in \Z$ such that $f=k\, \ell$.
\end{itemize} 
Moreover, any $x\in \Delta$ admits an open neighbourhood $U$ in $M$ such that
\begin{equation}
\label{eq:Delzant:local:halfspace:description} 
\Delta\cap U=\bigcap_{\{\Sigma\in\mathcal{S}^1(\Delta)\colon x\in \overline{\Sigma}\}} \{\ell_\Sigma\leq 0\}\cap U.
\end{equation}
The number of open facets of $\Delta$ with $x$ in their closure equals the depth of $x$ in $\Delta$.
\end{proposition}

Given a Delzant domain $\Delta$ of a $1$-connected integral affine manifold $(M,\Lambda)$ we let
\begin{equation} 
\label{eq:Delzant:halfspace:functions} 
\H(\Delta):=\{\ell_\Sigma\mid\Sigma\in \mathcal{S}^1(\Delta)\} 
\end{equation} 
be the collection of integral affine functions $\ell_\Sigma$ associated to the open facets of $\Delta$ given by the previous proposition. We call the elements of $\H(\Delta)$ \textbf{primitive boundary defining (integral affine) functions}.

\begin{example} It can happen that the same integral affine function is associated to different codimension $1$ strata. An example of this is
\[ M:=\R^2-\bigsqcup_{k\in \Z} \{k\}\times [0,\infty[,\quad \Delta=M\cap (\R\times ]-\infty,0]),
    \] 
where $M$ is equipped with the integral affine structure inherited from $(\R^2,\Z^2)$. Then $\Delta$ in fact has infinitely many codimension $1$ strata, but $\H(\Delta)$ consist of a single function.
\end{example}

\begin{remark} There might not exist boundary defining functions as in Proposition \ref{prop:Delzant:subspace:faces:functions} when $M$ is not 1-connected (see Example \ref{ex:inftypeDelzantspace}). 
\end{remark}

For the proof of Proposition \ref{prop:Delzant:subspace:faces:functions}, we first recall the following fact. 


\begin{lemma}
\label{lemma:primitive:hyperplane:normal:covector} 
Let $(V^*,\Lambda)$ be an integral affine vector space. Suppose that $H\subset V$ is an hyperplane, which is primitive in the sense that $\Span_\R(H\cap \Lambda^\vee)=H$. Then there is an $\alpha \in \Lambda$ such that:
\begin{enumerate}[(i)]
    \item $H=\ker(\alpha)$;
    \item if $\beta\in \Lambda$ and $H=\ker(\be)$, then $\beta=k\,\alpha$ for some $k\in \Z$.
\end{enumerate}
Moreover, $\alpha$ is unique up to a sign.
\end{lemma}

We call $\alpha\in \Lambda$, as in this lemma, a \textbf{primitive normal covector} to the hyperplane $H$.

\begin{lemma}
\label{lem:Delzant:subspace:faces:covectors}
Let $(M,\Lambda)$ be an integral affine manifold, let $\Delta$ be a Delzant domain and let $\Sigma$ be a codimension $1$ stratum of $\Delta$. For each $x\in \Sigma$, there is a unique outward-pointing primitive normal covector $\alpha_x\in \Lambda_x$ to the hyperplane $T_x\Sigma\subset (T_xM,\Lambda^{\vee}_x)$. The family of all such outward-pointing primitive normals is holonomy-invariant.
\end{lemma}

In this lemma, by 'holonomy invariant' we mean that for any path $\gamma$ in $\Sigma$, the holonomy representation  $\hol([\gamma]):T_{\gamma(0)}M\to T_{\gamma(1)}M$ pulls $\alpha_{\gamma(1)}$ back to $\alpha_{\gamma(0)}$. 

\begin{proof} 
The existence and uniqueness of the covectors $\alpha_x$ follows from the previous lemma and the fact that open facets of Delzant domains are primitive codimension $1$ submanifolds of $(M,\Lambda)$. By the uniqueness of primitive normal covectors, to prove holonomy-invariance it suffices to show that for any path $\gamma$ in $\Sigma$, the holonomy representation $\hol([\gamma]):T_{\gamma(0)}M\to T_{\gamma(1)}M$ maps the tangent cone $C_{\gamma(0)}(\Delta)$ onto the tangent cone $C_{\gamma(1)}(\Delta)$. It is enough to check this for a path contained in the domain of an integral affine chart.

Let $(U,\phi=(x^1,\dots,x^m))$ be an integral affine chart adapted to $\Delta$ in the sense that $\phi(U\cap \Delta)$ is an open set in $[0,\infty[\times \R^{m-1}$ around the origin. Then for any point $x\in\phi^{-1}(\{0\}\times \R^{m-1})$ we have
\[ 
C_{x}(\Delta)=\textrm{Cone}\left(\partial_{x^1},\pm\partial_{x^2}\dots,\pm\partial_{x^m}\right).
\] 
If $\gamma$ is a path lying in an open facet $\Sigma\cap U\subset\phi^{-1}(\{0\}\times \R^{m-1})$, then $\hol([\gamma])$ maps $C_{\gamma(0)}(\Delta)$ onto $C_{\gamma(1)}(\Delta)$, since the frames induced by integral affine coordinate charts are flat and the coordinate vector fields $\partial_{x^i}$ are invariant under the action of $\hol([\gamma])$.
\end{proof}


\begin{proof}
[Proof of Proposition \ref{prop:Delzant:subspace:faces:functions}] 
Let $\Sigma$ be a codimension $1$ stratum of $\Delta$. Since $M$ is $1$-connected, for any $x_0\in \Sigma$ there is a unique parallel $1$-form $\alpha\in\Omega^1(M)$ that evaluates at $x_0$ to the outward-pointing primitive normal covector. By the holonomy-invariance in Lemma \ref{lem:Delzant:subspace:faces:covectors}, for any $x\in \Sigma$ the value $\al_x$ is the outward-pointing primitive normal covector at $x$. As $M$ is $1$-connected, there is a unique function $\ell=\ell_\Sigma\in C^\infty(M)$ such that $\d \ell=\alpha$ and $\ell\vert_\Sigma=0$. Since $\alpha$ is parallel and $\alpha_{x_0}\in \Lambda_{x_0}$, $\alpha$ takes values in $\Lambda$. Hence, $\ell$ is an integral affine function. It satisfies the required properties (i)--(iii) and is the unique such integral affine function, due to the properties of the covectors in Lemma \ref{lem:Delzant:subspace:faces:covectors} and the fact that the differential of an integral affine function is a parallel $1$-form.

Next, fix $x\in \Delta$ of depth $k$. Pick an integral affine chart $(U,\phi=(x^1,\dots,x^n))$ centered at $x$ such that
$W:=\phi(U\cap\Delta)$ is a convex open set in $[0,\infty[^k\times\R^{n-k}$. The set of points in $U\cap\Delta$ of depth $1$ corresponds to the disjoint union of the convex subsets 
\[ 
W_i:=\left\{(x^1,\dots,x^m)\in W\mid x^i=0 \textrm{ and } x^j>0 \textrm{ for } 1\leq j\leq k,\text{ }j\neq i\right\},
\] 
for $i=1,\dots,k$. Since each $W_i$ is connected, $
\phi^{-1}(W_i)$ is contained in a single open facet $
\Sigma_i$ of $\Delta$, with defining function satisfying $\ell_{\Sigma_i}
|_U=-x^i$. The latter shows that $\Sigma_i\neq \Sigma_j$ if $i\neq j$. Therefore, the number of open facets with $x$ in their closure is indeed equal to $k$. It is now also clear that \eqref{eq:Delzant:local:halfspace:description} holds for this choice of $U$.
\end{proof}

\subsection{Delzant domains of finite type} 

Let $(M,\Lambda)$ be a connected integral affine manifold and let $\Delta$ be a connected Delzant domain. Consider the canonical map
\[ 
(i_\Delta)_*:\widetilde{\Delta}\to \widetilde{M},
\] 
from the universal cover of $\Delta$ into that of $M$, relative to some choice of base-point. The image $\widehat{\Delta}:=(i_\Delta)_*(\widetilde{\Delta})$ is a Delzant domain of $\widetilde{M}$. For the Delzant domains to which our construction applies, the collection $\H(\widehat{\Delta})$ of primitive boundary defining functions is finite and determines the Delzant domain in the following sense.
\begin{definition}
\label{defnDelzant:subspace:finite:half-space:type} A connected Delzant domain $\Delta$ is of \textbf{finite type} if $\H(\widehat{\Delta})$ is finite and   
    \begin{equation}
    \label{eq:Delzant:half-space:type} 
    \widehat{\Delta}=\bigcap_{\ell\in \H(\widehat{\Delta})} \{\ell\leq 0\}.
    \end{equation} 
\end{definition}  

Below we give two examples illustrating how each of these conditions may fail.
 \begin{example} Let 
 \[
 M:=\R^2-\{(0,y)\mid y\geq 0\}
 \]
 with the standard integral affine structure $\Lambda=\Z\d x\oplus \Z\d y$. Consider the Delzant domain
 \[
 \Delta=\left\{ (x,y)\in M\mid y\leq |x|\right\}. 
 \]
 This is closed in $M$ and has two open facets, with primitive boundary defining functions $\ell_{\pm}(x,y)=y\pm x$. However, $\Delta$ does not satisfy \eqref{eq:Delzant:half-space:type}. 
 \end{example}

\begin{example}\label{ex:inftypeDelzantspace} 
Consider the non-standard integral affine cylinder 
\[
(M,\Lambda):=(\Ss^1\times\R,\Z\d x\oplus \Z(\d h -x\d x))
\]
and the Delzant domain
\[
\Delta=\left\{(x,h)\in \Ss^1\times\R\mid h\leq \tfrac{1}{2}x^2-\tfrac{1}{8},\ -\tfrac{1}{2}\leq x\leq \tfrac{1}{2} \right\}.
\] 
The lift $\widehat{\Delta}\subset\R^2=\widetilde{M}$ is the region bounded from above by the curves
\[
\left\{h=\tfrac{1}{2}(x+n)^2-\tfrac{1}{8}\right\}\subset\R^2,\quad n\in \Z.
\] 
In this example, even though $\Delta$ has only one open facet, the collection $\H(\widehat{\Delta})$ is infinite, as it consists of the functions 
\[
\ell_n(x,h)=h-\tfrac{1}{2}(x+n)^2+\tfrac{1}{8},\quad n\in \Z.
\]
\end{example}
\medskip

In both of these examples the Delzant domain is not compact. We do not know of an example of a compact Delzant domain that is not of finite type.

In the coming subsection, the following property of Delzant domains of finite type will be important. 

\begin{proposition}
\label{prop:Delzant:half-space:depth:intersection:number}
Let $(M,\Lambda)$ be a $1$-connected integral affine manifold, $\Delta$ a Delzant domain of finite type and $x\in \Delta$. Consider the primitive boundary defining functions $\ell\in \H(\Delta)$ for which $\ell(x)=0$.
\begin{enumerate}[(i)]
\item The number of such functions equals the depth of $x$.
\item Their differentials at $x$ can be extended to a basis of $\Lambda_x$ and span $(T_xF)^0$.
\end{enumerate}
\end{proposition}

\begin{proof}[Proof of Proposition \ref{prop:Delzant:half-space:depth:intersection:number}] 
Let $x\in \Delta$. By Proposition  \ref{prop:Delzant:subspace:faces:functions} the number of $\ell\in \H(\Delta)$ for which $\ell(x)=0$ is at least $k:=\depth_\Delta(x)$. To show that these are in fact equal, given $\ell_0:=\ell_{\Sigma_0}\in \H(\Delta)$ such that $\ell_0(x)=0$, we need to prove that $\ell_0$ is one of the functions $\ell_i$ associated to one of the $k$ open facets $\Sigma_i$ with $x$ in their closure. Choose an integral affine chart $(U,\phi=(x^1,\dots,x^n))$ centered at $x$, as in the second part of the proof of Proposition \ref{prop:Delzant:subspace:faces:functions}. Since $\ell_0$ is affine and vanishes at $x$, its restriction to $U$ is of the form:  
\[
\ell_0\vert_U=-\sum_{i=1}^nt_ix^i, \quad t_i\in \R.
\]
In view of \eqref{eq:Delzant:half-space:type}, $\ell_0$ is non-positive on $\Delta$ and so $t_i\geq 0$ for all $i=1,\dots,k$ and $t_i=0$ for all $i=k+1,\dots,n$. 
As in the proof of Proposition \ref{prop:Delzant:subspace:faces:functions}, $\ell_i|_U=-x^i$ for each $i=1,..,k$. So, 
\[
\ell_0=\sum_{i=1}^kt_i\ell_i, \quad t_i\geq 0,
\] since this equality holds on $U$ and the functions on both sides are affine. It follows from property (iii) in Proposition \ref{prop:Delzant:subspace:faces:functions} that, if $\ell_0\neq \ell_i$ for all $i=1,\dots,k$, then there are at least two distinct $i_1,i_2\in\{1,\dots,k\}$ for which $t_{i_1},t_{i_2}\neq 0$. In that case, since $\Delta$ satisfies \eqref{eq:Delzant:half-space:type}, $\Sigma_0$ must be contained in the codimension $2$ submanifold $\{\ell_{i_1}=\ell_{i_2}=0\}$, which is a contradiction. So, our first claim follows. Moreover, the fact that $\ell_i\vert_U=-x^i$ for all $i=1,...,k$ shows that the differentials of these functions at $x$ extend to a basis of $\Lambda_x$ and span $(T_xF)^0$. 
\end{proof}


\subsection{The 1-connected case} 

In this section we give a Delzant type construction for a finite type Delzant domain of a $1$-connected integral affine manifold. Treating the 1-connected case first is convenient for the proof of the general case, and moreover it is interesting on its own, as it already covers a class of non-compact toric symplectic manifolds (cf. Section \ref{sec:classical:toric}).

Let $(M,\Lambda)$ be a 1-connected integral affine manifold and let $\Delta$ be a Delzant domain of finite type, so that
\[ \Delta=\bigcap_{i=1}^d\{\ell_i\leq 0\}, \]
where $\ell_1,\dots,\ell_d$ are the primitive boundary defining functions. We consider the map
\[
\phi:=(\ell_1,...,\ell_d):(M,\Lambda)\to (\R^d,\Z^d).
\]
Since this map is integral affine, it induces a morphism of symplectic torus bundles
\[
    (\phi,\underline{\phi}^*):(\T_{\Lambda},\Omega_{\Lambda})\dto (\Tt^d\times \R^d,\sum_{i=1}^d\d\theta_i\wedge \d x^i)
\] as in Proposition \ref{prop:ia:morphisms}. This, in turn, induces a Hamiltonian $\Tt^d$-action on $(\T_\Lambda,\Omega_\Lambda)$ with moment map the composition of $\phi$ with the bundle projection $\pi$ and action given by
\[
(\theta_1,...,\theta_d)\cdot[\alpha]=\Big[\alpha+\sum_{i=1}^d \theta_i\,\d \ell_i\Big].
\]

\begin{theorem}
\label{thm:Delzant:1:connected}
Consider the diagonal Hamiltonian $\Tt^d$-action with moment map 
\[ 
\mu_\phi:=\phi+(|z_1|^2,...,|z_d|^2):(\T_\Lambda,\Omega_\Lambda)\times (\C^d,\omega_\textrm{st})\to \R^d. 
\]
The $\Tt^d$-action is free on $\mu_\phi^{-1}(0)$ and the symplectic quotient 
\[ ((\T_\Lambda\times \C^d)\sslash\Tt^d,\omega_\red) \] 
is a toric Hamiltonian $\T_\Lambda$-space for the action induced by translations on the first factor. Moreover, this space is equivalent to the canonical toric Hamiltonian $\T_\Lambda$-space $\mu_\Delta:(S_\Delta,\omega_\Delta)\to M$.
\end{theorem}

\begin{proof}
We start by checking the diagonal $\Tt^d$-action on $\mu_\phi^{-1}(0)$ is free, so the symplectic quotient is well-defined. Let $([\alpha_x],z)\in \mu_\phi^{-1}(0)$ and $\theta=(\theta_1,...,\theta_d)\in \Tt^d$ be such that $\theta$ fixes $([\alpha_x],z)$. Since $\theta$ fixes $z$, it follows that $\theta_i=0 \mod \Z$ whenever $z_i\neq 0$. On the other hand, because $\ell_i(x)=-|z_i|^2$, the remaining $i\in \{1,...,d\}$, say $i_1,...,i_k$, are those for which $\ell_i(x)=0$. So, by Proposition \ref{prop:Delzant:half-space:depth:intersection:number}, $(\d \ell_{i_1})_x,\dots, (\d \ell_{i_k})_x$ extends to a basis of $\Lambda_x$. Hence, the fact that $\theta$ fixes $[\alpha_x]$ implies that $\theta_{i_1}=...=\theta_{i_k}=0 \mod \Z$ as well. 

Next, observe that the action of $\T_\Lambda$ on itself by translation descends to a Hamiltonian $\T_\Lambda$-action along
\[
\mu:\big((\T_\Lambda\times \C^d)\sslash\Tt^d,\omega_\red\big)\to M, \quad [\alpha,z]\mapsto \pi(\alpha),
\] 
which has moment map image 
\[
\bigcap_{i=1}^d\{\ell_i\leq 0\}=\Delta.
\]

Finally, we claim that we have an equivalence of Hamiltonian $\T_\Lambda$-spaces
\[
\xymatrix{
(S_\Delta,\omega_\Delta) \ar[dr]^{\mu_\Delta}\ar[rr]_---{\sim}^---{\Psi} & & ((\T_\Lambda\times \C^d)\sslash\Tt^d,\omega_\red\big)\ar[dl]_{\mu} \\
& M}
\]
given by
\[ [\alpha_x]\longmapsto [\alpha_x,(\sqrt{-\ell_1(x)},\dots,\sqrt{-\ell_d(x)})]. \]
This map is a $\T_\Lambda$-equivariant bijection with inverse induced by
\[ 
(\alpha_x,z)\longmapsto [t^{-1}\cdot\alpha_x],\]
where $t=(t_1,\dots,t_d)$ is given by
\[
t_i=t_i(z,x):=\begin{cases}
    \tfrac{z_i}{|z_i|},\quad \text{if $\ell_i(x)\ne 0$},\\ \\
    \ 0, \quad \text{if $\ell_i(x)= 0$}.
\end{cases}
\]
To see that $\Psi$ is a smooth map, we note that around a point $x_0$ of depth $k$ its expression in the chart \eqref{eqn:toriccoord:S_Delta} is
\[ 
(z_1,\dots,z_k,\theta_{k+1},\dots,\theta_n,x^{k+1},\dots,x^n)\mapsto \big[\sum_{i=k+1}^n \theta_i\d x^i,(w_1,\dots,w_d)\big],
\]
where
\[
w_i:=\begin{cases}
    \quad z_i,\quad \text{if $\ell_i(x_0)= 0$},\\ \\
    \sqrt{-\ell_i(x)}, \quad \text{if $\ell_i(x_0)\ne 0$}.
\end{cases}
\]
Finally, we observe that this map is a symplectomorphism. Indeed, over $\mathring{\Delta}$, where $k=0$, it is given by the symplectic map
\[ (\theta_{1},\dots,\theta_n,x^1,\dots,x^n)\mapsto \big[\sum_{i=1}^n \theta_i\d x^i,(\sqrt{-\ell_1(x)},\dots,\sqrt{-\ell_d(x)})\big]. \]
By density, we conclude that $\Phi$ is a symplectic map. Since it is a bijection between manifolds of the same dimension, it is a symplectomorphism.
\end{proof}

\begin{example}[The standard Delzant construction]
\label{ex:standard:Delzant}
One can recover the usual Delzant construction of a compact symplectic manifold from the construction in Theorem \ref{thm:Delzant:1:connected} via reduction by stages as follows. 

Fix a Delzant polytope $\Delta\subset \R^n$ with $d$ facets. In this case, Theorem \ref{thm:Delzant:1:connected} gives a symplectic toric manifold 
\begin{equation}
    \label{eq:not:Delzant}
    \mu:\big((\Tt^n\times\R^n\times \C^d)\sslash\Tt^d,\omega_\red\big)\to \R^n.
\end{equation}
If $\ell_1,\dots,\ell_d$ are the primitive boundary defining functions, one has a linear map
\[ \R^d\to \R^n, \quad e_a\mapsto \d\ell_a,\]
where $\{e_a\}$ is the standard basis of $\R^d$. The kernel of this map is a Lie subalgebra $\mathfrak{n}\subset \R^n$, which integrates to a subtorus $N\subset\Tt^d$. Choosing a subtorus $\Tt^n\subset \Tt^d$ complementary to $N$, one can first perform symplectic reduction by this subtorus. The resulting reduced space is canonically isomorphic to the $d$-dimensional complex vector space with its canonical symplectic form via the map
\[ 
(\Tt^n\times\R^n\times \C^d)\sslash\Tt^d,\omega_\red)\to (\C^d,\omega_\st), \quad (x,\theta,z)\mapsto (-\theta)\cdot z. 
\]
Therefore, the symplectic toric manifold  \eqref{eq:not:Delzant} is canonically isomorphic to 
\begin{equation}
    \label{eq:Delzant}
    \mu: (\C^d\sslash N,\omega_\red)\to \R^n.
\end{equation}
This is the compact toric symplectic manifold constructed by Delzant \cite{De88}.

\end{example}

\subsection{The general construction} 
Let $(M,\Lambda)$ be a connected integral affine manifold and let $\Delta$ be a Delzant domain of finite type. Consider the canonical map
\[ 
(i_\Delta)_*:\widetilde{\Delta}\to \widetilde{M},
\] 
from the universal cover of $\Delta$ into that of $M$, relative to a fixed base-point in $\Delta$. The image $\widehat{\Delta}:=(i_\Delta)_*(\widetilde{\Delta})$ is a Delzant domain of $\widetilde{M}$. Let $\Gamma$ be the image of the group homomorphism $(i_\Delta)_*:\pi_1(\Delta)\to \pi_1(M)$. Note that the canonical $\Gamma$-action on $\widetilde{M}$ is by integral affine isomorphisms and preserves $\widehat{\Delta}$. Therefore, it induces an action on the set of primitive boundary defining functions $\H(\widehat{\Delta})$. Fixing an ordering $\H(\Delta)=\{\ell_1,...,\ell_d\}$, this induces a homomorphism $\Gamma\to \text{Sym}(1,\dots,d)$ and, therefore, actions on $\C^d$ and $\Tt^d$ by coordinate permutations. The group $\Gamma$ also acts on $\T_{\widetilde{\Lambda}}$ via the cotangent lift of the action on $\widetilde{M}$. All together, we obtain a Hamiltonian action of the semi-direct product $\Gamma\ltimes \Tt^d$, with moment map 
\[ 
\mu_\phi:=\phi+(|z_1|^2,...,|z_d|^2):(\T_{\widetilde{\Lambda}},\Omega_{\widetilde{\Lambda}})\times (\C^d,\omega_\textrm{st})\to \R^d,
\] 
where $(\T_{\widetilde{\Lambda}},\Omega_{\widetilde{\Lambda}})$ is the symplectic torus bundle over $(\widetilde{M},\widetilde{\Lambda})$. 

\begin{theorem} 
\label{thm:Delzant}
The $\Gamma\ltimes \Tt^d$-action is free on $\mu_\phi^{-1}(0)$ and the symplectic quotient 
\[ 
((\T_{\widetilde{\Lambda}}\times \C^d)\sslash(\Gamma\ltimes \Tt^d),\,\omega_\red) 
\] 
is a toric Hamiltonian $\T_\Lambda$-space, which is equivalent to the canonical toric Hamiltonian $\T_\Lambda$-space $\mu_\Delta:(S_\Delta,\omega_\Delta)\to M$.
\end{theorem}

\begin{proof} 
The fact that the $\Gamma\ltimes \Tt^d$-action on $\mu^{-1}_\phi(0)$ is free and proper, follows using that the $\Tt^d$-action on this set is free and the $\Gamma$-action on $\widetilde{M}$ is free and proper. 

Next, note that the symplectic torus bundle $(\T_{\widetilde{\Lambda}},\Omega_{\widetilde{\Lambda}})$ is canonical isomorphic to the pull-back of $(\T_\Lambda,\Omega_\Lambda)$ along the covering map $\widetilde{M}\to M$. The action of $(\T_{\widetilde{\Lambda}},\Omega_{\widetilde{\Lambda}})$ on the first factor in $\T_{\widetilde{\Lambda}}\times\C^d$ descends to a Hamiltonian $\T_\Lambda$-action on the symplectic quotient. This makes it a toric $\T_\Lambda$-space. 

For the last statement, we use the following lemma.
\begin{lemma} 
Let $(M,\Lambda)$ be a connected integral affine manifold, let $\Delta$ be a connected Delzant domain and consider $\widehat{\Delta}$ and $\Gamma$ as above. The $\Gamma$-action on $\widehat{\Delta}$ lifts to a symplectic action on $(S_{\widehat{\Delta}},\omega_{\widehat{\Delta}})$ and there is an equivalence of Hamiltonian $\T_\Lambda$-spaces
\[
\xymatrix{
(S_{\widehat{\Delta}},\omega_{\widehat{\Delta}})/\Gamma\ar[dr]_{\underline{\mu_{\widetilde{\Delta}}}}\ar[rr]
& & (S_\Delta,\omega_\Delta)\ar[dl]^{\mu_\Delta}\\
&M
}
\]
where $\underline{\mu_{\widetilde{\Delta}}}$ is the composition of the map $S_{\widehat{\Delta}}/\Gamma\to\widetilde{M}/\Gamma$ induced by $\mu_{\widetilde{\Delta}}$ with the quotient map $\widetilde{M}/\Gamma\to M$.
\end{lemma}
\begin{proof} 
This readily follows from the naturality of the construction of the canonical toric space associated to a Delzant domain (see \cite{Maarten}). 
\end{proof}

Since $\Tt^d$ is a normal subgroup of $\Gamma\ltimes \Tt^d$, there is an induced $\Gamma$-action on the symplectic quotient $
((\T_{\widetilde{\Lambda}}\times \C^d)\sslash \Tt^d,\,\omega_\red)$. The equivalence between $\T_{\widetilde{\Lambda}}$-spaces
\[ (S_{\widehat{\Delta}},\omega_{\widehat{\Delta}})\to (\T_{\widetilde{\Lambda}}\times \C^d)\sslash \Tt^d,\,\omega_\red) \]
given by Theorem \ref{thm:Delzant:1:connected} is $\Gamma$-equivariant. Passing to the quotients, the equivalence in the statement follows from the previous lemma. 
\end{proof}

\begin{remark}[Realizing Lagrangian Chern classes]
Theorem \ref{thm:Delzant} gives a Delzant type construction for toric $\T$-spaces with trivial Lagrangian Chern class. This construction can be extended to the case of non-trivial Lagrangian Chern class. After possible shrinking $M$, one can assume that any class in $\check{H}^1(\Delta,\TLag)$ can be realized as the restriction of the Lagrangian Chern class of a principal $\T_\Lambda$-space $\mu_P:(P,\omega_P)\to M$. Then a toric $\T$-space with a non-trivial class, can be realized as a symplectic quotient of the form
\[ 
((q^*P\times \C^d)\sslash(\Gamma\ltimes \Tt^d),\,\omega_\red) 
\] 
where $q:\widetilde{M}\to M$ is the covering map.    
\end{remark}

\section{Invariant K\"ahler metrics I: Lagrangian fibrations}
\label{sec:inv:metrics:non-singular}
In this section we study $\T$-invariant K\"ahler structures on \emph{principal} Hamiltonian $\T$-spaces, i.e., Lagrangian fibrations. In the next section we will consider toric Hamiltonian $\T$-spaces.

\begin{definition}
Given a Hamiltonian $\T$-space $\mu:(S,\omega)\to M$, a \textbf{$\T$-invariant K\"ahler structure} is a $\T$-invariant K\"ahler metric $G$ on $S$ with K\"ahler form $\omega$.
\end{definition}

Given a Lagrangian fibration $\mu:(S,\omega)\to M$ we will call a compatible K\"ahler metric $g$ on $S$ \textbf{invariant} if it is $\T_\Lambda$-invariant, where $\Lambda\subset T^*M$ is the induced integral affine structure on the base. Invariant K\"ahler metrics on a Lagrangian fibration are intimately related to Hessian metrics on the base of the fibration. We recall that a \textbf{Hessian metric} on an (integral) affine manifold $M$ is a Riemannian metric $g$ satisfying
\begin{equation}
    \label{eq:Hessian}
    \d^\nabla g^\flat=0.
\end{equation}
where $g^\flat:TM\to T^*M$ denotes the $T^*M$-valued form on $M$ given by contraction with $g$. This condition is equivalent to the existence of a potential for the metric in each (integral) affine chart $(U,x^i)$ with contractible domain, i.e., a function $\phi\in C^\infty(U)$ such that
\[ g|_U=\sum_{i,j}\frac{\partial^2 \phi}{\partial x^i\partial x^j}\d x^i\otimes\d x^j. \]
We refer to \cite{Shima2007} for details on Hessian metrics.

The main theorem of this section is the following. We use the symbol $\perp$ to denote the orthogonal space relative to a Riemannian metric.
\begin{theorem}\label{thm:kahlmetrlagrfib}
Let $\mu:(S,\omega)\to M$ be a a Lagrangian fibration. A choice of invariant K\"ahler metric $G$ on $S$ makes $\mu:S\to M$ into a Riemannian submersion, with the following properties:
\begin{enumerate}[(i)]
    \item $(\ker\d\mu)^\perp$ is a flat Lagrangian connection;
    \item The metric $g$ induced on the base $M$ is a Hessian metric.
\end{enumerate}
Conversely, given a flat Lagrangian connection $D$ and a Hessian metric $g$ on $M$, there is unique invariant K\"ahler metric $G$ on $S$ inducing $g$ such that $D=(\ker\d\mu)^\perp$.
\end{theorem}

\begin{remark}
    This theorem shows that the existence problem for invariant K\"ahler metric on a Lagrangrian fibration reduces to two independent obstruction problems. As discussed in Section \ref{sec:Lag:connections:flat}, the obstruction to the existence of a flat Lagrangian connection for a Lagrangian fibration can be expressed in terms of its Lagrangian Chern class. The obstruction to the existence of a Hessian metric on an (integral) affine manifold are more subtle and have been studied by several authors (see, e.g., \cite{Shima2007} and the references therein).  
\end{remark}

\begin{example} 
\label{ex:KodairaThurston1'}
Consider the symplectic torus bundle in Example \ref{ex:KodairaThurston1}. Its base $(\Tt^2,\Z\d x^1\oplus \Z(\d x^2-x^1\d x^1))$ is an integral affine manifold which does not admit a Hessian metric. Indeed, this being a symplectic torus bundle, it admits a flat Lagrangian connection. However, its total space $(\T,\Omega)$ does not admit any K\"ahler structure, as it is the Kodaira-Thurston symplectic manifold. 

Notice that the symplectic manifold $(\T,\Omega)$ admits another Lagrangian fibration
\[ 
\mu:\T \to \Tt^2,\quad [\theta_1,\theta_2,x^1,x^2]\mapsto [\theta_2,x^1].
\]
The induced integral affine structure on the base $\Tt^2$ is now the standard one.  This fibration does not admit a Lagrangian section. In fact, due to Theorem \ref{thm:kahlmetrlagrfib}, it does not even admit a flat Lagrangian connection, since its base $(\Tt^2,\Lambda_\st)$ admits a Hessian metric, but its total space is not K\"ahler. 
\end{example}

Two equivalent K\"ahler Lagrangian fibrations induce the same metric on their base. Moreover, given two K\"ahler Lagrangian fibrations that induce the same Hessian metric on $M$, an equivalence of the underlying Lagrangian fibrations intertwines the K\"ahler structures if and only if it intertwines the orthogonal complements to the fibers. Therefore, we conclude the following from Theorems \ref{thm:bi-lagrangian:classification} and \ref{thm:kahlmetrlagrfib}.

\begin{corollary}
    \label{cor:equivofflatfib:preservesKahler}
    Given an integral affine manifold $(M,\Lambda)$, there is a canonical 1:1 correspondence  
\[ 
\left\{\\ \txt{K\"ahler Lagrangian fibrations\\ $\mu:(S,\omega,G)\to M$ \\ up to equivalence \,}\right\}\ 
\tilde{\longleftrightarrow}\ 
\left\{\\ \txt{Hessian metrics \\ $g$ on $M$\,}\right\}\ \times 
\check{H}^1(M,\TFlat)
\]
\end{corollary}

Choosing an integral affine chart $(U,x^i)$ with contractible domain, there exist action-angle coordinates $(\mu^{-1}(U),x^i,\theta_j)$ (see Remark \ref{rem:action-angle}) in which the connection takes the form
\[
\theta=\sum_{i=1}^n\d\theta_i\otimes \d x^i.
\] 
In such coordinates the invariant K\"ahler metric $G$ can be written as
\begin{equation}
    \label{eq:metric:action:angle}
    G=\sum_{i,j=1}^n g_{ij}\, \d x^i\otimes\d x^j+\sum_{i,j=1}^n g^{ij}\, \d \theta_i\otimes\d \theta_j,
\end{equation}
where
\[ 
g_{ij}=\frac{\partial^2 \phi}{\partial x^i\partial x^j},
\]
for some function $\phi\in C^\infty(U)$, and $(g^{ij})$ denotes the inverse matrix of $(g_{ij})$. The corresponding complex structure takes the form
\[
J=\sum_{i,j=1}^n g_{ij}\, \d x^i\otimes \partial_{\theta_j}-\sum_{i,j=1}^n g^{ij}\d \theta_i\otimes \partial_{x^j}.
\]

\subsection{Proof of Theorem \ref{thm:kahlmetrlagrfib}}
Let $\Lambda$ be the induced integral affine structure on the base of $\mu:(S,\omega)\to M$. Denote by $G$ a $\T_\Lambda$-invariant K\"ahler metric on $S$. The distribution orthogonal to the fibers
\[ D:=(\ker\d\mu)^\perp \]
is $\T_\Lambda$-invariant and the restriction of $G$ to $D$ is also $\T_\Lambda$-invariant. It follows that there is a unique Riemannian metric $g$ on $M$ such that $\d_x\mu:D_x\to T_p M$ is an isometry for all $x\in\mu^{-1}(p)$, i.e., $\mu:(S,G)\to(M,g)$ is a Riemannian submersion.

Now observe that we can express $D$ as:
\[ D=J(\ker\d\mu), \]
where $J:TS\to TS$ denotes the complex structure. Since $\ker\d\mu$ is Lagrangian, we conclude that $D$ is Lagrangian. It follows also that the vector fields $J\act(\al)$, with $\al$ flat 1-forms, generate $D$. Now, the $\T_\Lambda$-invariance of $J$ amounts to
\[    \Lie_{\act(\al)}J=0\quad\text{ if }\quad\nabla\al=0, \]
so we find that
\[ [\act(\al),J(\act(\be))]=(\Lie_{\act(\al)}J)+J([\act(\al),\act(\be)])=0, \]
for any flat 1-forms $\al$ and $\be$. Then the vanishing of the Nijenhuis torsion of $J$ gives:
\[ [J(\act(\al)),J(\act(\be))]=J([J(\act(\al)),\act(\be)]+[\act(\al),J(\act(\be))])+[\act(\al),\act(\be)]=0. \]
This shows that $D$ is flat, so (i) holds.

To prove that $g$ is an Hessian metric we will check that \eqref{eq:Hessian} holds:
\[ 
\d^\nabla g^\flat(X,Y)=\nabla_X g^\flat(Y)-\nabla_Y g^\flat(X)-g^\flat([X,Y])=0,\quad (X,Y\in\X(M)). 
\]
To prove this it is enough to consider vector fields of the form $X=g^\sharp(\al)$ and $Y=g^\sharp(\be)$, with $\al$ and $\be$ flat, since these locally generate all vector fields. Therefore, we are left to check that:
\[ [g^\sharp(\al),g^\sharp(\be)]=0\quad \text{ if }\quad \nabla\al=\nabla\be=0. \]
Now observe that the definition of the metric $g$ yields:
\[ \mu_*(J(\act(\al)))=g^\sharp(\al). \]
Hence, if $\al$ and $\be$ are flat, this last relation and the computation above shows that:
\[ [g^\sharp(\al),g^\sharp(\be)]=\mu_*([J(\act(\al)),J(\act(\be))]=0. \]
This shows that (ii) holds.

To prove the converse, note that given a flat Lagrangian connection $D$ we can define a unique Riemannian metric $G$ on $S$ by requiring the following properties to hold:
\begin{enumerate}[(a)]
    \item $\ker\d\mu$ and $D$ are orthogonal;
    \item On $D$ the metric is the unique one that gives isometries
    \[ \d_p\mu:(D_p,G|_{D_p})\to (T_{\mu(p)}M,g_{\mu(p)});\] 
    \item On $\ker\d\mu$ the metric is the unique one that gives isometries
    \[ (T^*_{\mu(p)}M,g^{-1}_{\mu(p)})\to (\ker\d_p\mu,G|_{\ker\d_p\mu}),\quad\al\mapsto \act(\al)|_p. \] 
\end{enumerate}
The resulting Riemannian metric is clearly invariant and we are left to prove that it is a K\"ahler metric with K\"ahler form $\omega$. 

Given $p\in S$ and $\al\in T^*_{\mu(p)}M$ there is a unique horizontal vector  $\widetilde{g^\sharp(\al)}\in D_p$ such that
\[ 
\d_p \mu(\widetilde{g^\sharp(\al)})=g^\sharp(\al). 
\]
Since both sides of the previous equation are linear in $\al$, it follows that there is a unique vector bundle map $J:\ker\d\mu\to D$ such that
\begin{equation}
    \label{eq:hessian:cplx}
    J(\act(\al)):=\widetilde{g^\sharp(\al)}.
\end{equation} 
Since $TS=\ker\d\mu\oplus D$, we obtain a vector bundle map $J:TS\to TS$ by requiring $J^2=-I$. Notice that $J$ is an invariant almost complex structure
\[ 
\Lie_{\act(\al)}J=0,\quad \text{ if }\quad \nabla\al=0. \]
We claim that
\[ G(X,Y)=\omega(X,JY),\quad (X,Y\in TS). \]
To see this, we check that the bilinear form $\omega(\cdot,J\cdot)$ satisfy properties (a)-(c) above:
\begin{enumerate}[(a)]
    \item Since $J(\ker\d\mu)=D$ and $J(D)=\ker\d\mu$, and both distributions are Lagrangian, it follows that they are orthogonal relative to the bilinear form $\omega(\cdot,J\cdot)$.
    \item By \eqref{eq:hessian:cplx}, it follows that \begin{align*}
        \omega(\widetilde{g^\sharp(\al)},J\widetilde{g^\sharp(\be)})&=-\omega(\widetilde{g^\sharp(\al)},\act(\be))\\
        &=(\mu^*\be)(\widetilde{g^\sharp(\al)})\\
        &=\be(g^\sharp(\al))=g(g^\sharp(\al),g^\sharp(\be)).
    \end{align*}
    Hence, the restriction of $\d\mu$ to $D$ maps the bilinear form $\omega(\cdot,J\cdot)$ to $g$.
    \item Again by \eqref{eq:hessian:cplx}, it follows that 
    \begin{align*}
        \omega(\act(\al),J(\act(\be)))&=
        \omega(\act(\al),\widetilde{g^\sharp(\be)}))\\
        &=(\mu^*\al)(\widetilde{g^\sharp(\be)})\\
        &=\al(g^\sharp(\be))=g^{-1}(\al,\be).
    \end{align*}
    Hence, the last item is also satisfied.
\end{enumerate}

It remains to show that $J$ is integrable, i.e., that its Nijenhuis torsion vanishes. For this, note that if $\nabla\al=\nabla\be=0$ then 
\begin{equation}
    \label{eq:Nijenhuis:1}
    [\act(\al),J(\act(\be))]=(\Lie_{\act(\al)}J)(\act(\be))+J([\act(\al),\act(\be)])=0,
\end{equation}
since $J$ is invariant and $\act(\al)$ and $\act(\be)$ commute. On the other hand, since $g$ is Hessian, we have for flat $\al$ and $\be$:
\begin{align*}
g^\flat([g^\sharp(\al),g^\sharp(\be)])&=g^\flat([g^\sharp(\al),g^\sharp(\be)])-\nabla_{g^\sharp(\al)}\be+\nabla_{g^\sharp(\be)}\al\\
&=-\d^\nabla g^\flat(g^\sharp(\al),g^\sharp(\be))=0.
\end{align*}
It follows that $[g^\sharp(\al),g^\sharp(\be)]=0$, and using \eqref{eq:hessian:cplx}, we conclude that if :
\begin{equation}
    \label{eq:Nijenhuis:2}
    [J(\act(\al)),J(\act(\be))]=0,\quad \text{ if }\quad \nabla\al=\nabla\be=0. 
\end{equation}
From \eqref{eq:Nijenhuis:1} and \eqref{eq:Nijenhuis:2}, we find that the Nijenhuis torsion of $J$ satisfies:
\[ N_J(\act(\al),\act(\be))=N_J(\act(\al),J(\act(\be)))=N_J(J(\act(\al)),J(\act(\be)))=0, \]
for any flat sections $\al$ and $\be$. Since the vector fields $\act(\al)$ and $J(\act(\al))$ locally generate all vector fields, we conclude that $N_J\equiv 0$. This concludes the proof.
\qed
\subsection{Abreu's equation}
The scalar curvature of an invariant K\"ahler structure on a Lagrangian fibration is an invariant smooth function on the total space, and hence corresponds to a smooth function on the base of the fibration. In his work on classical symplectic toric manifolds \cite{Abr98}, Abreu gave an explicit expression for this function in terms of integral affine coordinates and observed that the K\"ahler metric is extremal in the sense of Calabi (see, e.g., \cite{Calabi82}) if and only if this function is affine. This generalizes to toric $\T$-spaces. In this section we consider the case of Lagrangian fibrations. 

Recall that for an affine manifold $(M,\nabla)$ with a Hessian metric:
\begin{itemize}\item the \textbf{first Koszul form} $\alpha$ is the $1$-form on $M$ defined by:
\[
    \alpha(X)\nu=\nabla_X\nu, \quad X\in \X(M),
\] with $\nu$ the volume density of the Hessian metric,
\item the \textbf{second Koszul form} $\beta:=\nabla\alpha$ is the symmetric $(0,2)$-tensor given by:
\[ (\nabla\alpha)(X,Y)=(\nabla_X\alpha)(Y).
\]
\end{itemize}
{ In particular, these Koszul forms are defined for every integral affine manifold $(M,\Lambda)$ with a Hessian metric.}

As will be clear from the proof below, the second Koszul form is an incarnation of the Ricci tensor in Hessian geometry. Further, recall that any integral affine manifold $(M,\Lambda)$ carries a canonical volume density $\nu_\Lambda$, which in any integral affine chart $(U,x^i)$ is given by
\[
\nu_\Lambda\vert_U=|\d x^1\wedge \cdots\wedge \d x^n|.
\]

\begin{theorem} 
\label{thm:lagrfib:metrics}
Let $\mu:(S,\omega,G)\to M$ be a K\"ahler Lagrangian fibration, inducing an integral affine structure $\Lambda$ on $M$ and a Hessian metric $g$ on $(M,\Lambda)$. Let $\alpha$ and $\beta$ be the first and second Koszul forms of $(M,\Lambda,g)$, $\nu_\Lambda$ the volume density of $\Lambda$ and $\nabla^g$  the Levi-Civita connection of $g$. The scalar curvature of $G$ is given by
\begin{equation}
    \label{eqn:scalcurvexpr:lagrfib}
    S_G=\mathrm{Tr}_{g}(2\nabla^{g}\alpha-\beta)\circ\mu=\mathrm{div}_{\nu_\Lambda}(g^\sharp(\alpha))\circ\mu.
\end{equation}
Moreover, $G$ is an extremal metric if and only if $S_G$ is affine, viewed as function on the integral affine manifold $(M,\Lambda)$.
\end{theorem}

\begin{remark}
\label{rem:koszform:scalcurv:coordinates} $\quad$
\begin{enumerate}[(i)]
    \item Given a compact complex manifold $(S,J)$ and a K\"ahler class, the Calabi functional is defined on the space of $J$-compatible K\"ahler metrics $G$ in the given K\"ahler class and is given by
    \begin{equation}
        \label{eq:Calabi}
        G\mapsto \int_S S_G^2\, \d v_{G},
    \end{equation}
    where $v_G$ is the Riemannian volume form.
    A K\"ahler metric $G$ is called \textbf{extremal} if it is a critical point of this functional \cite{Calabi82}. Using the corresponding Euler-Lagrange equations one finds (see, e.g., \cite{Szekelyhidi14}) that a K\"ahler metric $G$ is extremal if and only if
    \begin{equation}
        \tag{E}\label{eq:extremal}
        (\grad S_G)^{(1,0)}\textrm{ is a holomorphic vector field.}
    \end{equation}
    We adopt condition \eqref{eq:extremal} as the definition of  extremal K\"ahler metric, whether $S$ is a compact manifold or not.
    
    \item In our approach, where we fix a symplectic manifold $(S,\omega)$, one has the following equivalent approach. The functional \eqref{eq:Calabi} is now defined on the space of $\omega$-compatible K\"ahler metrics $G$ whose complex structure belong to a fixed symplectomorphism class. The critical points are still the K\"ahler metrics satisfying  \eqref{eq:extremal}, which in symplectic terms can be reformulated as  
    \begin{equation}
        \tag{E'}\label{eq:extremal:symp}
        X_{S_G}\textrm{ is a Killing vector field.}
    \end{equation}
    \item  In an integral affine chart $(U,x^1,\dots,x^n)$ for $(M,\Lambda)$, the first Koszul form is given by
 \[
 \alpha|_U=\frac{1}{2}\sum_{j=1}^n\frac{\partial \log(\det(g))}{\partial x^j}
 \]
 and \eqref{eqn:scalcurvexpr:lagrfib} yields
    \[
   S_G|_U=-\frac{1}{2}\sum_{j,k=1}^n\frac{\partial^2g^{jk}}{\partial x^j\partial x^k},
    \]
where $g^{-1}=\sum_{j,k=1}^ng^{jk}\partial_{x^j}\partial_{x^k}$. The condition that the function $S_G$ is affine then becomes a 4th order PDE for the potential of the metric $g$, known as Abreu's equation (see \cite{Abr98}).
\end{enumerate}

\end{remark}
\begin{proof} 
The proof is essentially as in \cite{Abr98}. In view of Corollary \ref{cor:equivofflatfib:preservesKahler} and Remark \ref{rem:cohomology:TFlat}, $M$ can be covered by open sets $U$ over which the given K\"ahler Lagrangian fibration is equivalent to the Lagrangian fibration $(\T_\Lambda,\Omega_\Lambda)\to M$ equipped with the K\"ahler structure induced as in Theorem \ref{thm:kahlmetrlagrfib} by $g$ and the canonical horizontal distribution. So, we can assume that $(S,\omega)=(\T_\Lambda,\Omega_\Lambda)$ and $\mu$ is the bundle projection. 

Given an integral affine manifold $(M,\Lambda)$ with a Hessian metric $g$, there is another integral affine structure on $M$ defined by the lattice dual to
\[
\Lambda^\sharp:=g^\sharp(\Lambda)\subset TM.
\]
The metric $g$ is Hessian with respect to $\Lambda^\sharp$ as well. This defines the dual Hessian structure, in the terminology of \cite{Shima2007}. The second Koszul form of $(M,\Lambda^\sharp,g)$ is (see \cite[Lemma 3.1]{Shima2007})
\begin{equation}
\label{eqn:secondkoszul:dual}
\beta_\sharp=\beta-2\nabla^{g}\alpha.
\end{equation} 

Just like $\T_\Lambda=T^*M/\Lambda$, the torus bundle $\pi:TM/{\Lambda^\sharp}\to M$ comes with a canonical horizontal distribution and vector bundle isomorphism $
\act:\pi^*TM\xrightarrow{\sim}\ker(\d\pi)$. Moreover, $TM/{\Lambda^\sharp}$ comes with a natural complex structure $J$ given by
\[
J(\act(X))=\widetilde{X},\quad J(\widetilde{X})=-\act(X),
\] and the metric $g$ induces a metric $G_\sharp$ on $TM/{\Lambda^\sharp}$ given by
\[
G_\sharp(\act(X_1)+\widetilde{Y_1},\act(X_2)+\widetilde{Y_2})=g(X_1,X_2)+g(Y_1,Y_2),
\]
where $\widetilde{X}$ denotes the horizontal lift of $X\in TM$. These define a K\"ahler structure, which corresponds to the above K\"ahler structure on $\T_\Lambda$ via the torus bundle isomorphism: 
\begin{equation}\label{eqn:comptosymptorbuniso}
    TM/{\Lambda^\sharp}\xrightarrow{\sim} \T_\Lambda
\end{equation} induced by $g$ (see, e.g., \cite{ABCDGMSSW}). It follows from \cite[Proposition 3.5]{Shima2007} that the Ricci tensor $R_\sharp$ of $G_\sharp$ is given by:
\[
R_\sharp(\act(X_1)+\widetilde{Y_1},\act(X_2)+\widetilde{Y_2})=-\left(\beta_\sharp(X_1,X_2)+\beta_\sharp(Y_1,Y_2)\right),
\] with $\beta_\sharp$ the second Koszul form of $(M,\Lambda^\sharp,g)$. From this it readily follows that the scalar curvature $S_G$ (which  equals that of $G_\sharp$, since \eqref{eqn:comptosymptorbuniso} is an isometry) is equal to $-\textrm{Tr}_{g}(\beta_\sharp)\circ \mu$, which proves the first equality in \eqref{eqn:scalcurvexpr:lagrfib} in view of \eqref{eqn:secondkoszul:dual}. The second equality follows from Remark \ref{rem:koszform:scalcurv:coordinates}.

A straightforward computation using action-angle coordinates and expression \eqref{eq:metric:action:angle} for $G$ shows that for a function $f$, depending only on the action coordinates, the Hamiltonian vector field $X_f$ is a Killing vector field for $G$ if and only if $f$ is affine. This yields the equivalence of the extremal condition \eqref{eq:extremal:symp} with the condition for $S_G$.
\end{proof}

\subsection{Examples}

We will now illustrate how one can use Theorems \ref{thm:kahlmetrlagrfib} and \ref{thm:lagrfib:metrics} to construct extremal K\"ahler metrics on Lagrangian fibrations.

\begin{example}
\label{ex:complextori3}
 Any metric on $(\mathbb{S}^1,b\Z\d x)$, with $b>0$, is Hessian (for dimensional reasons) and of the form $f\d x^2$, for $f\in C^\infty(\mathbb{S}^1)$ strictly positive. It follows from \eqref{eqn:scalcurvexpr:lagrfib} that the scalar curvature of any invariant K\"ahler metric $G$ on the complex torus that induces this Hessian metric is given by
\begin{equation}\label{eqn:scalcurv:intaffcircle:example}
S_G=\frac{f''}{2f^2}-\frac{(f')^2}{f^3}.
\end{equation}
Since $\Ss^1$ is compact, if $S_G$ is affine, then it is in fact constant. Since $f$ is strictly positive and periodic, it follows that this constant must then be zero. The only solutions to the ODE
\[
ff''-2(f')^2=0,
\] are the constant functions on $\mathbb{S}^1$. Therefore, we conclude that the only extremal metrics are those for which $f$ is constant. 

As we saw in Example \ref{ex:complextori1}, any Lagrangian fibration inducing $(M,\Lambda)$ takes the form
\[ 
\pr_{2}:(\mathbb{S}^1\times\mathbb{S}^1,b\,\d y\wedge\d x)\to \mathbb{S}^1.
\] If we fix the Hessian metric $g=b^2\,\d x^2$ and choose any of the Lagrangian connections $D_a$ in that example, the resulting extremal K\"ahler metric on $S=\Tt^2$ is the flat metric
\[ G=(a^2+b^2)\,\d x^2+2a\,\d x\,\d y+\d y^2. \]
The $2$-torus equipped with corresponding complex structure is biholomorphic to the complex torus $\C/\langle 1,\tau\rangle$, where $\tau=a+ib$.
\end{example}

\begin{example} 
Consider the Lagrangian fibration $\pr_1:(\Tt^2\times\Tt^2,\omega_a)\to \Tt^2$ of Examples \ref{ex:trivial:chern:class} and \ref{ex:trivial:chern:class:2}, where $a\in \R$ and $\omega_a=\omega_\can+a\d x^1\wedge \d x^2$. This admits the flat Lagrangian connection
\[
D_a=\langle \partial_{x_1},\partial_{x_2}+a\partial_{\theta_1} \rangle.
\]
If we consider the flat Hessian metric $g=\d x^1\d x^1+\d x^2\d x^2$ on $\Tt^2$, we obtain from Theorem \ref{thm:kahlmetrlagrfib} the $\omega_a$-compatible extremal K\"ahler metric 
\[
G_a=\d x^1\d x^1+(1+a^2)\d x^2\d x^2-2a\d x^2\d\theta_1+\d \theta_1\d \theta_1+\d \theta_2\d \theta_2.
\]
\begin{align*}
J_a(\partial_{x^1})=\partial_{\theta_1}\ \ \quad &
J_a(\partial_{x^2})=a\partial_{x^1}+\partial_{\theta_2}\\
J_a(\partial_{\theta_1})=-\partial_{x^1}\quad &
J_a(\partial_{\theta_2})=-\partial_{x^2}-a\partial_{\theta_1}
\end{align*}
Notice that, when one varies $a\in [0,1[$, these give non-isomorphic K\"ahler metrics. In fact, the K\"ahler manifold $(\Tt^4,\omega_a,G_a,J_a)$ is isomorphic to the complex torus $\mathbb{C}^2/\Lambda_a$, where
\[
\Lambda_a=\langle(1,0),(i,0),(ia,1),(0,i) \rangle\subset \C^2,
\] via 
\[
\Tt^4\cong\C^2/\Lambda, \quad [x_1,x_2,\theta_1,\theta_2]\mapsto [x_1,\theta_1-ax_2,-x_2,-\theta_2].  
\]
\end{example}

\begin{example}\label{ex:trivspherebun:stdIA:cylinder:2}
Consider the standard integral affine cylinder 
\[
(M,\Lambda)=(\Ss^1\times \R,\Z\d x\oplus \Z\d h)
\]
of Example \ref{ex:trivspherebun:stdIA:cylinder}. Given a metric
\[
g=A\d h^2+2B\d h\d x+C\d x^2
\]
on this cylinder, We find that
\[
\d^\nabla g^\flat(\partial_h,\partial_x)=\left(\partial_hB-\partial_xA\right)\d h+\left(\partial_hC-\partial_xB\right)\d x.
\] Therefore, we have a family of Hessian metrics 
\[
g=\frac{1}{\tau(h)}\d h^2+\frac{1}{\epsilon(x)}\d x^2,
\] where both $\tau(h)$ and $\epsilon(x)$ are strictly positive. 

Note that $H^1(M,\TLag)=H^2(M,\mathcal{O}_\Lambda)=0$, so that any toric $\T_\Lambda$-space is isomorphic to the symplectic torus bundle over $(M,\Lambda)$ (see Example \ref{ex:trivspherebun:stdIA:cylinder}).
Applying \eqref{eqn:scalcurvexpr:lagrfib}, it follows that the scalar curvature of the lifted K\"ahler metric $G$ on this symplectic torus bundle is given by
\[
S_G=-\frac{1}{2}(\tau''(h)+\epsilon''(x)),
\] 
for any choice of flat Lagrangian connection. The condition for $G$ to be extremal, i.e., for $S_G$ to be affine, amounts to the  condition that $\epsilon(x)$ is constant and $\tau(h)$ is a cubic polynomial in $h$. Since $\tau(h)$ is strictly positive, this leads to the following two possibilities for $G$ to be extremal:
\begin{enumerate}[(i)]
\item $\tau(h)=bh^2+ch+d$ with $b>0$ and $c^2-4bd<0$, in which case $G$ has constant negative scalar curvature $-\tfrac{b}{2}$;
\item $\tau(h)$ is constant, in which case $G$ has zero scalar curvature. 
\end{enumerate}
\end{example}

\begin{example}\label{ex:nonstdcyl}
Consider the non-standard integral affine cylinder of Example \ref{ex:nontrivspherebun:nonstdIA:cylinder}. Given a general metric on $\Ss^1\times ]-2,\infty[$, written in the form
\[
g=A\,\d h^2+2B\,\d h\d x+C\,\d x^2.
\]
We find that
\[
\d^\nabla g^\flat(\partial_h,\partial_x)=\left(\partial_hB-\partial_xA+\tfrac{B}{h+2}\right)\d h+\left(\partial_hC-\partial_xB-\tfrac{C}{h+2} \right)\d x.
\] Therefore, we have a family of Hessian metrics 
\[
g=\frac{1}{\tau(h)}\d h^2+(h+2)\epsilon\,\d x^2,
\] 
where $\epsilon>0$ and $\tau(h)$ is a strictly positive smooth function on $]-2,\infty[$. 

Like in the previous example, $H^1(M,\TLag)=H^2(M,\mathcal{O}_\Lambda)=0$, so any toric $\T_\Lambda$-space is isomorphic to the symplectic torus bundle over $(M,\Lambda)$ (see Example \ref{ex:nontrivspherebun:nonstdIA:cylinder}). Applying \eqref{eqn:scalcurvexpr:lagrfib}, it follows that the scalar curvature of the lifted K\"ahler metric $G$ on this symplectic torus bundle is given by
\[
S_G=-\frac{1}{2}\tau''(h)-\frac{1}{h+2}\tau'(h),
\] 
for any choice of flat Lagrangian connection. Such a metric is extremal if 
\[ 
\frac{1}{2}\tau''(h)+\frac{1}{h+2}\tau'(h)=c_1(h+2)+c_0
\] for some constants $c_1$ and $c_0$. One finds that
\[
\tau(h)=\frac{c_1}{6}(h+2)^3+\frac{c_0}{3}(h+2)^2+\frac{c_2}{h+2}+c_3
\] 
where $c_2$ and $c_3$ are constants of integration. Therefore, choosing the constants such that $\tau(h)>0$, we obtain extremal K\"ahler metrics on the symplectic torus bundle. These include examples both with constant and non-constant scalar curvature.
\end{example}


\section{Invariant K\"ahler metrics II: toric Lagrangian fibrations}
\label{sec:inv:metrics:toric}

In this section we will generalize the characterization of compatible invariant K\"ahler metrics on Lagrangian fibrations given in Theorem \ref{thm:kahlmetrlagrfib} to toric $(\T,\Omega)$-spaces. 

Below, in Section \ref{sec:bmetrics}, we will introduce the notion of \textbf{hybrid $b$-metric} for manifolds with corners. These are metrics admitting a specific type of singularity at the boundary and with a well-defined notion of residue at the open facets. Using this notion, we can state the main result of this section as follows.

\begin{theorem}\label{thm:kahlmetrtoricsp}
Let $(M,\Lambda)$ be an integral affine manifold and let $\mu:(S,\omega)\to M$ be a toric $(\T_\Lambda,\Omega_\Lambda)$-space, with Delzant domain $\Delta:=\mu(S)$. Denote by $\mathring{\Delta}$ the interior of $\Delta$ and set $\mathring{S}:=\mu^{-1}(\mathring{\Delta})$. For any invariant K\"ahler metric $G$ on $S$ compatible with $\omega$, the following hold:
\begin{enumerate}[(i)]
    \item $(\ker\d\mu)^\perp|_{\mathring{S}}$ extends to an elliptic connection with zero radial residue over the open facets;
    \item The Hessian metric $g$ on $\mathring{\Delta}$, induced via Theorem \ref{thm:kahlmetrlagrfib}, extends to a hybrid $b$-metric on $\Delta$ with the residue at the open facets given by the primitive outward-pointing normals multiplied by $-\tfrac{1}{4\pi}$.
\end{enumerate} 
Conversely, given a flat Lagrangian elliptic connection $\theta$ with zero radial residues and a Hessian hybrid $b$-metric $g$ on $\Delta$ with residues as in (ii), there is a unique invariant compatible K\"ahler metric $G$ on $S$ inducing $g$ and such that $\ker\theta\vert_{\mathring{S}}=(\ker\d\mu)^\perp\vert_{\mathring{S}}$.
\end{theorem}

\begin{remark}
The scalar curvature $S_G$ of an invariant K\"ahler metric $G$ descends to a smooth function on $\Delta$, because it is $\T$-invariant. It follows from the case of smooth Lagrangian fibrations (by restricting to $\mathring{\Delta}$) that $G$ is an extremal K\"ahler metric if and only if $S_G:\Delta\to\R$ is an affine function.
\end{remark}

By the same reasoning as for Corollary \ref{cor:equivofflatfib:preservesKahler}, we conclude the following from Theorems \ref{thm:bi-lagrangian:classification2} and \ref{thm:kahlmetrtoricsp}. 

\begin{corollary}
    \label{cor:equivofflatfib:preservesKahler2}
    Given an integral affine manifold $(M,\Lambda)$ and a Delzant domain $\Delta$, there is a canonical 1:1 correspondence
   \[ 
\left\{\\ \txt{K\"ahler toric $\T$-spaces\\ $\mu:(S,\omega,G)\to M$ with\\ \,$\mu(S)=\Delta$, up to equivalence \,}\right\}\ 
\tilde{\longleftrightarrow}\ 
\left\{\\ \txt{Hessian hybrid \\ \,$b$-metrics on $\Delta$ with\,\\ \txt{residues as in (ii)}}\right\}\ \times \check{H}^1(\Delta,\TFlat)
\] 
\end{corollary}

\subsection{Hybrid $b$-metrics} 
\label{sec:bmetrics}

Let $\Delta$ be a manifold with corners. Given a smooth section
\begin{equation}
    g\in \Gamma^\infty(\tensor[^b]{T^*\Delta}{}\otimes T^*\Delta),
\end{equation} 
we let $g^\flat:\tensor[^b]{T\Delta}{}\to T^*\Delta$ denote the induced vector bundle map. We also denote by
\[
\widehat{g}\in \Gamma^\infty(\tensor[^b]{T^*\Delta}{}\otimes \tensor[^b]{T^*\Delta}{}) 
\]
the section induced by $g$ via the anchor map of $\tensor[^b]{T\Delta}{}$, so that $\widehat{g}^\flat=\rho^*\circ g^\flat$.

\begin{definition}\label{def:metricwithpoles} A smooth section $g$ of $\tensor[^b]{T^*\Delta}{}\otimes T^*\Delta$ will be called:
\begin{enumerate}[(i)]
    \item \textbf{non-degenerate} if $g^\flat_x$ is an isomorphism for each $x\in \Delta$; 
    \item \textbf{symmetric} if $\widehat{g}_x$ is symmetric for all $x\in \Delta$ (equivalently, for all $x\in \mathring{\Delta}$); 
    \item \textbf{positive semi-definite} if $\widehat{g}_x$ is positive semi-definite for all $x\in \Delta$ (equivalently, for all $x\in \mathring{\Delta}$).
\end{enumerate}
We will call $g$ a \textbf{hybrid $b$-metric} on $\Delta$ if it satisfies all these three properties. 
\end{definition}

A hybrid $b$-metric on $\Delta$ can be thought of as a smooth Riemannian metric on $\mathring{\Delta}$ with specified singularities at the boundary. In local coordinates, these singularities are of the same type as those of the singular metrics in \cite{Abr98}. Assume $\Delta$ is an open set in $\R^n_k$. A smooth section $g$ of $\tensor[^b]{T^*\Delta}{}\otimes T^*\Delta$ is symmetric if and only if it is of the form
\begin{equation}
\label{eqn:localform:hybrid-b-metric}
    g=\sum_{i=1}^k\frac{1}{x^i}g_{ii}\d x^i\otimes \d x^i+\sum_{(i,j)\in I_k} g_{ij}\d x^i\otimes \d x^j, 
\end{equation} 
for some $g_{ij}\in C^\infty(\Delta)$ such that $g_{ij}=g_{ji}$ for all $i,j\in \{1,\dots,n\}$, where we denote
\[
I_k:=\{(i,j)\in \{1,\dots,n\}^2\mid i\neq j \text{ or } i=j>k\}. 
\]
If $g$ is symmetric, then
\begin{itemize}\item positive semi-definiteness of $g$ means that for each $x\in \mathring{\Delta}$ the matrix $\mathring{g}(x)$ given by
\begin{equation}\label{eqn:coefmat:underlyingmetric}
    \mathring{g}_{ij}(x):=\begin{cases} 
    \frac{1}{x^i}g_{ii}(x)\quad&\text{ if } i=j\leq k, \\
    g_{ij}(x) \quad&\text{ if }(i,j)\in I_k,
    \end{cases} 
\end{equation} 
is positive semi-definite. 
\item non-degeneracy of $g$ at a point in $\Delta$ means that
\begin{equation}\label{eqn:detfncsingmetr} 
(x^1\cdot \dots \cdot x^k)\cdot \det(\mathring{g})\in C^\infty(\Delta)
\end{equation}
is non-zero at that point.  
\end{itemize} 
Therefore, if $g$ is symmetric and positive semi-definite, then it is non-degenerate if and only if (\ref{eqn:detfncsingmetr}) is strictly positive on all of $\Delta$. The latter condition is as in \cite{Abr98}. 

Assume now that $\Delta$ is an affine manifold with corners, i.e., $\Delta$ has an atlas for which the transition functions are restrictions of affine transformations. An \textbf{Hessian hybrid $b$-metric} $g$ on $\Delta$ is a hybrid $b$-metric satisfying
\[ \d^\nabla g^\flat=0, \]
where $\nabla$ is the associated flat $\tensor[^b]{T\Delta}{}$-connection on $T^*\Delta$. Such a $b$-metric $g$ admits local potentials around points in the boundary of $\Delta$ which are of a very particular form.

\begin{proposition}
\label{prop:portential:hybrid:metric}
Let $\Delta$ is an affine manifold with corners and $g$ a Hessian hybrid $b$-metric on $\Delta$. Around a point $x\in\partial \Delta$ of depth $k$ there exists a chart $(U,x^i)$ centered at $x$ for which
\[
    g=\sum_{i,j=1}^n \frac{\partial^2 \phi}{\partial x^i\partial x^j}\d x^i\otimes\d x^j
\]
where $\phi$ is a smooth function on $\mathring{\Delta}\cap U$ of the form
\begin{equation}\label{eqn:localpotential:hessian:hybridmetric}
    \phi=\sum_{i=1}^kc_ix^i\log(x^i)+f,\quad c_i:=g_{ii}(0)\in \R,\quad f\in C^\infty(U).
\end{equation} 
\end{proposition} 

\begin{proof} 
Fix an affine chart for $\Delta$ centered at $x$, onto an open set of the form $[0,\delta[^k\times ]-\delta,\delta[^{n-k}$. The fact that $g$ is Hessian means that the coefficients (\ref{eqn:coefmat:underlyingmetric}) of the Riemannian metric on $\mathring{\Delta}$ underlying $g$ satisfy the usual set of differential equations
\begin{equation}\label{eqn:1:hessianmetric:pde}
    \frac{\partial \mathring{g}_{ij}}{\partial x^l}=\frac{\partial{\mathring{g}_{il}}}{\partial x^j}, \quad 1\leq i,j,l\leq n.
\end{equation} 
In terms of the functions $g_{ij}\in C^\infty(U)$ as in (\ref{eqn:localform:hybrid-b-metric}) this means that
\begin{align}
    \frac{\partial g_{ii}}{\partial x^j}=x^i\frac{\partial g_{ij}}{\partial x^i}&\quad \text{if } i\leq k,\text{ } i\neq j, \label{eqn:2.1:hessianmetric:pde}\\
    \frac{\partial g_{ij}}{\partial x^l}=\frac{\partial g_{il}}{\partial x^j}&\quad \text{if } (i,j),(i,l)\in I_k. \label{eqn:2.2:hessianmetric:pde}
\end{align}
Since (\ref{eqn:2.1:hessianmetric:pde}) vanishes if $x^i=0$, for $i\in \{1,\dots,k\}$ the function
 \[
     x\mapsto g_{ii}(x^1,\dots,x^{i-1},0,x^{i+1},\dots,x^n)
 \] 
is constant on $U$ with value $c_i$. Since $g_{ii}(x)-c_i$ vanishes if $x^i=0$, $g_{ii}-c_{i}$ is of the form $x^ih_i$ for a smooth function $h_i\in C^\infty(U)$. The equations (\ref{eqn:2.1:hessianmetric:pde}) and (\ref{eqn:2.2:hessianmetric:pde}) then express the fact that the $1$-forms $\alpha_i$ on $U$ given by
 \[
     \alpha_i=
     \begin{cases} 
     h_i\d x^i+\sum_{j\neq i} g_{ij}\d x^j\quad &\text{if }1\leq i\leq k,\\ \\
     \sum_{j=1}^ng_{ij}\d x^j,\quad &\text{if }k<i\leq n,
     \end{cases}  
 \]
are closed. By the Poincar\'{e} lemma for manifolds with corners, each $\alpha_i$ must in fact be exact. So, there are $\phi_i\in C^\infty(U)$ such that $\d \phi_i=\alpha_i$. Applying the Poincar\'{e} lemma once more, this time to the closed $1$-form $\sum_{i=1}^n\phi_i\d x^i$ on $U$, we find an $f\in C^\infty(U)$ such that:
 \[
     \frac{\partial^2 f}{\partial x^i \partial x^j}=
     \begin{cases}
         h_i\quad &\text{if }i=j<k,\\
         \\
         g_{ij}\quad &\text{if }(i,j)\in I_k,
     \end{cases} 
 \]
 with $I_k$ as in (\ref{eqn:coefmat:underlyingmetric}). For this choice of $f$, the function $\phi$ defined by (\ref{eqn:localpotential:hessian:hybridmetric}) is indeed a potential for $g$ on $\mathring{\Delta}\cap U$. 
\end{proof}

Next, we introduce the notion of residue appearing in Theorem \ref{thm:kahlmetrtoricsp}.

\begin{definition} 
The \textbf{residue} of a hybrid $b$-metric $g$ at an open facet of $\Delta$ is the residue of the vector valued 1-form $g^\flat\in\Omega^1(\tensor[^b]{T\Delta}{},T^*\Delta)$ at the facet (see Definition \ref{def:residue:form}).
\end{definition}
\smallskip

\begin{example}
Consider a Delzant polytope
\[ \Delta:=\bigcap_{i=1}^d\{\ell_i\leq 0\}\subset \R^n. \]
Then $\mathring{\Delta}$ is contractible, so any Hessian hybrid $b$-metric on $\Delta$ admits a global potential $\phi\in C^\infty(\mathring{\Delta})$. By applying successively Proposition \ref{prop:portential:hybrid:metric}, one concludes that for a Delzant polytope a Hessian hybrid $b$-metric as in Theorem \ref{thm:kahlmetrtoricsp} (ii) has potential of the form
\[ \phi(x)=-\tfrac{1}{4\pi}\sum_{i=1}^d \ell_i(x)\log \ell_i(x) +f(x), \]
where $f$ is smooth on $\Delta$ (not just on $\mathring{\Delta}$). In view of this and \eqref{eqn:detfncsingmetr}, we recover the type of singular Hessian metrics on Delzant polytopes appearing in the works of Guillemin \cite{Guill94} and Abreu \cite{Abr98}.    
\end{example}

\begin{remark}
    In \cite{Abr98} Abreu also observed that the difference between any two of his singular Hessian metrics on a Delzant polytope extends smoothly over the boundary. This holds more generally for any two hybrid $b$-metrics on a manifold with corners $\Delta$ that have the same residues. 
\end{remark}

For the proof of Theorem \ref{thm:kahlmetrtoricsp}, it will be useful to also consider the notion dual to that of the singular metrics above. These will be smooth sections
\[
    \delta\in \Gamma^\infty(T\Delta\otimes\tensor[^b]{T\Delta}{} ).
\] 
Given such a section $\delta$,  we denote by $\delta^\sharp:T^*\Delta\to \tensor[^b]{T\Delta}{}$ the induced vector bundle map. We also denote by
\[
    \widehat{\delta}\in \Gamma^\infty(T\Delta\otimes T\Delta)
\] 
the section induced by $\delta$ via the anchor map of $\tensor[^b]{T\Delta}{}$, so that $\widehat{\delta}^\sharp=\rho\circ\delta^\sharp$.

\begin{definition}
\label{def:metricwithpoles:dual}
A smooth section $\delta$ of $T\Delta\otimes\tensor[^b]{T\Delta}{}$ will be called:
\begin{enumerate}[(i)]
    \item \textbf{non-degenerate} if $\delta^\sharp_x$ is an isomorphism for each $x\in \Delta$; 
    \item \textbf{symmetric} if $\widehat{\delta}_x$ is symmetric for all $x\in \Delta$ (equivalently, for all $x\in \mathring{\Delta}$); 
    \item \textbf{positive semi-definite} if $\widehat{\delta}_x$ is positive semi-definite for all $x\in \Delta$ (equivalently, for all $x\in \mathring{\Delta}$).
\end{enumerate}
We will call $\delta$ a \textbf{hybrid $b^*$-metric} on $\Delta$ if it satisfies all these three properties. 
\end{definition}

A hybrid $b^*$-metric has an underlying metric on $T^*\mathring{\Delta}$ that extends smoothly to the tensor $\widehat{\delta}$ on all of $\Delta$. Note, however, that $\widehat{\delta}$ is degenerate at the boundary. To see how, suppose that $\Delta$ is an open set around the origin in $\R^n_k$. Then a smooth section $\delta$ of $T\Delta\otimes\tensor[^b]{T\Delta}{}$ is symmetric if and only if it is of the form
    \begin{align}
        \delta=\sum_{i,j=1,j\neq i}^k& x^ix^j\delta^{ij}{\partial}_{x^i}\otimes {\partial}_{x^j}+\notag\\
        &+ \sum_{i=1}^kx^i\delta^{ii}{\partial}_{x^i}\otimes {\partial}_{x^i}
        +\sum_{i,j=k+1}^n\delta^{ij}{\partial}_{x^i}\otimes{\partial}_{x^j}+\label{eq:singmetrloccoord:dual}\\
        &\qquad \qquad +\sum_{i=1}^k\sum_{j=k+1}^n x^i\delta^{ij}\left({\partial}_{x^i}\otimes {\partial}_{x^j}+{\partial}_{x^j}\otimes{\partial}_{x^i}\right),\notag
    \end{align} 
    where 
    \[ \delta^{ij}\in C^\infty(\Delta), \quad \delta^{ij}=\delta^{ji}\ (1\leq i,j\leq n).\] 
    For positive semi-definiteness and non-degeracy, there are criteria analogous to those for hybrid $b$-metrics. 

On a given manifold with corners, hybrid $b$-metrics are in bijective correspondence with hybrid $b^*$-metrics, in which a hybrid $b$-metric $g$ and its corresponding hybrid $b^*$-metric $\delta$ are related by the fact that $g^\flat$ and $\delta^\sharp$ are inverse to each other.

\subsection{Proof of Theorem \ref{thm:kahlmetrtoricsp}}

In this section we prove the two directions of the statement of Theorem \ref{thm:kahlmetrtoricsp}. We start by proving the following local result concerning invariant metrics in the local model of a toric $\T$-space discussed in Section \ref{subsec:toric:lagrangian:fibrations}.

\begin{lemma}
\label{lemma:smoothnessinvmetrics:localmodel} 
Let $(S_{k,n},\omega_{k,n})$ be a standard local model with the $\mathbb{T}^n$-action (\ref{eqn:torusaction:localmodel}), and let $U$ be an open set in $\Delta:=\R^n_k$. Then:
\begin{enumerate}[(i)]
\item A smooth $\mathbb{T}^n$-invariant symmetric tensor:
\[ 
   G:\mu_{k,n}^{-1}(U\cap \mathring{\Delta})\to T^*S_{k,n}\otimes T^*S_{k,n} 
\] extends smoothly to all of $\mu_{k,n}^{-1}(U)$ if and only if it takes the form
\begin{align*}
    G&=\sum_{j=1}^k \left(g_{r_j,r_j}(\d r_j)^2+r_j^3 g_{r_j,\phi_j}\d r_j\d\phi_j+r_j^2 g_{\phi_j,\phi_j}(\d\phi_j)^2\right)\\
    & +\sum_{j,l=1,j\not=l}^k \left(r_jr_lg_{r_j,r_l}\d r_j\d r_l+r_j r_l^2 g_{r_j,\phi_l}\d r_j\d\phi_l+r_j^2r_l^2 g_{\phi_j,\phi_l}\d\phi_j\d\phi_l\right)\\
    &\quad +\sum_{j,l=1}^{k,n-k} \left(r_jg_{r_j,x^l}\d r_j\d x^l+r_j g_{r_j,\theta_l}\d r_j\d\theta_l+r_j^2 g_{\phi_j, x^l}\d\phi_j\d x^l+r_j^2 g_{\phi_j,\theta_l}\d\phi_j\d\theta_l\right)\\
    &\quad \quad + \sum_{j,l=1}^{n-k} \left(g_{\theta_j,\theta_l}\d \theta_j\d \theta_l+g_{x^j,\theta_l}\d x^j\d\theta_l+g_{x^j,x^l}\d x^j\d x^l\right)
\end{align*}
where the coefficients functions are $\Tt^n$-invariant smooth functions on $\mu_{k,n}^{-1}(U)$ and $g_{\phi_j,\phi_j}=(2\pi)^2 g_{r_j,r_j}$ over $U\cap \{x^j=0\}$.

    \item Dually, a smooth $\mathbb{T}^n$-invariant symmetric tensor
    \[ 
   G^*:\mu_{k,n}^{-1}(U\cap \mathring{\Delta})\to TS_{k,n}\otimes TS_{k,n} 
    \] 
    extends smoothly to all of $\mu_{k,n}^{-1}(U)$ if and only if its coefficients with respect to the coframe on $\mathring{S}$
    \[
    \d\phi_1,\dots,\d\phi_k,\frac{1}{r_1}\d r_1,\dots,\frac{1}{r_k}\d r_k,\d\theta_1,\dots,\d\theta_{n-k},\d x^1,\dots,\d x^{n-k}
    \] 
    are $\Tt^n$-invariant smooth functions on $\mu_{k,n}^{-1}(U)$, with the exception of the following coefficients, which take the form
    \[ G^*\left(\frac{1}{r_j}\d r_j,\frac{1}{r_j}\d r_j\right)=\frac{1}{r_j^2}g^*_{r_j,r_j},\quad G^*\left(\d\phi_j,\d\phi_j\right)=\frac{1}{r_j^2} g^*_{\phi_j,\phi_j},\]
    with $g^*_{r_j,r_j}$ and $g^*_{\phi_j,\phi_j}$ $\Tt^n$-invariant smooth functions on $\mu_{k,n}^{-1}(U)$ such that $g^*_{\phi_j,\phi_j}=(2\pi)^2 g^*_{r_j,r_j}$ over $U\cap \{x^j=0\}$.
\end{enumerate}
\end{lemma}

    \begin{proof} 
    The backward implication in (i) readily follows by expressing the coefficients of $G$ with respect to the frame:
    \[
        \partial_{u_1},\partial_{v_1},\dots,\partial_{u_k},\partial_{v_k}, \partial_{\theta_1},\dots,\partial_{\theta_{n-k}}, \partial_{x^1},\dots,\partial_{x^{n-k}}
    \] on $U$ in terms of the coefficients with respect to the frame \eqref{eqn:basis:elltngtbun} on $\mathring{U}$. For the forward implication, suppose that $G$ extends smoothly to all of $\mu^{-1}_{k,n}(U)$. By continuity this extension is $T$-invariant as well, hence so are its coefficient functions with respect to the frame \eqref{eqn:basis:elltngtbun}. Applying Proposition \ref{prop:invsmoothfunctions:localmodel} to these, it follows that the coefficients in the last 3 lines of the formula for $G$ are of the desired form. Moreover, it follows that there are $\Tt^n$-invariant smooth functions $g_{\phi_j,\phi_j},g_{r_j,r_j},\widetilde{g}_{\phi_j,r_j}$ on $\mu^{-1}_{k,n}(U)$ such that:
    \begin{align*}
            G\left(\partial_{\phi_j},\partial_{\phi_j}\right)&=r_j^2g_{\phi_j,\phi_j}, \\
            G\left(r_j\partial_{r_j},r_j\partial_{r_j}\right)&=r_j^2g_{r_j,r_j}, \\
            G\left(\partial_{\phi_j},r_j\partial_{r_j}\right)&=r_j^2\widetilde{g}_{\phi_j,r_j}.
        \end{align*}
    Denote the smooth extension of $G$ to $\mu^{-1}_{k,n}(U)$ by $\widehat{G}$. Comparing the limits of the function $\widehat{G}(\partial_{u_j},\partial_{v_j})$ along the two paths: 
    \begin{align*} &\R\ni \varepsilon \mapsto (z_1,\dots,z_{j-1},\varepsilon,z_{j+1},\dots,z_k,t,x)\\
    &\R\ni \varepsilon \mapsto (z_1,\dots,z_{j-1},i\varepsilon,z_{j+1},\dots,z_k,t,x)
    \end{align*} 
    as $\varepsilon$ tends to zero (for any given $j\in \{1,\dots,k\}$ and $(z,t,x)\in U\cap \{z_j=0\}$) leads to the conclusion that $\widetilde{g}_{\phi_j,r_j}$ vanishes over $U\cap \{x^j=0\}$, so $\widetilde{g}_{\phi_j,r_j}=r_j^2{g}_{\phi_j,r_j}$. Computing the same limits for the function $\widehat{G}(\partial_{u_j},\partial_{u_j})$ shows that $g_{\phi_j,\phi_j}$ and $(2\pi)^2\cdot g_{r_j,r_j}$ coincide over $U\cap \{x^j=0\}$, which completes the proof of (i). Part (ii) follows from the same type of arguments.
    \end{proof}

\subsubsection{From invariant (K\"ahler) metrics to hybrid $b$-metrics}

We will now prove one direction of Theorem \ref{thm:kahlmetrtoricsp}. Let $\mu:(S,\omega)\to M$ be a toric $(\T_\Lambda,\Omega_\Lambda)$-space with Delzant domain $\Delta:=\mu(S)$. Further, let $G$ be an $\omega$-compatible invariant K\"ahler metric on $S$. As in the proof of Theorem \ref{thm:kahlmetrlagrfib}, the $\T$-invariance of $G$ implies that there is a unique symmetric, positive semi-definite section $\widehat{\delta}\in \Gamma^\infty(T\Delta\otimes T\Delta)$ making the following diagrams commute
\begin{equation}\label{eqn:dualmetric:Delzspace:degenerate}
\vcenter{\xymatrix{
TS\ar[r]^{G^\flat} & T^*S\ar[d]^{\act^*} & & T^*S\ar[r]^{{(G^\flat)}^{-1}} & TS\ar[d]^{\d\mu}\\
\mu^*(T^*\Delta)\ar[u]^{\act}\ar[r]_{\widehat{\delta}^\sharp} & \mu^*(T\Delta)& & \mu^*(T^*\Delta)\ar[u]^{\d\mu^*}\ar[r]_{\widehat{\delta}^\sharp} & \mu^*(T\Delta)} 
}
\end{equation}
The tensor $\widehat{\delta}$ extends the metric on $T^*\mathring{\Delta}$ dual to the unique Riemannian metric on $g$ on $\mathring{\Delta}$ making $\mu:(\mathring{S},G)\to (\mathring{\Delta},g)$ a Riemannian submersion. For each locally defined $1$-form $\alpha$ on $\Delta$ the vector field $(G^\flat)^{-1}(\mu^*\alpha)$ is $\T$-invariant, since both $\mu^*\alpha$ and $G$ are. Therefore, $(G^\flat)^{-1}\circ (\d\mu)^*$ lifts to a unique vector bundle map:
\begin{equation}\label{eqn:lifttransverseinfact:metric}
\vcenter{\xymatrixcolsep{5pc}\xymatrix{
 & \tensor[^\mu]{TS}{}\ar[d]^{\rho}\\
\mu^*(T^*\Delta)\ar[r]_{(G^\flat)^{-1}\circ (\d\mu)^*}\ar@{-->}[ru]^{h} & TS}}
\end{equation}
From this and the commutativity of the right-hand square in (\ref{eqn:dualmetric:Delzspace:degenerate}), it follows that $\widehat{\delta}^\sharp$ takes $1$-forms on $\Delta$ to $b$-vector fields. Therefore, $\widehat{\delta}$ lifts to a symmetric and positive semi-definite section $\delta\in\Gamma^\infty(T\Delta\otimes\tensor[^b]{T\Delta}{})$ making the following diagrams commute
\begin{equation}\label{eqn:dualmetric:Delzspace:hybrid}
\vcenter{
\xymatrix{
 & \tensor[^b]{T\Delta}{}\ar[d]^\rho & & & \tensor[^\mu]{TS}{}\ar[d]^{\mu_*}\\
T^*\Delta\ar[ru]^{\delta^\sharp}\ar[r]_{\widehat{\delta}^\sharp} & T\Delta & & \mu^*(T^*\Delta)\ar[ru]^{h}\ar[r]_{\delta^\sharp} & \mu^*(\tensor[^b]{T\Delta}{})}} 
\end{equation} 
with $\mu_*$ as in Proposition \ref{prop:elliptictngbun:liftmomentummap}. In what follows, we will show that $\delta$ is non-degenerate, so that it is a hybrid $b^*$-metric on $\Delta$. Its dual will then be the hybrid $b$-metric on $\Delta$ extending the Riemannian metric $g$, as in item (ii) in Theorem \ref{thm:kahlmetrtoricsp}. Moreover, by commutativity of the right-hand triangle in (\ref{eqn:dualmetric:Delzspace:hybrid}), the composite
\begin{equation}
\label{eqn:extension:horizontallift}
\xymatrix{
\mu^*(\tensor[^b]{T\Delta}{})\ar[r]^{(\delta^\sharp)^{-1}}& \mu^*(T^*\Delta)\ar[r]^{h} & \tensor[^\mu]{TS}{}}
\end{equation} 
will be a section of $\mu_*$, which gives the extension of $(\ker\d\mu)^\perp|_{\mathring{S}}$ to the elliptic connection $\theta$ in item (i) in Theorem \ref{thm:kahlmetrtoricsp}.
\smallskip

It remains to prove that:
\begin{enumerate}[(i)]
    \item $\delta$ is non-degenerate at each $x\in \partial\Delta$;
    \item at each open facet the residue of the hybrid $b$-metric dual to $\delta$ is the primitive inward-pointing normal multiplied by $\frac{1}{4\pi}$;
    \item $\theta$ has zero radial residue. 
\end{enumerate}
{ Fix $x_0\in\partial\Delta$ of depth $k$. In coordinates centered at $x_0$ the  smooth section $\delta$ of $T\Delta\otimes \tensor[^b]{T\Delta}{}$ takes the form \eqref{eq:singmetrloccoord:dual}. At the point $x_0$ we have $x^1=\cdots=x^k=0$, so we can write the map $\delta^\sharp_{x_0}:T_{x_0}^*\Delta\to \tensor[^b]{T_{x_0}\Delta}{}$ as a matrix of the form
\begin{equation}
\label{eqn:matrix:deltasharp}
\begin{pmatrix}
 D_{k} & \rvline & D_{n-k,k} \\
\hline 
0 & \rvline & D_{n-k}
\end{pmatrix}, 
\end{equation}
with respect to standard bases 
\begin{align*}
    &\{\d x^1,\dots,\d x^n\}\text{ for }T^*_{x_0}\Delta,\\
    &\{x^1\partial_{x^1},\dots,x^k\partial_{x^k},\partial_{x^{k+1}},\dots,\partial_{x^n}\}\text{ for }\tensor[^b]{T_{x_0}\Delta}{},
\end{align*} 
where $D_{k}$ is the diagonal $k$ by $k$ matrix with $i^\textrm{th}$ diagonal entry $\delta^{ii}(x_0)$, $D_{n-k}$ is the $n-k$ by $n-k$ matrix with $(i,j)$-entry $\delta^{k+i,k+j}(x_0)$, and similarly for $D_{n-k,k}$. In standard toric coordinates $(z_i,\theta_j,x^j)$ centered at $x_0$  (see Section  \ref{subsec:toric:lagrangian:fibrations}), the vector fields associated to the standard coframe on $\R^n$ by the infinitesimal action $\act$ are} 
\[  
\partial_{\phi_1},\dots,\partial_{\phi_k}, \partial_{\theta_1},\dots,\partial_{\theta_{n-k}}.
\] 
It follows from part (i) of Lemma \ref{lemma:smoothnessinvmetrics:localmodel} and the left diagram in (\ref{eqn:dualmetric:Delzspace:degenerate}) that
\begin{equation}\label{eqn:deltacomputation:left-hand}
\delta^{ij}=\begin{cases}g_{\phi_i,\phi_j}\quad &\text{ if }i,j\leq k,\\
g_{\phi_i,\theta_{j-k}}\quad &\text{ if }i\leq k,\text{ }j\geq k+1,\\
   g_{\phi_j,\theta_{i-k}}\quad &\text{ if }i\geq k+1,\text{ }j\leq k,\\ 
   g_{\theta_{i-k},\theta_{j-k}}\quad &\text{ if }i\geq k+1,\text{ }j\geq k+1.
\end{cases}
\end{equation} From this and the fact that $G$ is positive definite, it follows that the functions $\delta^{ii}$ are strictly positive on $\Delta$ and the matrix (\ref{eqn:matrix:deltasharp}) is invertible. This proves item (i). 

For items (ii) and (iii) we can further assume that the depth of $x_0$ is $1$, i.e., $k=1$. For item (ii), we are left to show that $\delta^{11}(x_0)=4\pi$. Since
\[
    2r\d r,\d x^1,\dots,\d x^{n-1},
\] are the $1$-forms obtained by pulling back the standard coframe on $\R^n$ along $\mu$, it follows from part (ii) of Lemma \ref{lemma:smoothnessinvmetrics:localmodel} and the right diagram in (\ref{eqn:dualmetric:Delzspace:degenerate}) that 
\begin{equation}
\label{eqn:deltacomputation:right-hand} 
\delta^{ij}=\begin{cases}4\, g^*_{r,r},&\text{ if }i=j=1,\\
2\, g^*_{r,x^{j-1}}\quad &\text{ if }i=1,\ j\geq 2,\\
  2\, g^*_{r,x^{i-1}}\quad &\text{ if } i\geq 2,\ j=1,\\ 
   g^*_{x^{i-1},x^{j-1}}\quad &\text{ if }i,j\geq 2.
\end{cases} 
\end{equation}
By considering the coefficient functions of $g$ with respect to the frame
\[
    \frac{1}{r}\partial_{\phi},\partial_{r},\partial_{\theta_1},\dots,\partial_{\theta_{n-1}},\partial_{x^1},\dots,\partial_{x^{n-1}}
\] 
over $\mathring{S}$ and those of $g^*$ with respect to the dual coframe
\[ 
    r\d\phi,\d r,\d\theta_1,\dots,\d\theta_{n-1},\d x^1,\dots,\d x^{n-1}
\] 
over $\mathring{S}$, we obtain mutually inverse matrices of functions that converge as $z$ tends to zero. Comparing the two  limits as $z$ tends to zero leads to the conclusion that
\[
    g^*_{r,r}(x_0)\, g_{r,r}(x_0)=1.
\] 
By Lemma \ref{lemma:smoothnessinvmetrics:localmodel}, we also have that $g_{\phi,\phi}(x_0)=(2\pi)^2\, g_{r,r}(x_0)$. So using (\ref{eqn:deltacomputation:left-hand}) and (\ref{eqn:deltacomputation:right-hand}), it follows that $\delta^{11}(x_0)=4\pi$.

To prove item (iii), we ought to show that, for any $b$-vector field $\nu$ that is Euler-like with respect to an open facet $F$, the lift of $2\, \nu$ along \eqref{eqn:extension:horizontallift} is Euler-like with respect to $S_F$. In view of Remark \ref{rem:eulerlikevf:manwithcorncoord}, we can assume that $\nu=x^1\partial_{x^1}$. The lift of the $b$-vector field $2x^1\partial_{x^1}$  is of the form
\[
  \widetilde{2x^1\partial_{x^1}}=r\partial_{r}+f^\phi\partial_{\phi} +\sum_{j=1}^{n-1} f^\theta_j\partial_{\theta_j},
\] where $f^\phi$ and $f^\theta_j$ are $\Tt^n$-invariant smooth functions on $\mu_{1,n}^{-1}(U)$. By Remark \ref{rem:eulerlikevf:toriccoord}, it is enough to show that $f^\phi(x)=f^\theta_1(x)=\cdots=f^\theta_{n-1}(x)=0$ whenever $x^1=0$. Note that over $\mathring{S}$ one has
\begin{align*} 
0&=g\left(2\widetilde{x^1\partial_{x^1}},\partial_{\theta_l}\right)=r^2\,g_{r,\theta_l}+g_{\phi,\theta_l}\, f^\phi+\sum_{j=1}^{n-1}g_{\theta_j,\theta_l}\, f^\theta_j,\\
0&= \frac{1}{r^2}\, g\left(2\widetilde{x^1\partial_{x^1}},\partial_{\phi}\right)=r^2\, g_{\phi,r}+g_{\phi,\phi}\, f^\phi+\sum_{j=1}^{n-1}g_{\phi,\theta_j}\, f^\theta_j,
\end{align*}
where we first we used that $\ker(\d\mu)$ and $\mathcal{L}$ are orthogonal over $\mathring{S}$, and then we applied Lemma \ref{lemma:smoothnessinvmetrics:localmodel}. In view of (\ref{eqn:deltacomputation:left-hand}), letting $z$ tend to zero it follows that $(f^\phi(x),f_1^\theta(x),\dots,f_{n-1}^\theta(x))$ belongs to the kernel of the invertible matrix \eqref{eqn:matrix:deltasharp} whenever $x^1=0$. So,  $f^\phi(x)=f^\theta_1(x)=\dots=f^\theta_{n-1}(x)=0$ if $x^1=0$.  


    This concludes the proof of one direction of Theorem \ref{thm:kahlmetrtoricsp}.
    
    \subsubsection{From hybrid $b$-metrics to invariant (K\"ahler) metrics}
    
    To prove the other direction of Theorem \ref{thm:kahlmetrtoricsp}, let $\mu:(S,\omega)\to M$ be a toric $(\T_\Lambda,\Omega_\Lambda)$-space with Delzant domain $\Delta:=\mu(S)$. Further, let $\theta$ be an elliptic connection and let $g$ be a hybrid $b$-metric on $\Delta$ as in the statement of the theorem. We use $\delta$ to denote the corresponding hybrid $b^*$-metric and we let $G$ be the unique invariant $\omega$-compatible K\"ahler metric on $\mathring{S}$ such that $\mu:(\mathring{S},G)\to (\mathring{\Delta},g)$ is a Riemannian submersion and $\ker\theta\vert_{\mathring{S}}=(\ker\d\mu\vert_{\mathring{S}})^\perp$. By the proof of Theorem \ref{thm:kahlmetrlagrfib}, this data satisfies \eqref{eqn:dualmetric:Delzspace:degenerate} on $\mathring{S}$. We ought to show that $G$ extends to a Riemannian metric on all of $S$. This extension is then automatically invariant and compatible with $\omega$, due to density of $\mathring{S}$.
    
    First, we will show that $G$ extends smoothly to $S$ as a tensor. Since this is a local property, we can verify this in standard toric coordinates $(z,\theta,x)$ by showing that the coefficients of $G$ are of the form in part (i) of Lemma \ref{lemma:smoothnessinvmetrics:localmodel}. Because the left square in \eqref{eqn:dualmetric:Delzspace:degenerate} commutes, the coefficients of $G$ involving only 
    \[ 
    \partial_{\phi_1},\dots,\partial_{\phi_k},\partial_{\theta_1},\dots,\partial_{\theta_{n-k}}
    \]
    are of this form, with the required $\Tt^n$-invariant smooth functions on $\mu_{k,n}^{-1}(U)$ given by the pull-backs of the corresponding coefficients $\delta^{ij}$ of $\delta$, as in \eqref{eq:singmetrloccoord:dual}. For the coefficients of $G$ involving $r_i\partial_{r_i}$, notice that the horizontal lift of $2x^i{\partial}_{x^i}$ must be of the form
    \begin{equation}
    \label{eqn:horizontallift:vanishingcoeff}
        \widetilde{2x^i{\partial}_{x^i}}=r_i\partial_{r_i}+r_i^2\,\Big(\sum_{j=1}^kf^\phi_{i,j}\partial_{\phi_j}+\sum_{j=1}^{n-k}f_{i,j}^\theta\partial_{\theta_j}\Big),
    \end{equation}
    where $f^\phi_j$ and $f^\theta_j$ are $\Tt^n$-invariant smooth functions on $\mu_{k,n}^{-1}(U)$. Indeed, this follows from the condition in (ii) in Theorem \ref{thm:kahlmetrtoricsp}, Remarks \ref{rem:eulerlikevf:manwithcorncoord} and \ref{rem:eulerlikevf:toriccoord}, and Proposition \ref{prop:invsmoothfunctions:localmodel}. Since the horizontal lift is orthogonal to $\ker\d\mu$, it follows that the coefficients $G\left(r_i\partial_{r_i},\partial_{\phi_l}\right)$ and $G\left(r_i\partial_{r_i},\partial_{\theta_l}\right)$ are of the required form, and that
    \begin{equation}
    \label{eqn:proof:fromhybridtokahlermetric:1}
    G\left(\widetilde{2x^i{\partial}_{x^i}},\widetilde{2x^l\partial_{x^l}}\right)=G\left(\widetilde{2x^i{\partial}_{x^i}},r_l\partial_{r_l}\right).
    \end{equation} 
   for $i,l\leq k$. Since $\mu$ is a Riemannian submersion, it holds that:
   \[
    G\left(\widetilde{2x^i{\partial}_{x^i}},\widetilde{2x^l\partial_{x^l}}\right)=\begin{cases} r_i^2r_l^2\mu^*(g_{il}) &\text{ if }i\neq l,\\
           r_i^2\mu^*(g_{ii}) &\text{ if }i=l,
       \end{cases}
   \]
    with $g_{il}\in \mathcal{C}^\infty(U)$ as in \eqref{eqn:localform:hybrid-b-metric}. So, it follows from \eqref{eqn:proof:fromhybridtokahlermetric:1} that the coefficients $G\left(r_i\partial_{r_i},r_l\partial_{r_l}\right)$ are also of the required form, and that   
    \[ 
   g_{\phi_i,\phi_i}=\mu_{k,n}^*(\delta^{ii})\quad\text{and}\quad g_{r_i,r_i}=\mu^*_{k,n}(g_{ii})\quad \text{over}\quad x^i=0.
    \]
   Since $g^\flat$ and $\delta^\sharp$ are inverse to each other, it holds that $g_{ii}(x)\delta^{ii}(x)=1$ when $x^i=0$. By the assumption on the residue of $g$, one has that $g_{ii}(x)=\frac{1}{4\pi}$ when $x^i=0$. Hence, it follows that
    \[ 
    g_{\phi_i,\phi_i}=(2\pi)^2g_{r_i,r_i}\quad \text{over}\quad x^i=0. 
    \]
    By similar arguments, the other coefficients are of the required form as well. So, the tensor $G$ indeed extends smoothly to all of $S$. This extension is symmetric and positive semi-definite, since it is so on the dense subset $\mathring{S}$. So, to conclude that it is a Riemannian metric on $S$, it suffices to show that the dual metric $G^*$ on $\mathring{S}$ extends smoothly, as a tensor, to all of $S$. This follows from an argument along the same lines as that for $G$, by instead using part (ii) of Lemma \ref{lemma:smoothnessinvmetrics:localmodel} and the coframe over $\mathring{S}\cap\mu^{-1}(U)$ dual to the frame:
    \[
        \widetilde{2x^1\partial_{x^1}},\dots,\widetilde{2x^k\partial_{x^k}},\widetilde{\partial_{x^{k+1}}},\dots,\widetilde{\partial_{x^{n}}},\partial_{\phi_1},\dots,\partial_{\phi_k},\partial_{\theta_1},\dots,\partial_{\theta_{n-k}}.
    \]

This concludes the proof of Theorem \ref{thm:kahlmetrtoricsp}.

\subsection{Examples}

\begin{example} 
\label{ex:trivspherebun:stdIA:cylinder:3}
Consider the integral affine cylinder $(M,\Lambda)=(\mathbb{S}^1\times \R,\Z\d x\oplus \Z\d h)$ and the Delzant domain $\Delta=\mathbb{S}^1\times [-1,1]$ as in Example \ref{ex:trivspherebun:stdIA:cylinder:1}. Further, consider the corresponding canonical toric $\T_\Lambda$-space $\mu:(\Tt^2\times \mathbb{S}^2,\omega)\to \Ss^1\times\R$ described in Example \ref{ex:trivspherebun:stdIA:cylinder}. As in Example \ref{ex:trivspherebun:stdIA:cylinder:2}, there are Hessian hybrid $b$-metrics of the form 
\[
g=\left(\frac{1}{2\pi(1-h^2)}+f(h)\right)\d h^2+c(x)\,\d x^2,
\] 
where now the singular term is so that the residue at both facets is $\tfrac{1}{4\pi}$. By a computation as in Example \ref{ex:trivspherebun:stdIA:cylinder:2}, the extremal condition holds if and only if $f=0$ and $c(x)$ is constant, {in which case the scalar curvature is 
\[
S_G=2\pi.
\]
}Using the flat Lagrangian elliptic connection $D_{a,b}$ of Example \ref{ex:trivspherebun:stdIA:cylinder:1}, Theorem \ref{thm:kahlmetrtoricsp} gives the family of $\omega$-compatible extremal K\"ahler metrics 
\begin{equation}
    \label{eq:metric:trivspherebun}
    G=\left(\frac{1}{\tau(h)}+\frac{b^2}{c}\right)\d h^2+\tau(h)(b\d x-\d \phi)^2+\frac{2b}{c}(a\d x-\d y)\d h+g_{\Tt^2},
\end{equation}
where $\tau(h)=2\pi(1-h^2)$ and $g_{\Tt^2}$ is the flat metric on the torus given by
\[ g_{\Tt^2}=\frac{1}{c}\big((a^2+c^2)\d x^2-2a\d x\d y+ \d y^2\big). \]
Note that when $a=b=0$ and $c=1$ we recover the product of the round metric on $\Ss^2$ with the standard flat metric on $\Tt^2$. 
\end{example}

\begin{remark} {The underlying complex manifolds of the K\"ahler manifolds in Example \ref{ex:trivspherebun:stdIA:cylinder:3} can also be described as follows. Consider the complex torus $\C/\Lambda$, with $\Lambda=\Z\oplus (-a+ic)\Z$ for $a,c$ as in the Example \ref{ex:trivspherebun:stdIA:cylinder:3}. Further consider, for $\lambda\in \C^\times$, the ruled surface $X_{\lambda}:=(\C/\Z)\times_\Z \C\mathbb{P}^1$ over $\C/\Lambda$, where the $\Z$-action on $(\C/\Z)\times\C\mathbb{P}^1$ is generated by
\[
([z],[w_0:w_1])\mapsto ([z-a+ic],[w_0:\lambda w_1]).
\] The $X_\lambda$ form a non-trivial family of complex manifolds (see, e.g., \cite{Suwa69}). For each $b
\in \R$, the complex manifold $(\Tt^2\times \Ss^2,J:=J_{a,b,c})$ in Example \ref{ex:trivspherebun:stdIA:cylinder:3} is part of this family: it is biholomorphic to $X_\lambda$, for $\lambda=e^{2\pi ib}$, via the map
\[
\Tt^2\times \Ss^2\to X_\lambda,\quad (y,x,\phi,h)\mapsto \big[[y+(-a+ic)x-bh],\Phi(\phi-bx,h)\big],
\] where $\Phi:\Ss^2\to\C\mathbb{P}^1$ denotes the diffeomorphism $(u,v,h)\mapsto [\tfrac{u+iv}{1+h}:1]$.
}

\end{remark}

\begin{example} 
\label{ex:nontrivspherebun:nonstdIA:cylinder:3}
Let $(M,\Lambda)=(\mathbb{S}^1\times ]-2,\infty[,\Z((h+2)\d x+x\d h)\oplus \Z\d h)$ be  the non-standard integral affine cylinder and consider the Delzant domain $\Delta=\mathbb{S}^1\times [-1,1]$ (see Example \ref{ex:nontrivspherebun:nonstdIA:cylinder:1}). Further, we consider the corresponding canonical toric $\T_\Lambda$-space $\mu:(\Tt^2\twprod \mathbb{S}^2,\omega)\to \Ss^1\times]-2,\infty[$, which is described in Example \ref{ex:nontrivspherebun:nonstdIA:cylinder}. Similar to Example \ref{ex:nonstdcyl}, there are Hessian hybrid $b$-metrics of the form 
\[
g=\left(\frac{1}{2\pi(1-h^2)}+f(h)\right)\d h^2+ (h+2)c\,\d x^2,
\]
with $c>0$ a constant and $f\in C^\infty(]-2,\infty[)$. Let us define a function $\tau(h)$ by
\[
\frac{1}{\tau(h)}:=\frac{1}{2\pi(1-h^2)}+f(h).
\]
One finds the boundary conditions
\[ \tau(\pm 1)=0,\quad \tau'(\pm 1)=\mp 4\pi. \]
Applying formula \eqref{eqn:scalcurvexpr:lagrfib}, one finds that
\[ 
S_G=-\frac{1}{2}\tau''(h)-\frac{\tau'(h)}{h+2}. 
\]
Proceeding as in Example \ref{ex:nonstdcyl}, taking into account the boundary conditions, one finds that $S_G$ is an affine function if and only if
\[
\tau(h)=-\frac{2\pi}{11}\left(2(h+2)^3-5(h+2)^2-15+\frac{18}{h+2}\right).
\]
This gives the affine function
\[ S_G=\frac{2\pi}{11}(12h+9).\]
The corresponding Hessian hybrid $b$-metric takes the form
\[
g=\left(\frac{1}{2\pi(1-h^2)}+\frac{1}{\pi(2h^2+11h+20)}\right)\d h^2+(h+2)c\,\d x^2.
\]
We can now use the family of elliptic Lagrangian connections $D_{a}$ on the $\T_\Lambda$-space $\mu:(\Tt^2\twprod \mathbb{S}^2,\omega)\to \Ss^1\times]-2,\infty[$ discussed in Example \ref{ex:nontrivspherebun:nonstdIA:cylinder:1} to lift $g$ to the family of $\omega$-compatible extremal K\"ahler metrics
\begin{equation}
    \label{eq:metric:nontrivspherebun}
    G=\frac{1}{\tau(h)}\d h^2+\tau(h)\theta^2+(h+2)g_{\Tt^2},
\end{equation}
where $\theta:=(x\d y+\d \phi)$ and $g_{\Tt^2}$ is the flat metric on the torus given by
\[ g_{\Tt^2}=\frac{1}{c}\big((a^2+c^2)\d x^2-2a\d x\d y+ \d y^2\big). \]
Notice the similarity between formulas \eqref{eq:metric:trivspherebun} and \eqref{eq:metric:nontrivspherebun} when $b=0$. The metric \eqref{eq:metric:nontrivspherebun} appears as a special case of the metrics considered by Apostolov et al.~\cite{ACGT08} (see the formula in Theorem 3, loc.~cit.). The precise relation between the approach followed here and the results in \cite{ACGT08}  will be discussed in a sequel to this paper.
\end{example}

\subsection{Invariant K\"ahler metrics via K\"ahler reduction}
\label{sec:invariant:metrics:finite:type}
The Delzant-type construction from Section \ref{sec:delzant} (see Theorem \ref{thm:Delzant}) can be used to construct invariant K\"ahler metrics via K\"ahler reduction, much like in \cite{Guill94}.

Let $(M,\Lambda)$ be a connected integral affine manifold and $\Delta\subset M$ a Delzant domain of finite type. Assume that $M$ admits a Hessian metric $g$ and let $\widetilde{g}$ be its lift to the universal covering space $(\widetilde{M},\widetilde{\Lambda})$. Let $G_{\widetilde{\Lambda}}$ be the $\Omega_{\widetilde{\Lambda}}$-compatible invariant K\"ahler metric on $\T_{\widetilde{\Lambda}}$ obtained by lifting $\widetilde{g}$ via the canonical Lagrangian 
connection. 

\begin{proposition}
\label{prop:reduced:metric}
    Let $(M,\Lambda,g)$ be a connected integral affine Hessian manifold and let $\Delta$ be a Delzant domain of finite type. Consider the product metric $G_{\widetilde{\Lambda}}\times g_\st$ on $\T_{\widetilde{\Lambda}}\times \C^d$. The Hessian hybrid $b$-metric $g_\Delta$ on $\Delta$ induced by the reduced K\"ahler metric on the symplectic quotient 
    \[
    ((\T_{\widetilde{\Lambda}}\times \C^d)\sslash(\Gamma\ltimes \Tt^d),\,\omega_\red) 
    \] 
    takes the form
    \[
    g_\Delta=g+\mathrm{Hess}_\Lambda(\phi),
    \] 
    where $\phi$ is the smooth function on $\mathring{\Delta}$ corresponding to the function on $\widetilde{M}$ given by 
    \begin{equation}
        \label{eq:canonical:potential}
        \widetilde{\phi}=-\frac{1}{4\pi}\sum_{i=1}^d\ell_i\log|\ell_i|.
    \end{equation}
\end{proposition}

The proof follows the usual K\"ahler reduction recipe: one restricts the metric to the zero level, obtaining a $(\Gamma\ltimes \Tt^d)$-invariant metric, that then descends to the symplectic quotient, yielding the expression in the theorem. Notice that, although each term in  \eqref{eq:canonical:potential} is a function on $\widetilde{M}$ which may fail to descend to $\mathring{\Delta}$, the sum does descend to a smooth function, since the $\Gamma$-action permutes the primitive boundary defining functions.

\begin{remark}
For a Delzant polytope, Guillemin showed in \cite{Guill94} that, if one applies K\"ahler reduction to the standard Delzant construction, one obtains an invariant K\"ahler metric on the symplectic toric manifold which induces a Hessian metric on $\mathring{\Delta}$ with potential given by \eqref{eq:canonical:potential}. Our Delzant type construction is different from the standard one (see Example \ref{ex:standard:Delzant}) and results in the presence  of the original Hessian metric on $M$ as an additional term. This is unavoidable since for a general Delzant domain of finite type there may not exist a global potential. Moreover, the Hessian of \eqref{eq:canonical:potential} may not even define a metric in $\mathring{\Delta}$. This is the case, for instance, in Examples \ref{ex:trivspherebun:stdIA:cylinder:3} and 
\ref{ex:nontrivspherebun:nonstdIA:cylinder:3}. 
Also, the Hessian hybrid $b$-metrics constructed in those examples show that $\Delta$ may carry Hessian hybrid $b$-metrics with residues as in Theorem \ref{thm:kahlmetrtoricsp} which are not obtained by K\"ahler reduction as in Proposition \ref{prop:reduced:metric}.
\end{remark}

\begin{corollary}
    Any finite-type Delzant domain of a Hessian integral affine manifold admits a Hessian hybrid $b$-metric with residues as in Theorem \ref{thm:kahlmetrtoricsp}. 
\end{corollary}

\end{document}